\documentclass[a4paper,12pt]{article}
\usepackage[english]{babel}
\usepackage{tikz}
\usetikzlibrary{matrix,arrows,decorations.markings}
\usepackage{amsmath,amsfonts,amssymb,amsthm,url,textcomp}
\usepackage[utf8]{inputenc}
\usepackage{csquotes}
\usepackage{float}
\usepackage{a4wide}
\usepackage[numbers]{natbib}
\usepackage{fancyhdr}
\usepackage{anyfontsize}
\usepackage{hyperref}
\usepackage{hhline}
\usepackage{leftidx}
\usepackage{array}
\usepackage{longtable}
\usepackage{bm}
\usepackage{enumitem}
\usepackage{makecell}

\allowdisplaybreaks

\newtheorem{theorem}{Theorem}\numberwithin{theorem}{section}
\newtheorem{lemma}[theorem]{Lemma}

\newtheorem{theoremm}{Theorem}\numberwithin{theoremm}{subsection}
\newtheorem{lemmma}[theoremm]{Lemma}

\newtheorem{propposition}[theoremm]{Proposition}

\newtheorem{quesstion}[theoremm]{Question}
\numberwithin{theoremmm}{subsubsection}

\theoremstyle{definition}

\newtheorem{deffinition}[theoremm]{Definition}
\newtheorem{nottation}[theoremm]{Notation}
\newtheorem{remmark}[theoremm]{Remark}

\newcommand{\Rad}{\operatorname{Rad}}
\newcommand{\Aut}{\operatorname{Aut}}
\newcommand{\Alt}{\operatorname{Alt}}
\newcommand{\PSL}{\operatorname{PSL}}

\newcommand{\cl}{\operatorname{cl}}
\newcommand{\lcm}{\operatorname{lcm}}

\newcommand{\ord}{\operatorname{ord}}
\newcommand{\Sym}{\operatorname{Sym}}

\renewcommand{\l}{\mathfrak{l}}
\newcommand{\C}{\mathcal{C}}

\newcommand{\Soc}{\operatorname{Soc}}

\newcommand{\id}{\operatorname{id}}

\newcommand{\Out}{\operatorname{Out}}
\newcommand{\e}{\mathrm{e}}
\newcommand{\T}{\operatorname{T}}
\newcommand{\M}{\operatorname{M}}
\newcommand{\J}{\operatorname{J}}

\newcommand{\GL}{\operatorname{GL}}

\newcommand{\PSU}{\operatorname{PSU}}
\newcommand{\PSp}{\operatorname{PSp}}

\newcommand{\im}{\operatorname{im}}

\newcommand{\G}{\mathcal{G}}

\newcommand{\D}{\operatorname{D}}

\newcommand{\Comp}{\operatorname{Comp}}

\newcommand{\Mod}[1]{\ (\textup{mod}\ #1)}

\renewcommand{\k}{\operatorname{k}}

\renewcommand{\P}{\mathfrak{P}}

\newcommand{\F}{\operatorname{F}}

\renewcommand{\F}{\mathcal{F}}

\newcommand{\IN}{\mathbb{N}}
\newcommand{\Hcal}{\mathcal{H}}
\newcommand{\Sp}{\operatorname{Sp}}

\newcommand{\IF}{\mathbb{F}}

\newcommand{\GU}{\operatorname{GU}}
\newcommand{\IZ}{\mathbb{Z}}
\renewcommand{\o}{\operatorname{o}}

\newcommand{\bcpc}{\operatorname{bcpc}}

\newcommand{\IR}{\mathbb{R}}
\newcommand{\MCS}{\operatorname{MCS}}
\newcommand{\Inndiag}{\operatorname{Inndiag}}
\newcommand{\IP}{\mathbb{P}}
\newcommand{\PGU}{\operatorname{PGU}}
\newcommand{\POmega}{\operatorname{P}\Omega}
\newcommand{\GO}{\operatorname{GO}}

\newcommand{\Outdiag}{\operatorname{Outdiag}}
\newcommand{\CT}{\operatorname{CT}}
\newcommand{\supp}{\operatorname{supp}}

\newcommand{\omicron}{\operatorname{o}}
\newcommand{\q}{\mathfrak{q}}
\newcommand{\Ord}{\operatorname{Ord}}

\newcommand{\SL}{\operatorname{SL}}

\newcommand{\U}{\mathrm{U}}

\newcommand{\Mcal}{\mathcal{M}}

\newcommand{\B}{\operatorname{B}}
\newcommand{\m}{\mathfrak{m}}
\newcommand{\dfrak}{\mathfrak{d}}
\newcommand{\Exp}{\operatorname{Exp}}
\newcommand{\Div}{\operatorname{Div}}
\newcommand{\Adm}{\operatorname{Adm}}
\newcommand{\Omicron}{\operatorname{O}}
\newcommand{\bcp}{\operatorname{bcp}}

\newcommand{\Cent}{\operatorname{C}}

\newcommand{\ct}{\operatorname{ct}}
\newcommand{\coBD}{\operatorname{coBD}}
\newcommand{\good}{\mathrm{good}}
\newcommand{\bad}{\mathrm{bad}}
\newcommand{\HS}{\operatorname{HS}}
\newcommand{\Co}{\operatorname{Co}}
\newcommand{\McL}{\operatorname{McL}}
\newcommand{\Suz}{\operatorname{Suz}}
\newcommand{\He}{\operatorname{He}}
\newcommand{\HN}{\operatorname{HN}}
\newcommand{\Th}{\operatorname{Th}}
\newcommand{\Fi}{\operatorname{Fi}}
\newcommand{\ON}{\operatorname{O'N}}
\newcommand{\Ru}{\operatorname{Ru}}
\newcommand{\Ly}{\operatorname{Ly}}
\newcommand{\g}{\mathfrak{g}}
\newcommand{\alphabv}{\vec{\bm{\alpha}}}
\newcommand{\psiv}{\vec{\psi}}
\newcommand{\tor}{\mathrm{tor}}

\newcommand{\SU}{\operatorname{SU}}
\newcommand{\PGammaO}{\operatorname{P}\Gamma\operatorname{O}}

\newcommand{\PCO}{\operatorname{PCO}}

\newcommand{\Scal}{\mathcal{S}}
\renewcommand{\O}{\operatorname{O}}
\newcommand{\ssrm}{\mathrm{ss}}

\begin{document}

\title{Automorphism orbits and element orders in finite groups: almost-solubility and the Monster}

\author{Alexander Bors, Michael Giudici and Cheryl E. Praeger\thanks{First author's address: Johann Radon Institute for Computational and Applied Mathematics (RICAM), Altenbergerstra{\ss}e 69, 4040 Linz, Austria. E-mail: \href{mailto:alexander.bors@ricam.oeaw.ac.at}{alexander.bors@ricam.oeaw.ac.at} \newline Second and third author's address: The University of Western Australia, Centre for the Mathematics of Symmetry and Computation, 35 Stirling Highway, Crawley 6009, WA, Australia. E-mail: \href{mailto:michael.giudici@uwa.edu.au}{michael.giudici@uwa.edu.au} and \href{mailto:cheryl.praeger@uwa.edu.au}{cheryl.praeger@uwa.edu.au} \newline The first author is supported by the Austrian Science Fund (FWF), project J4072-N32 \enquote{Affine maps on finite groups}. The second and third authors were supported by the Australian Research Council Discovery Project DP160102323. \newline 2010 \emph{Mathematics Subject Classification}: Primary: 20D60. Secondary: 20D05, 20D45. \newline \emph{Key words and phrases:} Finite groups, Automorphism orbits, Element orders, Simple groups, Monster simple group}}

\date{\today}

\maketitle

\abstract{For a finite group $G$, we denote by $\omega(G)$ the number of $\Aut(G)$-orbits on $G$, and by $\omicron(G)$ the number of distinct element orders in $G$. In this paper, we are primarily concerned with the two quantities $\dfrak(G):=\omega(G)-\omicron(G)$ and $\q(G):=\omega(G)/\omicron(G)$, each of which may be viewed as a measure for how far $G$ is from being an AT-group in the sense of Zhang (that is, a group with $\omega(G)=\omicron(G)$). We show that the index $|G:\Rad(G)|$ of the soluble radical $\Rad(G)$ of $G$ can be bounded from above both by a function in $\dfrak(G)$ and by a function in $\q(G)$ and $\omicron(\Rad(G))$. We also obtain a curious quantitative characterisation of the Fischer-Griess Monster group $\M$.}

\numberwithin{equation}{subsection}
\section{Introduction}\label{sec1}

The underlying theme of this paper is the study of finite groups that are \enquote{highly homogeneous}. Homogeneity conditions on structures in the general, model-theoretic sense (i.e., sets endowed with operations and relations) have been studied in various contexts for a long time. Fra\"{\i}ss{\'e} \cite{Fra53a} called a structure \emph{homogeneous} if and only if any isomorphism between finitely generated substructures extends to an automorphism of the whole structure. This notion of homogeneity has received much attention for certain classes of structures, such as graphs \cite{Gar76a}, groups \cite{CF00a} and linear spaces (in the design-theoretic sense) \cite{DD98a}.

Of course, Fra\"{\i}ss{\'e}'s notion of a \enquote{homogeneous structure} is rather strong, and weaker conditions have been studied as well. For example, rather than involving all (finitely generated) substructures of a given structure $X$ in the condition, one can restrict one's attention to the simplest subsets of $X$, such as vertices (equivalently, singleton subsets) or edges when $X$ is a graph, blocks or flags when $X$ is a design, or elements when $X$ is a group. The homogeneity conditions would then consist of transitivity assumptions on the natural action of the automorphism group $\Aut(X)$ on these simple subsets, leading to well-studied notions such as vertex-transitive graphs \cite[Definition 4.2.2, p.~85]{BW79a}, block-transitive designs \cite{CP93a,CP93b}, flag-transitive designs \cite{Hub09a}, or flag-transitive finite projective planes \cite{Tha03a}.

It should be noted that even some of these weaker (compared to Fra\"{\i}ss{\'e}'s approach) homogeneity conditions on structures $X$ are so strong that only a few \enquote{standard} examples for $X$ satisfy them. For example, the only finite $8$-arc-transitive graphs are cycles \cite{Wei81a}, and the only group $G$ such that $\Aut(G)$ acts transitively on $G$ is the trivial group. In these cases, it may be fruitful to study even weaker conditions: $s$-arc-transitive graphs with $2\leqslant s<8$ have received a lot of attention (see e.g.~the paper \cite{LSS15a} and references therein), and for finite groups $G$, the following weakenings of the condition \enquote{$\Aut(G)$ acts transitively on $G$.} have been studied:

\begin{itemize}
\item \enquote{$\Aut(G)$ admits exactly $c$ orbits on $G$.} for some given, small constant $c$. For $c=2$, it is not difficult to show that this is equivalent to $G$ being nontrivial and elementary abelian (i.e., the underlying additive group of a finite vector space). For general results pertaining to $c\in\{3,4,5,6,7\}$, see the papers \cite{BD16a,DGB17a,LM86a,MS97a,Stro02a} by various authors, and for a classification of the finite simple groups with $c\leq17$, see Kohl's result \cite[Table 3]{Koh04a}.
\item \enquote{$\Aut(G)$ admits at least one orbit of length at least $\rho|G|$ on $G$.} for some given constant $\rho\in\left(0,1\right]$. For example, it is known that if $\rho>\frac{18}{19}$, then $G$ is necessarily soluble \cite[Theorem 1.1.2(1)]{Bor19a}.
\item \enquote{For each element order $o$ in $G$, $\Aut(G)$ acts transitively on elements of order $o$ in $G$.}. In other words, $\Aut(G)$ is \enquote{as transitive as possible} given that automorphisms must preserve the orders of elements. Such finite groups $G$ are called \emph{AT-groups} and are studied extensively by Zhang in \cite{Zha92a}.
\end{itemize}

We note that there is a connection between a slightly weaker version of AT-groups, studied in \cite{LP97a}, and so-called CI-groups (groups $G$ such that any two isomorphic Cayley graphs over $G$ are \enquote{naturally isomorphic} via an automorphism of $G$), which have been studied by various authors, see the survey \cite{Li02a}.

The aim of this paper is to study notions of finite groups that are \enquote{close to being AT-groups}. That is, we will view AT-groups as extremal structures, lying at one end of a quantitative spectrum of homogeneity conditions. We will do so by comparing, for a given finite group $G$, the numbers of $\Aut(G)$-orbits on $G$ and of distinct element orders in $G$ respectively, observing that $G$ is an AT-group if and only if these two numbers are equal.

One of our main results, Theorem \ref{mainTheo1}, provides upper bounds on the smallest index of a soluble normal subgroup of a finite group $G$ that is \enquote{almost an AT-group}, with the caveat that for one of the two notions of \enquote{almost-AT-groups} with which Theorem \ref{mainTheo1} is concerned, one must also assume that the maximum number of element orders in a soluble normal subgroup of $G$ is bounded. So, roughly speaking, Theorem \ref{mainTheo1} states that \enquote{Almost-AT-groups are almost soluble.}. Before being able to prove such a result on finite groups in general, we will study the important special case of nonabelian finite simple groups, with which Theorem \ref{mainTheo2} is concerned. This involves a curious quantitative characterisation of the Fischer-Griess Monster group $\M$, see Theorem \ref{mainTheo2}(5).

\subsection{Statement of our main results}\label{subsec1P1}

Let us first introduce some notation in order to be able to state our main results in a concise way:

\begin{deffinition}\label{mainDef1}
Let $G$ be a finite group.

\begin{enumerate}
\item We denote by $\omega(G)$ the number of $\Aut(G)$-orbits on $G$.
\item We denote by $\Ord(G)$ the set of element orders in $G$, and we set $\omicron(G):=|\Ord(G)|$, the number of element orders in $G$.
\item For $o\in\Ord(G)$, we denote by $\omega_o(G)$ the number of $\Aut(G)$-orbits on the set of order $o$ elements in $G$.
\item We introduce the following parameters:
\begin{enumerate}
\item $\dfrak(G):=\omega(G)-\omicron(G)$;
\item $\q(G):=\omega(G)/\omicron(G)$;
\item $\m(G):=\max_{o\in\Ord(G)}{\omega_o(G)}$.
\end{enumerate}
\item Finally, we denote by $\Rad(G)$ the \emph{soluble radical of $G$} (the largest soluble normal subgroup of $G$).
\end{enumerate}
\end{deffinition}

We note that $\dfrak(G)\geqslant0$ and $\m(G)\geqslant\q(G)\geqslant1$ for all finite groups $G$, and that $\dfrak(G)=0$ if and only if $\m(G)=1$ if and only if $\q(G)=1$ if and only if $G$ is an AT-group. Throughout this paper, $\exp$ denotes the natural exponential function (with base the Euler constant $\e$, and not to be confused with the notation $\Exp(G)$ used for the exponent of the finite group $G$), and $\log$ denotes the natural logarithm (with base $\e$). Our first main result is the following:

\begin{theoremm}\label{mainTheo1}
There are monotonically increasing (in each component) functions $f_i:\left[0,\infty\right)^i\rightarrow\left[1,\infty\right)$ for $i=1,2$ such that for all finite groups $G$, the following hold:
\begin{enumerate}
\item $|G:\Rad(G)|\leqslant f_1(\dfrak(G))$.
\item $|G:\Rad(G)|\leqslant f_2(\q(G),\omicron(\Rad(G)))$.
\end{enumerate}
Moreover, denoting by $\M$ the Fischer-Griess Monster group and setting
\[
c:=\frac{\log\log{|\M|}}{\log\log{(413/73)}}\approx 8.76843,
\]
$f_1$ may be chosen as follows:
\begin{align*}
&f_1(x)=\\
&\exp((2^x+x)\log^c{(2^x+x+3)}\exp(\log^c{(2^x+x+3)}))\cdot \\
&(\log^{-1}{2}\cdot(2^x+x+3))^{(2^x+x)\exp(\log^c{(2^x+x+3)})}\cdot \\
&((2^x+x)!)^{\exp(\log^c{(2^x+x+3)})}.
\end{align*}
\end{theoremm}

We note that obtaining an explicit example of a function $f_2$ as in Theorem \ref{mainTheo1} appears to be a nontrivial open problem -- let us discuss why. Using known explicit bounds such as \cite[Corollary 3.1]{Mar03a} (a lower bound on the integer partition counting function) or \cite[Theorem 3]{Dus99a} (a lower bound on the $k$-th prime number), it is in principle possible to make most of our asymptotic argument for Theorem \ref{mainTheo1}(2) in Section \ref{sec4} explicit. However, one of the crucial auxiliary observations for the proof of Theorem \ref{mainTheo1}(2), stated in asymptotic form as Lemma \ref{qTildeLem2}, is the following: For nonabelian finite simple groups $S$, as certain parameters associated with $S$ tend to $\infty$ (such as the degree $m$ of $S$ when $S=\Alt(m)$ is alternating), the number of $\Aut(S)$-conjugacy classes intersecting a given coset of $S$ in $\Aut(S)$ grows faster than any power of the number of distinct orders of elements in that coset. For alternating groups, proving this boils down to showing that $\omicron(\Sym(m))$, the number of element orders in the symmetric group $\Sym(m)$, is of the form $\exp(\o(1)\sqrt{m})$ as $m\to\infty$. A stronger asymptotic statement was proved by Erd\H{o}s and Tur{\'a}n, see \cite[Theorem I]{ET68a}, but a corresponding explicit upper bound on $\omicron(\Sym(m))$ does not seem to be directly available in the literature nor easily derivable from known results. We note that the second part of Lemma \ref{qTildeLem2}, which is concerned with simple Lie type groups, is not affected, as it only relies on the asymptotics of the number of divisors function, for which explicit (upper) bounds are available, see e.g.~\cite[Th{\'e}or{\`e}me 1]{NR83a}. For another potential approach of constructing $f_2$ explicitly, see Question \ref{openQues2} and the discussion thereafter.

As the quotient $G/\Rad(G)$ is a so-called \emph{semisimple group} (a group without nontrivial soluble normal subgroups, following \cite[p.~89]{Rob96a}), and the class of finite semisimple groups is closely connected to the class of nonabelian finite simple groups (see \cite[3.3.18, p.~89]{Rob96a}), it is not surprising that an important stepping stone in proving Theorem \ref{mainTheo1} is the investigation of $\q$-values of nonabelian finite simple groups $S$. For these, we introduce the numerical parameters
\[
\epsilon_{\omega}(S):=\frac{\log\log{\omega(S)}}{\log\log{|S|}}
\]
and
\begin{equation}\label{qsPlus3Eq}
\epsilon_{\q}(S):=\frac{\log\log{(\q(S)+3)}}{\log\log{|S|}}.
\end{equation}
The addition of a positive quantity in the numerator of $\epsilon_{\q}(S)$ is necessary because there are examples where $\q(S)=1$ (e.g., $S=\Alt(5)$), and $\log\log{1}$ is not defined. On the other hand, $\log\log{(\q(S)+3)}$ is always defined and positive, and $3$ is also the largest positive constant $c$ such that $\q(S)+c\leqslant\omega(S)$ for all $S$ (this is because by Burnside's $p^aq^b$-theorem, one always has $\omega(S)\geqslant\omicron(S)\geqslant4$, and $\omega(\Alt(5))=\omicron(\Alt(5))=4$). Hence by definition, one has $0<\epsilon_{\q}(S)\leqslant\epsilon_{\omega}(S)<1$, and
\[
\exp(\log^{\epsilon_{\omega}(S)}{|S|})=\omega(S)
\]
as well as
\[
\exp(\log^{\epsilon_{\q}(S)}{|S|})=\q(S)+3.
\]
As an important stepping stone toward proving Theorem \ref{mainTheo1}, we will prove the following collection of statements on nonabelian finite simple groups, which is our second main result:

\begin{theoremm}\label{mainTheo2}
Let $S$ be a nonabelian finite simple group. Then the following hold:

\begin{enumerate}
\item $\liminf_{|S|\to\infty}{\epsilon_{\omega}(S)}=\frac{1}{2}$.
\item $\epsilon_{\omega}(S)\geqslant\frac{\log\log{4}}{\log\log{60}}\approx 0.231720$, with equality if and only if $S\cong\Alt(5)$.
\item $\frac{\log{\omicron(S)}}{\log{\omega(S)}}\to0$ as $|S|\to\infty$.
\item $\liminf_{|S|\to\infty}{\epsilon_{\q}(S)}=\frac{1}{2}$. In particular, $\q(S)\to\infty$ as $|S|\to\infty$.
\item Denoting by $\M$ the Fischer-Griess Monster group, we have that $\epsilon_{\q}(S)\geqslant\epsilon_{\q}(\M)=\frac{\log\log{(413/73)}}{\log\log{|\M|}}\approx 0.114045$, with equality if and only if $S\cong\M$.
\end{enumerate}
\end{theoremm}

The appearance of the monster group $\M$ in Theorem \ref{mainTheo2}(5) seems to suggest itself when considering just how small $\q(\M)$ is compared to $\q(S)$ for other nonabelian finite simple groups $S$ of roughly the same order as $\M$. Indeed, one has $|\M|\approx 8\cdot10^{53}$ and $\q(\M)=\frac{194}{73}\approx2.65753$, see Table \ref{sporadictable} below, whereas, for example, $|\Alt(43)|\approx 3\cdot 10^{52}$ and $\q(\Alt(43))=\frac{31659}{559}\approx 56.63506$, and $|\PSL_7(13)|\approx 3\cdot 10^{53}$ and $\q(\PSL_7(13))\geqslant\frac{2614423}{423}\approx 6180.66903$. In the case of $S=\Alt(43)$, we computed the exact values of $\omega(S)$ and $\omicron(S)$ using GAP \cite{GAP4}, whereas for $S=\PSL_7(13)$, the specified numerator and denominator are the lower bound $\lceil\k(S)/|\Out(S)|\rceil$ on $\omega(S)$ (with $\k(S)$ denoting the number of conjugacy classes of $S$, computed with GAP) and an upper bound on $\omicron(S)$, computed by a GAP implementation of an algorithm described below in Case (5) of our proof of Theorem \ref{mainTheo2}(5) in Subsection \ref{subsec2P3}, between Tables \ref{conjClassTable} and \ref{onlyOmicronTable}. Note that the nonabelian finite simple groups $S$ that are AT-groups (which are precisely the groups $\PSL_n(q)$ with $(n,q)\in\{(2,5),(2,7),(2,8),(2,9),(3,4)\}$, see \cite[Theorem 3.1]{Zha92a}) cannot achieve such a small $\epsilon_{\q}$-value because their orders are too small (the addition of $3$ in the numerator of $\epsilon_{\q}(S)$ causes the total value of the fraction to become too large).

We would like to mention that the proofs in Section \ref{sec2} involve various computational checks, which were all carried out using GAP \cite{GAP4} and are mentioned whenever they occur. The associated GAP source code, together with a documentation for it, is available from the first author's website under \url{https://alexanderbors.wordpress.com/sourcecode/orbord/}, and the documentation is also available as a preprint on arXiv under \url{https://arxiv.org/abs/1910.12570}.

\subsection{Overview of the proofs of Theorems \ref{mainTheo1} and \ref{mainTheo2}}\label{subsec1P2}

We first discuss the proof of Theorem \ref{mainTheo2}, as it will be done first (since Theorem \ref{mainTheo2} is needed in the proofs of both statements in Theorem \ref{mainTheo1}). Theorem \ref{mainTheo2} is proved in Section \ref{sec2}, and the proof is split into the three cases \enquote{$S$ is sporadic}, \enquote{$S$ is alternating} and \enquote{$S$ is of Lie type}.
\begin{itemize}
\item The sporadic finite simple groups $S$, dealt with in Subsection \ref{subsec2P1}, are irrelevant for the asymptotic statements (1), (3) and (4) of Theorem \ref{mainTheo2}, so one only needs to verify the universal bounds in statements (2) and (5) for them, which is straightforward using information from the ATLAS of Finite Group Representations \cite{ATLAS}.
\item For the alternating groups $\Alt(m)$, which are discussed in Subsection \ref{subsec2P2}, the two key ideas are, firstly, that $\omega(\Alt(m))$ and $\omicron(\Alt(m))$ are \enquote{almost equal to} the corresponding parameters of the symmetric group $\Sym(m)$, and, secondly, that both $\omega(\Sym(m))$ and $\omicron(\Sym(m))$ can be expressed in terms of certain integer partition counting functions. One can therefore apply number-theoretic results, dating back to Hardy and Ramanujan's 1918 paper \cite{HR18a} but also involving comparatively recent results such as Mar{\'o}ti's \cite[Corollary 3.1]{Mar03a}, which together yield information on the asymptotic growth rate of and explicit bounds on those partition counting functions.
\item Finally, Subsection \ref{subsec2P3} is concerned with the finite simple groups $S$ of Lie type. A lower bound on $\omega(S)$ can be produced, using \cite[Corollary 1, p.~506]{Ern61a}, from a well-known (see e.g.~\cite[Theorem 1.1(1)]{FG12a}) lower bound on the number of conjugacy classes of $S$, using that $|\Aut(S):S|=|\Out(S)|$ is \enquote{small} (see e.g.~\cite{Koh03a}). On the other hand, an upper bound on $\omicron(S)$ can be obtained as follows: Firstly, one notes that $\omicron(S)$ is \enquote{almost} (up to a small factor) the same as the number of \emph{semisimple} element orders (i.e., element orders not divisible by the defining characteristic of $S$). Secondly, every semisimple element of $S$ is contained in a maximal torus of $S$, so the number of semisimple element orders in $S$ can be bounded from above nicely using classical results on the conjugacy classes of maximal tori of $S$ and their orders, see \cite[Section 3]{Car81a}, \cite[Lemma 3.3]{Har92a} and \cite[Theorem 1.2(b), p.~1.8]{Gag73a} (cf.~also \cite[Proposition 25.1, p.~219]{MT11a}). The asserted asymptotic results follow swiftly from this, and for the universal bounds, one needs to determine which \enquote{small} cases are not clear by the asymptotic arguments and repeat essentially the same arguments in a more careful manner.
\end{itemize}

Next, we talk about the proof of Theorem \ref{mainTheo1}(1), which is the subject of Section~\ref{sec3}. Assume that $G$ is a finite group; our goal is to bound $|G:\Rad(G)|$ in terms of $\dfrak(G)$. We first generalise a result of Zhang \cite[Lemma 1.1]{Zha92a} to show that if $N$ is a characteristic subgroup of a finite group $G$, then $\m(G/N)\leqslant 2^{\dfrak(G)}+\dfrak(G)$, see Lemma \ref{charQuotLem}(2). Applied with $N:=\Rad(G)$, we find that $\m(G/\Rad(G))$ is bounded in terms of $\dfrak(G)$, which allows us to restrict our attention to finite semisimple groups $H$, and show that $|H|$ can be bounded from above in terms of $\m(H)$. But $\m(\Soc(H))\leqslant\m(H)$, see Lemma \ref{charSubLem}, and since $H$ embeds into $\Aut(\Soc(H))$ via its conjugation action on $H$, it suffices to show that $|\Soc(H)|$ can be bounded from above in terms of $\m(\Soc(H))$. Since $\Soc(H)$ is a direct product of nonabelian finite simple groups, that last statement easily reduces to Theorem \ref{mainTheo2}, see Lemma \ref{socLem}.

Finally, we give an overview of the proof of Theorem \ref{mainTheo1}(2), with which Section \ref{sec4} is concerned.
\begin{itemize}
\item In Subsection \ref{subsec4P1}, a crucial starting observation is made, namely that it suffices to show that for finite semisimple groups $H$, the order of $H$ can be bounded from above in terms of $\q(H)$ (compare this with bounding $|H|$ in terms of $\m(H)$, which is needed in the proof of Theorem \ref{mainTheo1} and is much easier). The remainder of Section \ref{sec4} is concerned with proving this result for finite semisimple groups $H$.
\item In order to bound $\q(H)$ suitably from below, we partition $H$ into certain unions of cosets of $\Soc(H)$ and study the \enquote{$\q$-values of these subsets}; note that at the moment, $\q(G)$ is only defined when $G$ is a finite group, not a subset $M$ thereof, but the definition will be extended accordingly in Notation \ref{subsetNot}(3), writing $\q_G(M)$. Subsection \ref{subsec4P2} provides simple, but important abstract tools, in the form of Lemmas \ref{partitionLem} and \ref{mixedPartitionLem}, to make this idea of working with partitions of $H$ feasible.
\item In the first instance, the \enquote{partition approach} described in the previous bullet point allows one to show that $\q(H)$ is large when $\Soc(H)$ contains a (nonabelian) composition factor $S$ for which a certain other parameter, $\tilde{\q}(S)$ (see Notation \ref{tildeNot}(1)), is large; in other words, it gives a partial reduction, carried out in Subsection \ref{subsec4P4}, to nonabelian finite simple groups, and corresponding auxiliary results on nonabelian finite simple groups $S$ are proved for later use in Subsection \ref{subsec4P3}.
\item In the brief Subsection \ref{subsec4P5}, we change our perspective: Our goal can be equivalently restated as showing that for each constant $c\geqslant1$, the class $\Hcal^{(c)}$ (see Notation \ref{classNot}(1)), of finite semisimple groups $H$ with $\q(H)\leqslant c$, is finite. The remainder of the proof is concerned with giving more and more restrictions on the members of an arbitrary, but fixed class $\Hcal^{(c)}$ until it becomes clear that only finitely many finite semisimple groups can satisfy all those restrictions. A first result in this direction is Lemma \ref{hcalLem}, which shows, as an application of the theory developed so far, that $\Hcal^{(c)}$ is contained in a certain other class of finite semisimple groups, $\Hcal_{\hat{m},\hat{d},\hat{p}}$ (see Notation \ref{classNot}(2,3) for the precise definition), whose members satisfy numerical restrictions with regard to the composition factors of their socles.
\item Subsection \ref{subsec4P6} contains a few elementary number-theoretic results, which serve as auxiliary results in the subsequent subsection.
\item Subsection \ref{subsec4P7} consists of some technical results holding for all finite semisimple groups $H$ belonging to a fixed class $\Hcal_{\hat{m},\hat{d},\hat{p}}$ as introduced in Subsection \ref{subsec4P5}. First, it is observed that only a very specific kind of socle coset in a finite semisimple group $H$, called an \emph{$\hat{h}$-small socle coset} (see Notation \ref{hNot} for the details) is \enquote{problematic} as far as the partition idea from Subsection \ref{subsec4P2} is concerned, see Lemma \ref{hHatLargeLem}. Next, Lemma \ref{constantLem} is proved, which basically states that $\hat{h}$-small socle cosets (in finite semisimple groups lying in a class $\Hcal_{\hat{m},\hat{d},\hat{p}}$) contain only few distinct element orders (which is useful in view of Lemma \ref{mixedPartitionLem}). Lemma \ref{qGammaLem} narrows the set of \enquote{problematic} socle cosets $C$ further, based on the common permutation action of the members of $C$ on the coordinates of $\Soc(H)$. Finally, Lemma \ref{aLem}, an application of Lemmas \ref{constantLem} and \ref{qGammaLem} as well as the results of Subsection \ref{subsec4P2}, exhibits a partition of a \enquote{large part of $H$} in which every partition member has large $\q_H$-value.
\item The last few remaining tools for proving Theorem \ref{mainTheo1}(2) are provided in Subsection \ref{subsec4P8}. Firstly, a third kind of class of finite semisimple groups, $\Hcal_{\hat{m},\hat{d},\hat{p},\hat{r},f}$ (contained in $\Hcal_{\hat{m},\hat{d},\hat{p}}$), is introduced in Notation \ref{classNot2}, and it is shown that each class $\Hcal^{(c)}$ is contained in such a class, see Lemma \ref{hcalLem2}. But also, each intersection $\Hcal^{(c)}\cap\Hcal_{\hat{m},\hat{d},\hat{p},\hat{r},f}$ is finite, see Lemma \ref{coupLem}. Combining these two facts, one gets that indeed, $\Hcal^{(c)}$ is always finite, as required.
\item To round Section \ref{sec4} off, Subsection \ref{subsec4P9} gives the actual proof of Theorem \ref{mainTheo1}(2) in a concise form, referring to results from the other subsections as needed.
\end{itemize}

\subsection{Some related open questions}\label{subsec1P3}

In this subsection, we discuss three open questions related to the results and proofs of this paper. The following is natural to ask when comparing statements (1) and (2) in Theorem \ref{mainTheo1}:

\begin{quesstion}\label{openQues1}
Does there exist a monotonically increasing function $f:\left[1,\infty\right)\rightarrow\left[1,\infty\right)$ such that $|G:\Rad(G)|\leqslant f(\q(G))$ for all finite groups $G$?
\end{quesstion}

As will become clear later from Remark \ref{semisimpleRem}, Theorem \ref{mainTheo1}(2) is essentially just a statement about finite semisimple groups, which is a very helpful observation, since the structure of finite semisimple groups is well understood. However, in order to answer Question \ref{openQues1} in the affirmative, one would need to improve on the (probably very pessimistic) bounds
\[
\omega(G)\geqslant\omega(G/\Rad(G))
\]
and
\[
\omicron(G)\leqslant\omicron(\Rad(G))\cdot\omicron(G/\Rad(G))
\]
from the discussion in Remark \ref{semisimpleRem}, and it seems inevitable that in order to do so, one needs to gain a better understanding of the \enquote{interplay} between $\Rad(G)$ and $G/\Rad(G)$, i.e., of the theory of extensions of finite semisimple groups by finite soluble groups. In the authors' opinion, this is a probably very challenging, but also interesting research problem, and even partial results putting structural restrictions on $\Rad(G)$ or $G/\Rad(G)$ (for example, assuming that $\Rad(G)$ is cyclic) would be of interest.

The second open question concerns the following numerical parameter associated with each finite group:

\begin{nottation}\label{lNot}
Let $G$ be a nontrivial finite group. We set
\[
\l(G):=\frac{\log{\omega(G)}}{\log{\omicron(G)}}.
\]
Moreover, we define the $\l$-value of the trivial group to be $1$.
\end{nottation}

Note that by definition, the parameters $\dfrak(G)$, $\q(G)$ and $\l(G)$ satisfy the following equations for each finite group $G$, which could be used as implicit definitions for them (except for $\l(G)$ when $G$ is trivial):
\[
\omega(G)=\omicron(G)+\dfrak(G),
\]
\begin{equation}\label{qImplicitEq}
\omega(G)=\omicron(G)\cdot\q(G)
\end{equation}
and
\[
\omega(G)=\omicron(G)^{\l(G)}.
\]
Note also that Theorem \ref{mainTheo2}(3) just says that for nonabelian finite simple groups $S$, $\l(S)\to\infty$ as $|S|\to\infty$. However, in contrast to $\dfrak(G)$, we have the following:

\begin{propposition}\label{lProp}
There is no monotonically increasing function $f:\left[1,\infty\right)\rightarrow\left[1,\infty\right)$ such that $|G:\Rad(G)|\leqslant f(\l(G))$ for all finite groups $G$.
\end{propposition}

\begin{proof}
By contradiction: Fix such a function $f$. Let $S$ be a nonabelian finite simple group with $|S|>f(2)$. Set $k:=|S|$, and fix $k$ pairwise distinct primes $p_1,\ldots,p_k$ none of which divides $|S|$. Set $G:=S\times(\IZ/(p_1\cdots p_k)\IZ)$. Then the second factor, $\IZ/(p_1\cdots p_k)\IZ$, is the soluble radical of $G$, and the first factor, $S$, is the derived subgroup of $G$. So both factors are characteristic in $G$, and thus
\[
\Aut(G)=\Aut(S)\times\Aut(\IZ/(p_1\cdots p_k)\IZ)
\]
and
\begin{equation}\label{lOmegaEq}
\omega(G)=\omega(S)\cdot\omega(\IZ/(p_1\cdots p_k)\IZ)=\omega(S)\cdot 2^k.
\end{equation}
Moreover, we have a surjection
\[
\Ord(S)\times\Ord(\IZ/(p_1\cdots p_k)\IZ)\rightarrow\Ord(G), (o_1,o_2)\mapsto\lcm(o_1,o_2),
\]
and by the choice of $p_1,\ldots,p_k$, this surjection is a bijection. Hence
\begin{equation}\label{lOmicronEq}
\omicron(G)=\omicron(S)\cdot\omicron(\IZ/(p_1\cdots p_k)\IZ)=\omicron(S)\cdot 2^k.
\end{equation}
Recall that by definition, $k=|S|\geqslant60$, which entails that
\[
\omega(S)\leqslant k\leqslant 2^{k/2},
\]
and thus
\begin{equation}\label{lLastEq}
\frac{\log{\omega(S)}}{k\log{2}}\leqslant\frac{1}{2}.
\end{equation}
Combining Formulas (\ref{lOmegaEq}), (\ref{lOmicronEq}) and (\ref{lLastEq}), we conclude that
\[
\l(G)=\frac{\log{\omega(G)}}{\log{\omicron(G)}}=\frac{\log{\omega(S)}+k\log{2}}{\log{\omicron(S)}+k\log{2}}=\frac{\frac{\log{\omega(S)}}{k\log{2}}+1}{\frac{\log{\omicron(S)}}{k\log{2}}+1}\leqslant\frac{\log{\omega(S)}}{k\log{2}}+1\leqslant\frac{1}{2}+1<2,
\]
and thus
\[
f(2)\geqslant f(\l(G))\geqslant |G:\Rad(G)|=|S|>f(2),
\]
a contradiction.
\end{proof}

Observing that for the groups $G$ used as counter-examples in the proof of Proposition \ref{lProp}, $\omicron(\Rad(G))$ depends on $|G:\Rad(G)|$, it still seems reasonable to ask the following:

\begin{quesstion}\label{openQues2}
Is there a function $f:\left[1,\infty\right)^2\rightarrow\left[1,\infty\right)$ that is monotonically increasing in both components and such that $|G:\Rad(G)|\leqslant f(\l(G),\omicron(\Rad(G)))$ for all nontrivial finite groups $G$?
\end{quesstion}

Note that an affirmative answer to Question \ref{openQues2} implies Theorem \ref{mainTheo1}(2). Indeed, for each nontrivial finite group $G$, by taking logarithms on both sides of Formula (\ref{qImplicitEq}) and then dividing both sides by $\log{\omicron(G)}$, we find that
\[
\l(G)=\frac{\log{\omega(G)}}{\log{\omicron(G)}}=1+\frac{\log{\q(G)}}{\log{\omicron(G)}}\leqslant 1+\frac{\log{\q(G)}}{\log{2}}.
\]
Hence for every finite group $G$ (the trivial group is just checked separately),
\[
\l(G)\leqslant 1+\frac{\log{\q(G)}}{\log{2}},
\]
and so if $f$ is as in Question \ref{openQues2}, then
\[
f_2(x,y):=\begin{cases}1, & \text{if }\min\{x,y\}<1, \\ f(1+\frac{\log{x}}{\log{2}},y), & \text{if }\min\{x,y\}\geqslant1\end{cases}
\]
is a suitable choice for the function $f_2$ in Theorem \ref{mainTheo1}(2).

In Subsection \ref{subsec2P2}, we show that for all $n<25000$, $\omicron(\Sym(n))$ is at most $\exp(\sqrt{n})$. This leads to the following question, an affirmative answer to which would give a simple universally valid upper bound on $\omicron(\Sym(n))$:

\begin{quesstion}\label{openQues3}
Is it true that $\omicron(\Sym(n))\leqslant\exp(\sqrt{n})$ for all positive integers $n$?
\end{quesstion}

We note that Erd\H{o}s and Tur{\'a}n's result \cite[Theorem I]{ET68a} on the asymptotics of $\omicron(\Sym(n))$, see also Formula (\ref{orderAsymptoticsEq}), implies that the inequality in Question \ref{openQues3} does hold for all large enough $n$. In view of this, a possible approach to answering Question \ref{openQues3} would be to
\begin{enumerate}
\item work through Erd\H{o}s and Tur{\'a}n's proof and check whether each of the asymptotic number-theoretic results which they use has a counterpart with explicit bounds, so that $\omicron(\Sym(n))\leqslant\exp(\sqrt{n})$ could at least be proved for all $n\geqslant N_0$ for an \emph{explicit} positive integer $N_0$, and
\item check with a computer whether $\omicron(\Sym(n))\leqslant\exp(\sqrt{n})$ for $n<N_0$.
\end{enumerate}

\section{Notation}\label{secNot}

In this section, we fix some basic notation that will be used throughout this paper. The symbol $\IN$ will always denote the set of natural numbers including $0$, and $\IN^+$ denotes the set of positive integers. For a finite set $\Omega$, we denote by $\Sym(\Omega)$ the symmetric group on $\Omega$, and for $n\in\IN^+$, $\Sym(n)$ and $\Alt(n)$ denote the symmetric and alternating group on $\{1,\ldots,n\}$ respectively. For a prime power $q$, the finite field with $q$ elements will be denoted by $\IF_q$, and the algebraic closure of a field $K$ is denoted by $\overline{K}$. For a finite group $G$, we denote by $\k(G)$ the number of conjugacy classes of $G$ and by $\Exp(G)$ the exponent of $G$ (i.e., the least common multiple of the element orders in $G$).

If $n$ is a positive integer and $p$ is a prime, we denote by $\nu_p(n)$ the \emph{$p$-adic valuation of $n$}, i.e., the largest nonnegative integer $k$ such that $p^k\mid n$. We will also write $\Div(n)$ for the set of (positive) divisors of $n$, and $\tau(n)$ for the number of (positive) divisors of $n$ (this needs to be distinguished from the variable $\tau$ used to denote so-called \enquote{$S$-types} in Section \ref{sec4}, see Definition \ref{sTypeDef}(1)). For a positive integer $n$ and a power $\pi$ of some prime $\ell$, we write $\pi\mid\mid n$, read \enquote{$\pi$ sharply divides $n$}, when $\pi$ divides $n$, but $\pi\cdot\ell$ does not divide $n$.

For functions $f,g$ mapping from some unbounded set $M\subseteq\left[0,\infty\right)$ to $\left[0,\infty\right)$, we will use the Landau notation $f=\O(g)$, meaning that there is a constant $c>0$ such that $f(x)\leqslant c\cdot g(x)$ for all $x\in M$.

In what follows, we set up some notation regarding the finite simple groups of Lie type. For a prime $p$ and a Lie symbol $X_d\in\{A_d,B_d,C_d,D_d\mid d\geqslant1\}\cup\{E_6,E_7,E_8,F_4,G_2\}$, we denote by $X_d(\overline{\IF_p})$ the associated simple Chevalley group (i.e., simple linear algebraic group of adjoint type) over $\overline{\IF_p}$. If $\sigma$ is a Lang-Steinberg endomorphism (\enquote{Frobenius map} in the terminology of \cite[p.~104]{Har92a}) on $X_d(\overline{\IF_p})$, then $X_d(\overline{\IF_p})_{\sigma}$ denotes the (finite) fixed point subgroup of $\sigma$ in $X_d(\overline{\IF_p})$. For a finite group $G$ and a prime $p$, $O^{p'}(G)$ is the subgroup of $G$ generated by the $p$-elements (elements of order a power of $p$) of $G$. The notation we use for finite simple groups of Lie type follows the approach taken in \cite[Section 3, pp.~104f.]{Har92a}, so that $\leftidx{^t}{X_d(p^{f\cdot t})}$, where the pre-superscripted $t$ is usually omitted if it is $1$, denotes $O^{p'}(X_d(\overline{\IF_p})_{\sigma})$, where $\sigma$ is a Lang-Steinberg endomorphism of $X_d(\overline{\IF_p})$ satisfying the following conditions involving the parameters $t=t(\sigma)$ and $f=f(\sigma)$: Let $B$ be any $\sigma$-invariant Borel subgroup of $X_d(\overline{\IF_p})$, and let $T$ be any $\sigma$-invariant maximal torus of $X_d(\overline{\IF_p})$ contained in $B$. Then $t$ is the unique smallest positive integer (independent of the choice of $B$ and $T$) such that the $t$-th power of the map $\sigma^{\ast}$ on the character group $X(T)$ induced by $\sigma$ is a positive integral multiple of $\id_{X(T)}$, and $f\in\IN^+/2=\{\frac{1}{2},1,\frac{3}{2},\ldots\}$ is such that $\sigma^{\ast}=p^f\sigma_0$ with $\sigma_0^t=\id_{X(T)}$; $f$ also does not depend on the choice of $B$ and $T$. So $p^f=q(\sigma)$ in the notation of \cite{Har92a}, which is also a notation we will be using, and $f\in\IN^+$ unless $\leftidx{^t}{X_d(p^{f\cdot t})}$ is one of the Suzuki or Ree groups, in which case $f$ is half of an odd positive integer. For us, a \emph{finite simple group of Lie type} is by definition any group of the form $\leftidx{^t}X_d(p^{ft})$, even if it is not a simple group (such as $A_1(2)$). We say that $\leftidx{^t}X_d(p^{ft})$ is of \emph{untwisted Lie rank} $d$; with a few small exceptions (such as $A_1(7)\cong A_2(2)$), each finite simple group of Lie type has precisely one untwisted Lie rank. In the context of finite simple groups of Lie type, the terms \enquote{graph automorphism}, \enquote{field automorphism} and \enquote{graph-field automorphism} (the last meaning \enquote{product of a field and a graph automorphism}) and the associated notations $\Phi_S$ and $\Gamma_S$ are used as explained in \cite[p.~105]{Har92a}. Moreover, as in \cite{GLS98a}, $\Inndiag(S)$ denotes the inner diagonal automorphism group of $S$ (so $\Inndiag(\leftidx{^t}X_d(p^{f\cdot t}))\cong X_d(\overline{\IF_p})_{\sigma}$ in the above notation), and $\Outdiag(S)$, the \emph{outer diagonal automorphism group of $S$}, is the image of $\Inndiag(S)$ under the canonical projection $\Aut(S)\rightarrow\Out(S)$. As in \cite[Theorem 2.5.12(b), p.~58]{GLS98a}, we also view $\Phi_S$ and $\Gamma_S$ as subsets of $\Out(S)$, depending on the context. When $\alpha\in\Aut(S)$ (resp.~$\alpha\in\Out(S)$), then as stated in \cite[p.~105]{Har92a}, $\alpha$ admits a unique factorisation into an element of $\Inndiag(S)$ (resp.~$\Outdiag(S)$), an element of $\Phi_S$ and an element of $\Gamma_S$, and we call these the \emph{inner diagonal component} (resp.~\emph{outer diagonal component}), \emph{field component} and \emph{graph component of $\alpha$}, respectively. The product of the field and graph component of $\alpha$ is also called the \emph{graph-field component of $\alpha$}.

\section{Proof of Theorem \ref{mainTheo2}}\label{sec2}

\subsection{Sporadic groups}\label{subsec2P1}

In Table \ref{sporadictable}, we give an overview of the values of $\omega(S)$ and $\omicron(S)$ as well as of $\epsilon_{\omega}(S)$ and $\epsilon_{\q}(S)$ for the sporadic nonabelian finite simple groups $S$, thus proving Theorem \ref{mainTheo2} for them. This was mostly just read off from the ATLAS of Finite Group Representations \cite{ATLAS}; only for $S=\T$, the Tits group, $\omega(S)$ could not be determined directly from the ATLAS, but was computed by comparing centraliser orders, in $\T$ and $\Aut(\T)$ respectively, of conjugacy class representatives of $\T$.

\begin{center}
\begin{longtable}[H]{|c|c|c|c|c|}
\caption{Overview of the sporadic groups and the Tits group}\label{sporadictable}\\
\hline
$S$ & $\omicron(S)$ & $\omega(S)$ & $\epsilon_{\omega}(S)\approx$ & $\epsilon_{\q}(S)\approx$ \\ \hline
$\M_{11}$ & $8$ & $10$ & $0.380024$ & $0.168333$ \\ \hline
$\M_{12}$ & $9$ & $12$ & $0.373194$ & $0.156934$ \\ \hline
$\M_{22}$ & $9$ & $11$ & $0.340952$ & $0.142251$ \\ \hline
$\M_{23}$ & $12$ & $17$ & $0.374450$ & $0.142269$ \\ \hline
$\M_{24}$ & $15$ & $26$ & $0.398909$ & $0.149020$ \\ \hline
$\HS$ & $13$ & $21$ & $0.388150$ & $0.148125$ \\ \hline
$\J_2$ & $11$ & $16$ & $0.393933$ & $0.155060$ \\ \hline
$\Co_1$ & $32$ & $101$ & $0.406933$ & $0.158971$ \\ \hline
$\Co_2$ & $21$ & $60$ & $0.409051$ & $0.165308$ \\ \hline
$\Co_3$ & $21$ & $42$ & $0.400357$ & $0.144505$ \\ \hline
$\McL$ & $15$ & $19$ & $0.356873$ & $0.122978$ \\ \hline
$\Suz$ & $19$ & $37$ & $0.390324$ & $0.142663$ \\ \hline
$\He$ & $15$ & $26$ & $0.381463$ & $0.142502$ \\ \hline
$\HN$ & $22$ & $44$ & $0.379828$ & $0.135821$ \\ \hline
$\Th$ & $25$ & $48$ & $0.369346$ & $0.127106$ \\ \hline
$\Fi_{22}$ & $22$ & $59$ & $0.406280$ & $0.159655$ \\ \hline
$\Fi_{23}$ & $32$ & $98$ & $0.405230$ & $0.156730$ \\ \hline
$\Fi_{24}'$ & $35$ & $97$ & $0.378603$ & $0.139755$ \\ \hline
$\B$ & $49$ & $184$ & $0.379739$ & $0.148826$ \\ \hline
$\M$ & $73$ & $194$ & $0.344642$ & $0.114045$ \\ \hline
$\J_1$ & $10$ & $15$ & $0.399899$ & $0.163849$ \\ \hline
$\ON$ & $18$ & $25$ & $0.355275$ & $0.118954$ \\ \hline
$\J_3$ & $13$ & $17$ & $0.362182$ & $0.131708$ \\ \hline
$\Ru$ & $18$ & $36$ & $0.393116$ & $0.146573$ \\ \hline
$\J_4$ & $31$ & $62$ & $0.370447$ & $0.124360$ \\ \hline
$\Ly$ & $28$ & $53$ & $0.377735$ & $0.126657$ \\ \hline
$\T$ & $11$ & $17$ & $0.369863$ & $0.147333$ \\ \hline
\end{longtable}
\end{center}

\subsection{Alternating groups}\label{subsec2P2}

We will prove the following five statements:
\begin{enumerate}[label=(\Roman*)]
\item $\frac{\log{\omicron(\Alt(n))}}{\log{\omega(\Alt(n))}}\to0$ as $n\to\infty$.
\item For all $n\geqslant5$, $\epsilon_{\omega}(\Alt(n))\geqslant\frac{\log\log{4}}{\log\log{60}}\approx 0.231720$, with equality if and only if $n=5$.
\item $\epsilon_{\omega}(\Alt(n))\to\frac{1}{2}$ as $n\to\infty$.
\item For all $n\geqslant5$, $\epsilon_{\q}(\Alt(n))\geqslant\frac{\log\log{(77/16)}}{\log\log{19958400}}\approx 0.160121$, with equality if and only if $n=11$.
\item $\epsilon_{\q}(\Alt(n))\to\frac{1}{2}$ as $n\to\infty$.
\end{enumerate}

\begin{proof}[Proof of statement (I)]
This can be obtained by combining the following facts (recall that $\k(G)$ denotes the number of conjugacy classes of the finite group $G$):
\begin{enumerate}
\item For $n\geqslant7$, one has $\omega(\Alt(n))\geqslant\frac{1}{2}\k(\Alt(n))\geqslant\frac{1}{4}\k(\Sym(n))=\frac{1}{4}p(n)$, where the last inequality uses \cite[Formula (1.6), p.~90]{DET69a} and $p(n)$ denotes the number of (unordered) integer partitions of $n$.
\item The partition number $p(n)$ has the following asymptotics (see \cite{HR18a}):
\begin{equation}\label{partitionAsymptoticsEq}
p(n)\sim\frac{1}{4\sqrt{3}n}\exp(\frac{2\pi}{\sqrt{6}}\sqrt{n})
\end{equation}
\item Clearly, $\omicron(\Alt(n))\leqslant\omicron(\Sym(n))$, and by \cite[Theorem I]{ET68a}, the number of element orders in $\Sym(n)$ has the following asymptotics:
\begin{equation}\label{orderAsymptoticsEq}
\omicron(\Sym(n))=\exp(\frac{2\pi}{\sqrt{6}}\sqrt{\frac{n}{\log{n}}}+\O(\frac{\sqrt{n}\log\log{n}}{\log{n}})).
\end{equation}
\end{enumerate}
\end{proof}

\begin{proof}[Proof of statement (II)]
This can be checked directly for $n=5,6,7,8$. For $n\geqslant9$, we use that
\begin{equation}\label{omega56Eq}
\omega(\Alt(n))\geqslant\frac{1}{4}p(n)>\frac{1}{56}\exp(2\sqrt{n}),
\end{equation}
where the first inequality was already discussed in the proof of statement (I), and the second is a result of Mar{\'o}ti, see \cite[Corollary 3.1]{Mar03a}. Hence it suffices to check the inequality
\[
\frac{\log\log{(\exp(2\sqrt{n})/56)}}{\log\log{(n!/2)}}>\frac{\log\log{4}}{\log\log{60}}
\]
for all $n\geqslant9$. For $n=9,\ldots,54$, one just verifies this with a computer, and for $n\geqslant55$, where one has
\begin{itemize}
\item $\e\sqrt{n}\leqslant\frac{n}{\e}$,
\item $\sqrt{n}\geqslant\log{56}$, and
\item $\log{(n+1)}+\log\log{n}\leqslant 2\log{n}$,
\end{itemize}
one proceeds as follows:
\begin{itemize}
\item Firstly, $\log\log{(\exp(2\sqrt{n})/56)}=\log{(2\sqrt{n}-\log{56})}\geqslant\log{\sqrt{n}}=\frac{1}{2}\log{n}$.
\item Secondly, using the simple upper bound $n!\leqslant\e\sqrt{n}(\frac{n}{\e})^n$, which follows from Robbins' sharper bound \cite{Rob55a}, we find that
\begin{align*}
&\log\log{(n!/2)}\leqslant\log\log{n!}\leqslant\log\log{(\e\sqrt{n}(\frac{n}{\e})^n)}\leqslant\log\log{(\frac{n}{\e})^{n+1}}\leqslant \\
&\log{((n+1)(\log{n}-1))}\leqslant\log{((n+1)\log{n})}=\log{(n+1)}+\log\log{n}\leqslant2\log{n}.
\end{align*}
\end{itemize}
Combining these inequalities, one gets that
\[
\frac{\log\log{(\exp(2\sqrt{n})/56)}}{\log\log{(n!/2)}}\geqslant\frac{\frac{1}{2}\log{n}}{2\log{n}}=\frac{1}{4}>\frac{\log\log{4}}{\log\log{60}}.\qedhere
\]
\end{proof}

\begin{proof}[Proof of statements (III) and (V)]
Observe that for each constant $c>0$, using Stirling's approximation, we have that as $n\to\infty$,
\[
\frac{\log\log{\exp(c\sqrt{n})}}{\log\log{(\frac{1}{2}n!)}}\to\frac{1}{2}.
\]
It is thus sufficient to show that there are positive constants $c<c'$ such that for all large enough $n$,
\begin{equation}\label{constantsEq}
\exp(c\sqrt{n})\leqslant\q(\Alt(n))\leqslant\omega(\Alt(n))\leqslant\exp(c'\sqrt{n}).
\end{equation}
Also, by Formula (\ref{partitionAsymptoticsEq}), for large enough $n$,
\[
\omega(\Alt(n))\leqslant\omega(\Sym(n))\leqslant\k(\Sym(n))=p(n)\leqslant\exp(\frac{2\pi}{\sqrt{6}}\sqrt{n}),
\]
and so $c':=\frac{2\pi}{\sqrt{6}}$ is a possible choice in Formula (\ref{constantsEq}). Moreover, the bounds discussed in the proof of statement (I) yield that any choice of $c<\frac{2\pi}{\sqrt{6}}$ works in Formula (\ref{constantsEq}).
\end{proof}

\begin{proof}[Proof of statement (IV)]
For $n=5,\ldots,37$, one verifies this directly with the aid of GAP \cite{GAP4}. For $n=38,\ldots,24999$, we give an argument that relies partially on computer calculations. Indeed, one can compute $\omicron(\Sym(n))$ exactly for each such $n$ using a certain recursion which we will now describe. Observe that
\[
\omicron(\Sym(n))=\sum_{k=0}^n{r(k)}
\]
where $r(k)$ denotes the number of integer partitions of $k$ into pairwise coprime prime powers each greater than $1$ (note that in view of the empty partition, $r(0)=1$). Moreover, recalling the notation $\IP$ for the set of primes, we have
\[
r(k)=\sum_{p\leqslant k,p\in\IP}{r_p(k)}
\]
where $r_p(k)$ denotes the number of integer partitions of $k$ into pairwise coprime prime powers greater than $1$ such that the smallest prime base which occurs is $p$. The numbers $r_p(k)$ satisfy the recursion
\[
r_p(k)=\sum_{e=1}^{\lfloor\log_p(k)\rfloor}{\begin{cases}\sum_{p<\ell\leqslant k,\ell\in\IP}{r_{\ell}(k-p^e)}, & \text{if }k>p^e, \\ 1, & \text{if }k=p^e.\end{cases}}
\]
This allows one to check that $\omicron(\Sym(n))\leqslant\exp(\sqrt{n})$ for all $n<25000$. Using this and Formula (\ref{omega56Eq}), it follows that for all $n\in\{38,\ldots,24999\}$,
\[
\q(\Alt(n))=\frac{\omega(\Alt(n))}{\omicron(\Alt(n))}\geqslant\frac{1}{4}\frac{p(n)}{\omicron(\Sym(n))}\geqslant\frac{1}{56}\frac{\exp(2\sqrt{n})}{\exp(\sqrt{n})}=\frac{1}{56}\exp(\sqrt{n}),
\]
and thus
\begin{align*}
&\epsilon_{\q}(\Alt(n))\geqslant\frac{\log\log{(\exp(\sqrt{n})/56)}}{\log\log{(n!/2)}}\geqslant\frac{\log{(\sqrt{n}-\log{56})}}{\log\log{(\frac{\e}{2}\sqrt{n}(\frac{n}{\e})^n)}}= \\
&\frac{\log{(\sqrt{n}-\log{56})}}{\log{(1-\log{2}+\frac{1}{2}\log{n}+n(\log{n}-1))}},
\end{align*}
and one can verify with a computer that for $n=38,\ldots,24999$, the last expression is always at least $0.164>\epsilon_{\q}(\Alt(11))\approx 0.160121$.

It remains to deal with the case $n\geqslant25000$. For this, we will use a different, worse upper bound on $\omicron(\Sym(n))$ than $\exp(\sqrt{n})$, obtained as follows: Note that $\omicron(\Sym(n))\leqslant\sum_{k=0}^n{s(k)}$, where $s(k)$ denotes the number of integer partitions of $k$ into pairwise distinct parts. By \cite[Subsection 5.2]{Wil08a}, we have that $s(k)=\sum_r{p(\frac{k-r(r+1)/2}{2})}$, where $r$ ranges over the nonnegative integers such that $k-\frac{r(r+1)}{2}$ is an even nonnegative integer. There are at most $\sqrt{2k}$ such $r$, and each corresponding summand is of the form $p(j)$ where $0\leqslant j\leqslant \frac{k}{2}$. Using this and Erd\H{o}s's explicit upper bound $p(j)\leqslant\e^{\frac{2\pi}{\sqrt{6}}\sqrt{j}}$, see \cite[pp.~437f.]{Erd42a}, it follows that $s(k)\leqslant\sqrt{2k}\,\e^{\frac{\pi}{\sqrt{3}}\sqrt{k}}$, and thus $\omicron(\Sym(n))\leqslant(n+1)\sqrt{2n}\,\e^{\frac{\pi}{\sqrt{3}}\sqrt{n}}$.

Hence (and in view of Formula (\ref{omega56Eq})) we get that
\begin{align*}
&\q(\Alt(n))\geqslant\frac{1}{4}\frac{p(n)}{\omicron(\Sym(n))}\geqslant\frac{1}{56}\frac{\exp(2\sqrt{n})}{(n+1)\sqrt{2n}\cdot\exp((\pi/\sqrt{3})\sqrt{n})}= \\
&(56\cdot(n+1)\sqrt{2n})^{-1}\cdot\exp((2-\frac{\pi}{\sqrt{3}})\sqrt{n}).
\end{align*}
Note that $2-\frac{\pi}{\sqrt{3}}>0.1862$, and that for $n\geqslant25000$, one has
\[
\exp(0.1362\sqrt{n})\geqslant56(n+1)\sqrt{2n},
\]
so that for all such $n$,
\[
\q(\Alt(n))\geqslant\exp(\frac{1}{20}\sqrt{n}).
\]
Therefore, still for $n\geqslant25000$, and using again the upper bound $n!\leqslant\e\sqrt{n}(\frac{n}{\e})^n$,
\begin{align*}
&\epsilon_{\q}(\Alt(n))\geqslant\frac{\log\log{\exp(\frac{1}{20}\sqrt{n})}}{\log\log{(n!/2)}}\geqslant\frac{\frac{1}{2}\log{n}-\log{20}}{\log{(1-\log{2}+\frac{1}{2}\log{n}+n(\log{n}-1))}}\geqslant \\
&\frac{\frac{1}{2}\log{n}-\log{20}}{\log{(\frac{3}{2}n\log{n})}}=\frac{\frac{1}{2}\log{n}-\log{20}}{\log{n}+\log{\frac{3}{2}}+\log\log{n}}=\frac{\frac{1}{2}-\frac{\log{20}}{\log{n}}}{1+\frac{\log{(3/2)}}{\log{n}}+\frac{\log\log{n}}{\log{n}}}\geqslant \\
&\frac{\frac{1}{2}-\frac{\log{20}}{\log{25000}}}{1+\frac{\log{(3/2)}}{\log{25000}}+\frac{\log\log{25000}}{\log{25000}}}>\epsilon_{\q}(\Alt(11)).\qedhere
\end{align*}
\end{proof}

\subsection{Groups of Lie type}\label{subsec2P3}

We first verify the asymptotic statements (1), (3) and (4) of Theorem \ref{mainTheo2} for the finite simple groups of Lie type. Before giving the actual proofs, we make some preparatory observations. As in the previous subsection, for a finite group $G$, denote by $\k(G)$ the number of conjugacy classes of $G$. As explained in Section \ref{secNot}, $S=\leftidx{^t}X_d(p^{ft})$, and we set $q:=p^f$. Then $\k(\Inndiag(S))\geqslant q^d$ (see, for instance, \cite[Theorem 1.1(1)]{FG12a}), and so, using \cite[Corollary 1, p.~506]{Ern61a}, we have
\begin{equation}\label{conjClEq}
\k(S)\geqslant\frac{q^d}{|\Inndiag(S):S|}\geqslant\frac{q^d}{\min\{d+1,q+1\}}
\end{equation}
Moreover, $|S|\leqslant q^{4d^2}$, as a simple case-by-case inspection shows. Therefore, using Kohl's bound \cite{Koh03a},
\[
|\Out(S)|\leqslant\log_2{|S|}\leqslant\log_2{(q^{4d^2})}=4d^2\log_2{q},
\]
it follows that
\begin{equation}\label{justBeforeN1Eq}
\omega(S)\geqslant\frac{\k(S)}{|\Out(S)|}\geqslant\frac{q^d}{4d^2\log_2{q}\min\{d+1,q+1\}}.
\end{equation}
In particular, for any $\epsilon>0$, there is an $N_1=N_1(\epsilon)\in\IN^+$ such that if $\max\{d,q\}\geqslant N_1$, then $\omega(S)\geqslant q^{(1-\epsilon)d}$.

In what follows, we explain how to suitably bound $\omicron(S)$ from above. Actually, our argument even provides an upper bound on $\omicron(\Inndiag(S))$. We will use the Landau notation $f=\Theta(g)$, which is defined for all functions $f$ and $g$ mapping from a common, unbounded set of positive real numbers to $\left[0,\infty\right)$, and it just means that $f=\O(g)$ and $g=\O(f)$. We will also write $\tau(n)$ for the number of divisors of a positive integer $n$.
\begin{enumerate}
\item We first consider unipotent element orders in $\Inndiag(S)$. By \cite[Corollary 0.5]{Tes95a}, the $p$-adic valuation of $\Exp(\Inndiag(S))$ is just $\lceil\log_p(H(X_d)+1)\rceil$ where $H(X_d)$ is the height of the highest root of the root system $X_d$. Denote by $h(X_d)$ the \emph{Coxeter number of $X_d$}, i.e., the order of any \emph{Coxeter element} of the Weyl group $W(X_d)$ (by definition, Coxeter elements are just those elements of $W(X_d)$ that can be obtained by multiplying together, in any order, the elements of any fixed set of simple roots). Then by \cite[Theorem, p.~84]{Hum90a}, $H(X_d)+1=h(X_d)$, so we can also write the $p$-adic valuation of $\Exp(\Inndiag(S))$ as $\lceil\log_p(h(X_d))\rceil$ and conclude that there are exactly $1+\lceil\log_p(h(X_d))\rceil$ elements in $\Ord(S)$ that are powers of $p$. The Coxeter numbers of the various indecomposable root systems can be found in tabulated form in \cite[Table 2, p.~80]{Hum90a}; for our asymptotic observations, the key property is that $h(X_d)=\Theta(d)$.
\item Secondly, we will consider the number of semisimple element orders in $\Inndiag(S)$. We can bound this number from above by the product of
\begin{itemize}
\item $\k_{\tor}(\Inndiag(S))$, the number of conjugacy classes of maximal tori of $\Inndiag(S)$, with
\item the maximum number of element orders in a maximal torus of $\Inndiag(S)$.
\end{itemize}
Concerning these two quantities:
\begin{itemize}
\item $\k_{\tor}(\Inndiag(S))$ is equal to the number of $\phi$-conjugacy classes in the corresponding Weyl group $W=W(X_d)$ (i.e., orbits of the action of $W$ on itself via $w^v=v^{-1}wv^{\phi}$), for a suitable $\phi=\phi(\leftidx{^t}X_d)\in\Aut(W)$, see \cite[Theorem 1.2(b), p.~1.8]{Gag73a}. Since there are only finitely many Lie symbols (and thus finitely many Weyl groups) for Lie type groups of a given rank $d$, one has $\k_{\tor}(\Inndiag(S))\leqslant g(d)$ for some unary function $g$, which is of subexponential growth (see \cite[Section 3]{Car81a} and use the asymptotics of the partition number $p(n)$ from Formula (\ref{partitionAsymptoticsEq})).
\item Recall that $\tau(n)$ denotes the number of (positive) divisors of the positive integer $n$. Since the order of a maximal torus in $\Inndiag(S)$ is at most $(q+1)^d$ (see \cite[Lemma 3.3]{Har92a}, for instance), we find (by Lagrange's theorem) that the maximum number of element orders in a maximal torus of $\Inndiag(S)$ is at most $h(d,q):=\max\{\tau(1),\tau(2),\ldots,\tau((q+1)^d)\}\leqslant 2(q+1)^{d/2}$.
\end{itemize}
\item Combining the above bounds, we get that
\begin{equation}\label{omicronLieEq}
\omicron(S)\leqslant\omicron(\Inndiag(S))\leqslant(1+\lceil\log_2{\Theta(d)}\rceil)\cdot g(d)\cdot h(d,q),
\end{equation}
and for each $\epsilon>0$, this is at most $q^{(\frac{1}{2}+\epsilon)d}$ if $\max\{d,q\}\geqslant N_2=N_2(\epsilon)$.
\end{enumerate}

\begin{proof}[Proof of Theorem \ref{mainTheo2}(1)]
Note that by the results on alternating groups from the previous subsection, it suffices to show that for finite simple groups of Lie type $S=\leftidx{^t}X_d(q^t)$, we have $\liminf_{|S|\to\infty}{\epsilon_{\omega}(S)}\geqslant\frac{1}{2}$. That is, we need to show that for each $\delta>0$, there is an $N=N(\delta)$ such that if $\max\{d,q\}\geqslant N$, then $\epsilon_{\omega}(S)\geqslant\frac{1}{2}-\delta$. Assume w.l.o.g.~that $\max\{d,q\}\geqslant N_1(\frac{1}{2})$, with $N_1(\epsilon)$ as defined above just after Formula (\ref{justBeforeN1Eq}). Hence $\omega(S)\geq q^{(1-1/2)d}=q^{d/2}$, and thus
\[
\epsilon_{\omega}(S)=\frac{\log\log{\omega(S)}}{\log\log{|S|}}\geqslant\frac{\log\log{q^{d/2}}}{\log\log{q^{4d^2}}}=\frac{\log{d}-\log{2}+\log\log{q}}{\log{4}+2\log{d}+\log\log{q}},
\]
which is bounded from below by $\frac{1}{2}-\delta$ if and only if
\[
\log{d}-\log{2}+\log\log{q}\geqslant(\frac{1}{2}-\delta)\log{4}+(1-2\delta)\log{d}+(\frac{1}{2}-\delta)\log\log{q},
\]
or equivalently,
\[
(\frac{1}{2}+\delta)\log\log{q}+2\delta\log{d}\geqslant2(1-\delta)\log{2},
\]
and this is indeed true if $\max\{d,q\}$ is large enough (relative to $\delta$).
\end{proof}

\begin{proof}[Proof of Theorem \ref{mainTheo2}(4)]
Note that by the definitions of $N_1(\epsilon)$ and $N_2(\epsilon)$, if
\[
\max\{d,q\}\geqslant\max\{N_1(\frac{1}{8}),N_2(\frac{1}{8})\},
\]
then $\omega(S)\geqslant q^{7d/8}$ and $\omicron(S)\leqslant q^{5d/8}$, whence
\[
\q(S)=\frac{\omega(S)}{\omicron(S)}\geqslant q^{d/4}.
\]
The second assertion of the statement is immediate from this, and the first also follows, using this lower bound on $\q(S)$ with the same argument used for proving statement (1).
\end{proof}

\begin{proof}[Proof of Theorem \ref{mainTheo2}(3)]
Note that
\[
\omega(S)\leqslant\omega(\Inndiag(S))\leqslant\k(\Inndiag(S))=q^{\O(1)d},
\]
where the implied upper bound on $\k(\Inndiag(S))$ is again by \cite[Theorem 1.1(1)]{FG12a}. But as explained after Formula (\ref{justBeforeN1Eq}), we also have $d=\O(\log_q{\omega(S)})$; combining these two facts, we get
\[
\omega(S)=\omega(\leftidx{^t}X_d(p^{ft}))=p^{\Theta(1)df}\text{ as }\max\{p,d,f\}\to\infty.
\]
On the other hand, $\omicron(S)\leqslant(1+\lceil\log_2{\Theta(d)}\rceil)\cdot g(d)\cdot h(d)$, and so, by the asymptotics of the number of divisors function $\tau$ (see for example \cite[Th{\'e}or{\`e}me 1]{NR83a}),
\begin{equation}\label{lieOmicronEq}
\omicron(S)=\omicron(\leftidx{^t}X_d(p^{ft}))\leqslant p^{\omicron(1)df}\text{ as }\max\{p,d,f\}\to\infty.
\end{equation}
\end{proof}

This concludes the verification of the three asymptotic statements in Theorem \ref{mainTheo2}. It remains to prove the universal lower bounds on $\epsilon_{\omega}(S)$ and $\epsilon_{\q}(S)$ from statements (2) and (5) respectively for finite simple groups of Lie type.

\begin{proof}[Proof of Theorem \ref{mainTheo2}(2)]
The proof idea is simply to carefully study lower bounds on $\omega(S)$ similar to the one in Formula (\ref{justBeforeN1Eq}) in order to obtain a theoretical argument which proves that $\epsilon_{\omega}(S)>\epsilon_{\omega}(\Alt(5))$ for all nonabelian finite simple groups $S$ nonisomorphic to $\Alt(5)$ except possibly those from an explicit finite list. These finitely many remaining exceptions are then dealt with using GAP \cite{GAP4}. By the results of Subsections \ref{subsec2P1} and \ref{subsec2P2}, we may assume that $S=\leftidx{^t}X_d(p^{ft})$ is of Lie type. Throughout, we set $q:=p^f$. Moreover, we will use the following conventions: We denote by $\log$ the function $\IR\rightarrow\IR\cup\{-\infty\}$ mapping
\[
x\mapsto\begin{cases}\log{x}, & \text{if }x>0, \\ -\infty, & \text{else}.\end{cases}
\]
Furthermore, by convention,
\begin{itemize}
\item $-\infty<x$ for all real numbers $x$,
\item $-\infty+x=-\infty$ for all real numbers $x$, and
\item $\frac{-\infty}{c}=-\infty$ for all $c>0$.
\end{itemize}
These conventions imply the following, which will be used various times without further mentioning: If $x_1,x_2,y_1,y_2$ are positive real numbers with $x_1\geqslant x_2$ and $y_1\leqslant y_2$, then
\[
\frac{\log\log{x_1}}{y_1}\geqslant\frac{\log\log{x_2}}{y_2}.
\]

Our arguments for bounding $\epsilon_{\omega}(S)$, with $S=\leftidx{^t}X_d(p^{ft})$, are split into the two cases \enquote{$d\leqslant2$} and \enquote{$d\geqslant3$}.
\begin{enumerate}
\item Case: $d\leqslant2$. There are seven families of Lie type groups $S$ of untwisted Lie rank $d$ at most $2$, as listed in Tables \ref{dAtMost2Table1} and \ref{dAtMost2Table2} below. We take the following unified approach to show that $\epsilon_{\omega}(S)>\epsilon_{\omega}(\Alt(5))$ for each of them apart from $A_1(4)\cong A_1(5)\cong \Alt(5)$: For each of the seven families, there is a reference in the literature for a precise formula for the number of conjugacy classes $\k(S)$, as displayed in Table \ref{dAtMost2Table1}. We note, however, that the given reference \cite{SM03a} for $\k(B_2(q))=\k(C_2(q))=\k(\PSp_4(q))$ when $q$ is odd appears to contain an error, because, according to \cite[Table 2]{SM03a}, $\k(\PSp_4(q))=\frac{1}{2}q^2+\frac{13}{4}q+\frac{23}{4}$ when $q\equiv3\Mod{4}$, which implies that $\k(\PSp_4(7))=53$, although actually (as one can check with GAP \cite{GAP4}) $\k(\PSp_4(7))=52$. The formula for $\k(\PSp_4(q))$ given in Table \ref{dAtMost2Table1} is based on the fact that in each of the three cases \enquote{$q$ is even}, \enquote{$q\equiv1\Mod{4}$} and \enquote{$q\equiv3\Mod{4}$}, the conjugacy class number $\k(\PSp_4(q))$ is a quadratic polynomial in $q$ (the authors would like to thank Frank L{\"u}beck for bringing this to their attention), and so in each of the three cases, the precise formula for $\k(\PSp_4(q))$ can be obtained by computing the conjugacy class number for three different values of $q$ from the respective congruence class, which can be done with GAP \cite{GAP4}. Column 2 of Table \ref{dAtMost2Table2} contains the well-known formula for $|\Out(S)|$, and column 3 contains an upper bound $\overline{|S|}$ on $|S|$, which is easily obtained from the well-known formula for the exact value of $|S|$. Recall from Formula (\ref{justBeforeN1Eq}) that
\begin{equation}\label{omegaOutEq}
\omega(S)\geqslant\k(S)/|\Out(S)|.
\end{equation}
Column 4 of Table \ref{dAtMost2Table2} lists a lower bound $\underline{\omega}(S)$ on $\omega(S)$ that can easily be derived from Formula (\ref{omegaOutEq}) and the information in Table \ref{dAtMost2Table1} (note that $f=\log_p{q}\leqslant\log_2{q}$). In order for $\epsilon_{\omega}(S)>\epsilon_{\omega}(\Alt(5))$ to hold, it is sufficient to have
\begin{equation}\label{fails1Eq}
\frac{\log\log{\underline{\omega}(S)}}{\log\log{\overline{|S|}}}>\epsilon_{\omega}(\Alt(5)),
\end{equation}
and with elementary calculus, one can check that Formula (\ref{fails1Eq}) holds in all but finitely many cases, which are listed in column 5 of Table \ref{dAtMost2Table2} as \enquote{fails 1}. Moreover, it is routine to check that among the finitely many groups $S$ corresponding to column 5 of Table \ref{dAtMost2Table2}, all but those listed in column 6 as \enquote{fails 2} satisfy the inequality
\[
\frac{\log\log{\lceil\frac{\k(S)}{|\Out(S)|}\rceil}}{\log\log{|S|}}>\epsilon_{\omega}(\Alt(5)),
\]
which is also sufficient for $\epsilon_{\omega}(S)>\epsilon_{\omega}(\Alt(5))$. Finally, for each of the remaining groups $S$ listed in column 6, one can compute the exact value of $\omega(S)$ with a simple GAP algorithm \cite{GAP4} written by the authors, and use this to check that $\epsilon_{\omega}(S)>\epsilon_{\omega}(\Alt(5))$ in those cases as well (for those exceptions $S$ listed in column 6 that are of type $A_1$ or ${^2}B_2$, one could alternatively use Kohl's formulas from \cite[Theorems 2.5 and 3.4]{Koh02a}). Our algorithm proceeds by first computing the conjugacy classes of $S$ using GAP's built-in command {\tt ConjugacyClasses}, and then computes the orbits of the action of $\Out(S)$ on the set of conjugacy classes of $S$ using the built-in commands
\begin{itemize}
\item {\tt AutomorphismGroup},
\item {\tt InnerAutomorphismsAutomorphismGroup},
\item {\tt RightTransversal}, and
\item {\tt IsConjugate}.
\end{itemize}

\begin{center}
\begin{longtable}[H]{|c|c|c|}
\caption{Formulas for $\k(S)$ in case $d\leqslant2$.}
\label{dAtMost2Table1}\\
\hline
$S$ & formula for $\k(S)$ & reference for $\k(S)$ \\ \hline
$A_1(q)$ & $\k(S)=\begin{cases}q+1, & \text{if }2\mid q, \\ \frac{q+5}{2}, & \text{if }q\nmid q\end{cases}$ & \cite[Formula (5.2), p.~43]{Mac81a} \\ \hline
$A_2(q)$ & $\k(S)=\begin{cases}q^2+q, & \text{if }q\equiv0,2\Mod{3}, \\ \frac{q^2+q+10}{3}, & \text{if }q\equiv1\Mod{3}\end{cases}$ & \cite[Formula (5.2), p.~43]{Mac81a} \\ \hline
$\leftidx{^2}A_2(q^2)$ & $\k(S)=\begin{cases}q^2+q+2, & \text{if }q\equiv0,1\Mod{3}, \\ \frac{q^2+q+12}{3}, & \text{if }q\equiv2\Mod{3}\end{cases}$ & \cite[Formula (6.13), p.~47]{Mac81a} \\ \hline
$B_2(q)\cong C_2(q)$ & $\k(S)=\begin{cases}q^2+2q+3, & \text{if }2\mid q, \\ \frac{q^2+6q+13}{2}, & \text{if }2\nmid q\end{cases}$ & \thead{\cite[Theorem 3.7.3]{Wal63a} for $q$ even; \\ \cite[Tables 1 and 2]{SM03a} for $q$ odd, but see \\ the paragraph before Formula (\ref{omegaOutEq}) above} \\ \hline
$G_2(q)$ & $\k(S)=\begin{cases}q^2+2q+9, & \text{if }q\equiv1,5\Mod{6}, \\ q^2+2q+8, & \text{if }q\equiv2,3,4\Mod{6}\end{cases}$ & \cite{Lue} \\ \hline
$\leftidx{^2}B_2(2^{2k+1})$ & $\k(S)=2^{2k+1}+3$ & \cite{Lue} \\ \hline
$\leftidx{^2}G_2(3^{2k+1})$ & $\k(S)=3^{2k+1}+8$ & \cite{Lue} \\ \hline
\end{longtable}
\end{center}

\begin{center}
\begin{longtable}[H]{|c|c|c|c|c|c|}
\caption{Remaining information for the case $d\leqslant2$.}
\label{dAtMost2Table2}\\
\hline
$S$ & $|\Out(S)|$ & $\overline{|S|}$ & $\underline{\omega}(S)$ & fails 1 & fails 2 \\ \hline
$A_1(q)$, $q\geqslant7$ & $\gcd(2,q-1)\cdot f$ & $q^3$ & $\frac{q}{4\log_2{q}}$ & $q\leqslant199$ & $q=7,8,9,11,13,16,25,27$ \\ \hline
$A_2(q)$, $q\geqslant3$ & $\gcd(3,q-1)\cdot 2f$ & $q^8$ & $\frac{q^2}{18\log_2{q}}$ & $q\leqslant25$ & $q=4,7,16$ \\ \hline
$\leftidx{^2}A_2(q^2)$, $q\geqslant3$ & $\gcd(3,q+1)\cdot 2f$ & $q^8$ & $\frac{q^2}{18\log_2{q}}$ & $q\leqslant25$ & $q=5,8$ \\ \hline
$B_2(q)\cong C_2(q)$, $q\geqslant3$ & $2f$ & $q^{10}$ & $\frac{q^2}{4\log_2{q}}$ & $q\leqslant9$ & none \\ \hline
$G_2(q)$, $q\geqslant3$ & $\gcd(2,q-1)f$ & $q^{14}$ & $\frac{q^2}{2\log_2{q}}$ & $q\leqslant5$ & none \\ \hline
$\leftidx{^2}B_2(2^{2k+1})$, $k\geqslant1$ & $2k+1$ & $2^{10k+6}$ & $2^{k+1}$ & $k\leqslant2$ & $k=1$ \\ \hline
$\leftidx{^2}G_2(3^{2k+1})$, $k\geqslant1$ & $2k+1$ & $3^{14k+8}$ & $3^{k+1}$ & none & none \\ \hline
\end{longtable}
\end{center}

\item Case: $d\geqslant3$. Let $S=\leftidx{^t}X_d(p^{ft})=\leftidx{^t}X_d(q^t)$. We will use the bound $|S|\leqslant q^{4d^2}$ from the beginning of this subsection. Note also that
\[
|\Out(S)|\leqslant\min\{d+1,q+1\}\cdot 6f,
\]
as a simple case-by-case analysis shows that the number of outer diagonal components of automorphisms of $S$ is always at most $\min\{d+1,q+1\}$, while the number of graph-field components is at most $6f$. Using Formula (\ref{conjClEq}), this implies that
\begin{align}\label{omegaLieEq}
\notag \omega(S) &\geqslant\frac{\k(S)}{|\Out(S)|}\geqslant\frac{q^d}{\min\{d+1,q+1\}\cdot|\Out(S)|}\geqslant\frac{q^d}{\min\{d+1,q+1\}^2\cdot 6f} \\
&=q^{d-\log_q(\min\{d+1,q+1\}^26f)},
\end{align}
and so, using also that the function $x\mapsto\frac{\log{x}}{x}$, assumes its global maximum on $\left[1,\infty\right)$ at $x=\e$,
\begin{align}\label{smileyEq}
\notag \epsilon_{\omega}(S) &=\frac{\log\log{\omega(S)}}{\log\log{|S|}}\geqslant\frac{\log\log{q^{d-\log_q(\min\{d+1,q+1\}^26f)}}}{\log\log{q^{4d^2}}} \\
\notag &=\frac{\log(d-\log_q(\min\{d+1,q+1\}^26f))+\log\log{q}}{\log(4d^2)+\log\log{q}} \\
\notag &\geqslant\frac{\log(d-\frac{2\log(d+1)}{\log{q}}-\frac{\log{6}}{\log{q}}-\frac{\log{f}}{f\log{p}})+\log{f}+\log\log{p}}{\log(4d^2)+\log{f}+\log\log{p}} \\
&\geqslant \frac{\log(d-\frac{2\log(d+1)}{\log{q}}-\frac{\log{6}}{\log{q}}-\frac{1}{\e\log{p}})+\log{f}+\log\log{p}}{\log(4d^2)+\log{f}+\log\log{p}}.
\end{align}
Note that the smallest value of $q$ that we need to consider is $2$ (we can ignore the group $\leftidx{^2}F_4(2)$, and the Tits group $\leftidx{^2}F_4(2)'$ is included among the sporadic groups in Subsection \ref{subsec2P1}). Let us make a subcase distinction:
\begin{enumerate}
\item Subcase: $q=2$. Then by Formula (\ref{smileyEq}),
\[
\epsilon_{\omega}(S)\geqslant\frac{\log(d-\frac{2\log(d+1)}{\log{2}}-\frac{\log{6}}{\log{2}})+\log\log{2}}{\log(4d^2)+\log\log{2}},
\]
which is strictly larger than $\epsilon_{\omega}(\Alt(5))$ for $d\geqslant19$.
\item Subcase: $q>2$. Then either $p\geqslant3$, or $p=2$ and $f\geqslant 3/2$, and so then
\[
\log{f}+\log\log{p}\geqslant\min\{\log\log{3},\log(3/2)+\log\log{2}\}>0.
\]
Hence in view of Formula (\ref{smileyEq}), if
\begin{equation}\label{IEq}
\frac{\log(d-\frac{2\log(d+1)}{\log{q}}-\frac{\log{6}}{\log{q}}-\frac{1}{\e\log{2}})}{\log(4d^2)}>\epsilon_{\omega}(\Alt(5)),
\end{equation}
then we also have $\epsilon_{\omega}(S)>\epsilon_{\omega}(\Alt(5))$. Conveniently, the left-hand side in Formula (\ref{IEq}) is monotonically increasing in $q$. Since we are currently assuming that $q>2$, we actually have $q\geqslant 2^{3/2}$, and so $\epsilon_{\omega}(S)>\epsilon_{\omega}(\Alt(5))$ as long as
\[
\frac{\log(d-\frac{2\log(d+1)}{1.5\log{2}}-\frac{\log{6}}{1.5\log{2}}-\frac{1}{\e\log{2}})}{\log(4d^2)}>\epsilon_{\omega}(\Alt(5)),
\]
which holds for $d\geqslant12$.
\end{enumerate}
So summarising what we know so far in the case $d\geqslant3$, we have $\epsilon_{\omega}(S)>\epsilon_{\omega}(\Alt(5))$ for all finite simple groups $S=\leftidx{^t}X_d(q^t)$ of Lie type of rank $d\geqslant19$, and for $d=11,\ldots,18$, only the case $q=2$ is open (to include $d=11$, substitute $q=3$ instead of $q=\sqrt{8}$ in Formula \ref{IEq} above, noting that non-integer values of $q$ are only relevant for $d=4$ (and $d=2$, which is not part of this case)). Similarly, using Formula (\ref{IEq}), we can show for $d=4,\ldots,10$ that if $q\geqslant q_0(d)$ with $q_0(d)$ as listed in Table \ref{table2} below, then $\epsilon_{\omega}(S)>\epsilon_{\omega}(\Alt(5))$.

For $d=3$, we argue as follows that one may choose $q_0(3)=10^5$: The number of graph-field automorphisms of any rank $3$ finite simple group of Lie type $S=\leftidx{^t}X_3(q^t)$ is at most $2f$, not just at most $6f$ as in the general case. This allows us to improve our upper bound on $|\Out(S)|$ from the beginning of the considerations for this case to $\min\{d+1,q+1\}\cdot 2f$, and repeating the chain of inequalities in Formula (\ref{smileyEq}) but with this improved bound, we find that
\[
\epsilon_{\omega}(S)\geqslant\frac{\log(3-\frac{2\log{4}}{\log{q}}-\frac{\log{2}}{\log{q}}-\frac{\log{f}}{\log{q}})}{\log{36}}.
\]
Now, assuming that $q\geqslant10^5$, we have $f\geqslant10$ or $p\geqslant\sqrt{10}$, the latter of which implies $p\geqslant5$. Hence
\[
\frac{\log{f}}{\log{q}}=\frac{\log{f}}{f\log{p}}\leqslant\max(\frac{\log{10}}{10\log{2}},\frac{1}{\e\log{5}})=\frac{\log{10}}{10\log{2}},
\]
and so
\[
\epsilon_{\omega}(S)\geqslant\frac{\log(3-\frac{2\log{4}}{\log{10^5}}-\frac{\log{2}}{\log{10^5}}-\frac{\log{10}}{10\log{2}})}{\log{36}}>\epsilon_{\omega}(\Alt(5)).
\]
This concludes the argument that $q_0(3)$ may be chosen as $10^5$.
\begin{center}
\begin{longtable}[H]{|c|c|}
\caption{Lower bounds $q_0(d)$ as described above}
\label{table2}\\
\hline
$d$ & $q_0(d)$ \\ \hline
${}\geqslant19$ & $2$ \\ \hline
$11,\ldots,18$ & $3$ \\ \hline
$9,10$ & $4$ \\ \hline
$8$ & $5$ \\ \hline
$7$ & $7$ \\ \hline
$6$ & $13$ \\ \hline
$5$ & $32$ \\ \hline
$4$ & $373$ \\ \hline
$3$ & $10^5$ \\ \hline
\end{longtable}
\end{center}
For dealing with these finitely many remaining groups, the authors proceeded as follows: They wrote a GAP function which computes for each finite simple group of Lie type $S=\leftidx{^t}X_d(p^{ft})$ a number $\underline{\k}(S)$ which is a lower bound on $\k(S)$ (in some cases, $\underline{\k}(S)=\k(S)$). More precisely:
\begin{itemize}
\item If $S$ is classical and at least one of the following holds
\begin{itemize}
\item $\leftidx{^t}X\in\{A,\leftidx{^2}A,B\}$;
\item $\leftidx{^t}X\in\{D,\leftidx{^2}D\}$ and $S$ is its own Schur cover;
\item $p=2$;
\end{itemize}
then set $\underline{\k}(S):=\k(S)$, to be computed according to \cite[Formulas (5.2) and (6.13)]{Mac81a} for $\leftidx{^t}X\in\{A,\leftidx{^2}A\}$, or \cite[Theorems 3.19(1), 3.13(1), 3.16(1,2) and 3.22(1,2)]{FG12a} for $\leftidx{^t}X\in\{B,C,D,\leftidx{^2}D\}$.
\item If $\leftidx{^t}X=C$ and $p>2$, set $\underline{\k}(S):=\lceil\frac{\k(\Sp_{2d}(q))}{2}\rceil$, to be computed according to \cite[Subsection 2.6, Case (B), statement (iii), p.~36]{Wal63a}.
\item If $\leftidx{^t}X\in\{D,\leftidx{^2}D\}$, $p>2$ and the Schur cover $\tilde{S}=\Omega^{\pm}_{2d}(q)$ of $S$ has nontrivial centre, set $\underline{\k}(S):=\lceil\frac{\k(\tilde{S})}{2}\rceil$, to be computed according to \cite[Theorem 3.18(1)]{FG12a}.
\item If $S$ is exceptional, then L{\"u}beck's database \cite{Lue} provides an exact formula for $\k(S)$, and we set $\underline{\k}(S):=\k(S)$.
\end{itemize}
The authors then wrote another GAP algorithm, which simply computes
\[
\frac{\log\log(\lceil\underline{\k}(\leftidx{^t}X_d(q^t))\rceil/|\Out(\leftidx{^t}X_d(q^t))|)}{\log\log{|\leftidx{^t}X_d(q^t)|}},
\]
a lower bound on $\epsilon_{\omega}(\leftidx{^t}X_d(q^t))$, checks whether this numerical value is strictly larger than $\epsilon_{\omega}(\Alt(5))$, and, if not, adds the group $S$ to a list of exceptions which is output by the algorithm at the end. It turns out that there are only four such exceptions, namely $A_3(5),\leftidx{^2}A_3(3^2),\leftidx{^2}A_4(4^2)$ and $D_4(3)$. For each of these four groups $S$, the value of $\omega(S)$ can be computed exactly with GAP \cite{GAP4}, as explained in the argument for \enquote{$d\leqslant2$} above, and one can thus check that $\epsilon_{\omega}(S)>\epsilon_{\omega}(\Alt(5))$ for these four groups $S$ as well, which concludes the proof.\qedhere
\end{enumerate}
\end{proof}

\begin{proof}[Proof of Theorem \ref{mainTheo2}(5)]
We use the same conventions with respect to $\log$ and $-\infty$ as described at the beginning of the proof of Theorem \ref{mainTheo2}(2). By the results of Subsections \ref{subsec2P1} and \ref{subsec2P2}, we may assume that $S=\leftidx{^t}X_d(q^t)$ is of Lie type with $q=p^f$. Let us first show that $\epsilon_{\q}(S)>\epsilon_{\q}(\M)$ when $S$ is \emph{exceptional}, i.e., when $\leftidx{^t}X_d \in \{\leftidx{^2}B_2,G_2,\leftidx{^2}G_2,F_4,\leftidx{^2}F_4,\leftidx{^3}D_4,E_6,\leftidx{^2}E_6,E_7,E_8\}$.

Recall that $\k_{\tor}(\Inndiag(S))$ denotes the number of conjugacy classes of maximal tori of $\Inndiag(S)$. As noted at the beginning of this subsection, $\k_{\tor}(\Inndiag(S))$ does not depend on $q$, but only on the symbol $\leftidx{^t}X_d$, and its values, which we give in the second column of Table \ref{exceptionalTable1} below, can be found in the two references \cite{DF91a} and \cite{Gag73a}, see the third column of Table \ref{exceptionalTable1} for more details. The Coxeter numbers $h(X_d)$ are given in the fourth column of Table \ref{exceptionalTable1}; see \cite[Table 2, p.~90]{Hum90a} for a reference. We argue as follows:

By the observations from the beginning of this subsection,
\begin{align*}
\omicron(S)&=\omicron(\leftidx{^t}X_d(q^t))\leqslant\omicron(\Inndiag(S))\leqslant\k_{\tor}(\Inndiag(S))\cdot 2(q+1)^{d/2}\cdot(1+\lceil\log_p(h(X_d))\rceil) \\
&\leqslant 2\k_{\tor}(\Inndiag(S))(1+\lceil\log_2(h(X_d))\rceil)\cdot(q+1)^{d/2}=:\overline{\o}(S).
\end{align*}
On the other hand, the information in L{\"u}beck's database \cite{Lue} allows one to derive a lower bound $\underline{\k}(S)$ on the number of conjugacy classes of $S$, found in the fifth column of Table \ref{exceptionalTable1}. Together with the upper bound $\overline{\Out}(S)=c(S)f$ on $|\Out(S)|$ from the sixth column in Table \ref{exceptionalTable1}, one obtains
\[
\omega(S)\geqslant\k(S)/|\Out(S)|\geqslant\underline{\k}(S)/\overline{\Out}(S).
\]
Finally, the well-known formulas for $|S|$ allow one to conclude that $|S|\leqslant q^{e(S)}$ with $e(S)$ as in the seventh column of Table \ref{exceptionalTable1}. One thus has
\[
\epsilon_{\q}(S)\geqslant\frac{\log\log{(\underline{\k}(S)/(\overline{\Out}(S)\cdot\overline{\omicron}(S)))}}{\log\log{q^{e(S)}}}\geqslant\frac{\log\log{(\underline{\k}(S)/(c(S)\overline{\omicron}(S)\cdot\log_2(q)))}}{\log\log{q^{e(S)}}},
\]
and one can check with elementary calculus that this lower bound on $\epsilon_{\q}(S)$ is strictly larger than $\epsilon_{\q}(\M)$ unless $q$ is from an explicit finite set specified in the eighth column of Table \ref{exceptionalTable1}. Below Table \ref{exceptionalTable1}, we explain how to deal with those finitely many remaining cases.

\begin{center}
\begin{longtable}[H]{|c|c|c|c|c|c|c|c|}
\caption{Rough treatment of exceptional groups.}
\label{exceptionalTable1}\\
\hline
$\leftidx{^t}X_d$ & $\k_{\tor}$ & $\k_{\tor}$ ref. & $h(X_d)$ & $\underline{\k}(S)$ & $\overline{\Out}$ & $e(S)$ & remaining $q$ \\ \hline
$\leftidx{^2}B_2$ & $3$ & \cite[Prop.~7.3]{Gag73a} & $4$ & $q^2=2^{2m+1}$ & $2f=2m+1$ & $10$ & $2^{3/2},\ldots,2^{21/2}$ \\ \hline
$G_2$ & $6$ & \cite[\S 5.2]{Gag73a} & $6$ & $q^2$ & $2f$ & $14$ & all $q<6947$ \\ \hline
$\leftidx{^2}G_2$ & $4$ & \cite[Prop.~7.4]{Gag73a} & $6$ & $q^2=3^{2m+1}$ & $2f=2m+1$ & $14$ & $3^{3/2},\ldots,3^{17/2}$ \\ \hline
$F_4$ & $25$ & \cite[\S 5.3]{Gag73a} & $12$ & $q^4$ & $2f$ & $52$ & all $q<157$ \\ \hline
$\leftidx{^2}F_4$ & $11$ & \cite[Prop.~7.5]{Gag73a} & $12$ & $q^4$ & $2f=2m+1$ & $52$ & $2^{3/2},\ldots,2^{13/2}$ \\ \hline
$\leftidx{^3}D_4$ & $7$ & \cite[Prop.~7.41]{Gag73a} & $6$ & $q^4$ & $3f$ & $29$ & all $q<79$ \\ \hline
$E_6$ & $25$ & \cite{DF91a} & $12$ & $q^6/3$ & $6f$ & $78$ & all $q<59$ \\ \hline
$\leftidx{^2}E_6$ & $25$ & \cite{DF91a} & $12$ & $q^6/3$ & $6f$ & $78$ & all $q<59$ \\ \hline
$E_7$ & $60$ & \cite{DF91a} & $18$ & $q^7/2$ & $2f$ & $133$ & all $q<29$ \\ \hline
$E_8$ & $112$ & \cite{DF91a} & $30$ & $q^8$ & $f$ & $248$ & all $q<16$ \\ \hline
\end{longtable}
\end{center}

Let us now discuss how to handle the finitely many remaining exceptional Lie type groups. Table \ref{exceptionalTable2} gives an overview of references with information that allows one to compute the exact number $\omicron_{\ssrm}(\Inndiag(S))$ of semisimple element orders in $\Inndiag(S)$, as well as $\omicron(S)$ (except for $\omicron(E_8(q))$).

More precisely, the second column of Table \ref{exceptionalTable2} gives a reference for the complete list of cyclic structures of maximal tori of $\Inndiag(S)$, from which the exponents, and thus the sets of element orders, of the maximal tori of $\Inndiag(S)$ can be computed. Since every semisimple element of $\Inndiag(S)$ lies in some maximal torus, this is enough to compute $\omicron_{\ssrm}(\Inndiag(S))$.

Moreover, for all exceptional finite simple Lie type groups $S$ except for $E_8(q)$, there is a result in the literature specifying a subset $\nu(S)$ of $\Ord(S)$ such that $\Ord(S)$ is the closure of $\nu(S)$ under taking divisors. These references are given in the third column of Table \ref{exceptionalTable2}.

Now, in those cases where $\omicron(S)$ can be computed exactly (i.e., for all exceptional $S$ apart from the groups $E_8(q)$), go through the finitely many remaining values of $q$ from the last column in Table \ref{exceptionalTable1}, set $\k_0(S):=\k(\Inndiag(S))/|\Outdiag(S)|$ (the precise values of $\k(\Inndiag(S))$ in the various cases can be read off from L{\"u}beck's database \cite{Lue}) and $\underline{\omega}(S):=\lceil\k_0(S)/|\Out(S)|\rceil$, and check whether the following lower bound on $\epsilon_{\q}(S)$ is greater than $\epsilon_{\q}(\M)$:
\[
\frac{\log\log{(\underline{\omega}(S)/\omicron(S)+3)}}{\log\log{|S|}}.
\]
For $S=E_8(q)$, define $\underline{\omega}(S)$ as for the other exceptional groups, but additionally, set $\overline{\omicron}(S):=\omicron_{\ssrm}(\Inndiag(S))\cdot(1+\lceil\log_p(30)\rceil)$, and check whether the following lower bound on $\epsilon_{\q}(S)$ is greater than $\epsilon_{\q}(\M)$:
\[
\frac{\log\log{(\underline{\omega}(S)/\overline{\omicron}(S)+3)}}{\log\log{|S|}}.
\]
Only very few cases resist even these refined checks, and they are listed in the last column of Table \ref{exceptionalTable2} and will be discussed further below.

\begin{center}
\begin{longtable}[H]{|c|c|c|c|}
\caption{Refined treatment of exceptional groups.}
\label{exceptionalTable2}\\
\hline
$\leftidx{^t}X_d$ & ref.~for cyclic structure & ref.~for $\nu(S)$ & remaining $q$ \\ \hline
$\leftidx{^2}B_2$ & \cite[Prop.~7.3]{Gag73a} & \cite[Theorem 2]{Shi92a} & $2^{3/2}$ \\ \hline
$G_2$ & \cite[\S 5.2, Table 5.1]{Gag73a} & \cite[Lemma 1.4]{VS13a} & none \\ \hline
$\leftidx{^2}G_2$ & \cite[Prop.~7.4]{Gag73a} & \cite[Lemma 4]{BS93a} & $3^{3/2}$ \\ \hline
$F_4$ & \cite[\S 5.3, Table 5.2]{Gag73a} & \cite[Theorem 3.1]{GZ16a} & none \\ \hline
$\leftidx{^2}F_4$ & \cite[\S 7.4, Table 7.3]{Gag73a} & \cite[Lemma 3]{DS99a} & $2^{3/2}$ \\ \hline
$\leftidx{^3}D_4$ & \cite[\S 7.5, Table 7.5]{Gag73a} & \cite[Theorem 3.2]{GZ16a} & $2$ \\ \hline
$E_6$ & \cite{DF91a} & \cite[Theorem 1]{But13a} & none \\ \hline
$\leftidx{^2}E_6$ & \cite{DF91a} & \cite[Theorem 1]{But13a} & $2$ \\ \hline
$E_7$ & \cite{DF91a} & \cite[Theorem 2]{But16a} & none \\ \hline
$E_8$ & \cite{DF91a} & no ref. & none \\ \hline
\end{longtable}

\end{center}

We now discuss the remaining five exceptional Lie type groups $S$ specified by the last column in Table \ref{exceptionalTable2}.
\begin{itemize}
\item Note that by the definition of $\epsilon_{\q}$, for every nonabelian finite simple group $S$, $\epsilon_{\q}(S)\geqslant\frac{\log\log{4}}{\log\log{|S|}}$. For $S=\leftidx{^2}B_2(8)$, this trivial lower bound is actually larger than $\epsilon_{\q}(\M)$.
\item For $S=\leftidx{^2}G_2(27)$: By \cite[Lemma 4]{BS93a}, we have $\omicron(S)=11$, and in order to conclude that $\epsilon_{\q}(S)>\epsilon_{\q}(\M)$, it is enough to know that $\omega(S)\geqslant13$, which we will show now. Let $\C$ be the set of conjugacy classes of $S$, and let $\Mcal$ be the multiset of positive integers obtained by replacing each class $C\in\C$ by the common order (in $S$) of the elements of $C$. Then the elements of $\Mcal$ are just the element orders in $S$, and the multiplicity in $\Mcal$ of each element order $o$ in $S$ is just $\k_o(S)$, the number of $S$-conjugacy classes of order $o$ elements in $S$. From the page on $\leftidx{^2}G_2(27)$ in the ATLAS of Finite Group Representations \cite{ATLAS}, one can read off that $\Mcal$ is the following multiset (where the notation $x_n$ is shorthand for $n$ copies of $x$):
\[
\Mcal=\{1_1,2_1,3_3,6_2,7_1,9_3,13_6,14_3,19_3,26_6,37_6\}.
\]
Since $|\Out(S)|=3$, each orbit of the natural action of $\Aut(S)$ on $\C$ is of length $1$ or $3$, and so, writing $\k_o(S)=3\cdot q_o+r_o$ with $r_o\in\{0,1,2\}$, we find that $\omega_o(S)\geqslant q_o+r_o$ and
\[
\omega(S)=\sum_{o\in\Ord(S)}{\omega_o(S)}\geqslant\sum_{o\in\Ord(S)}{(q_o+r_o)}=15>13,
\]
as required.
\item For $S=\leftidx{^2}F_4(8)$: By \cite[Lemma 3]{DS99a}, we have $\omicron(S)=28$. In order to conclude that $\epsilon_{\q}(S)>\epsilon_{\q}(\M)$, it is sufficient to show that $\omega(S)\geqslant52$, which we will do now, based on the extended character table of $S$ available in GAP \cite{GAP4}. There are 19 unipotent conjugacy classes in $S$. The following is a drawing of the finite digraph whose vertices are these $19$ unipotent conjugacy classes and which has an edge $c_1\rightarrow c_2$ if and only if the elements in $c_1$ square to elements in $c_2$:

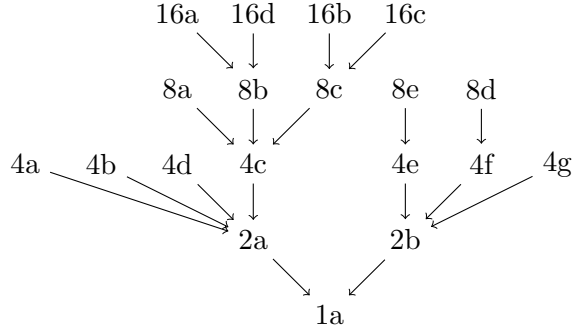
\begin{center}
\begin{tikzpicture}[->]
\node (v1a) at (0,0) {1a};
\node (v2a) at (-1,1) {2a};
\node (v2b) at (1,1) {2b};
\node (v4a) at (-4,2) {4a};
\node (v4b) at (-3,2) {4b};
\node (v4d) at (-2,2) {4d};
\node (v4c) at (-1,2) {4c};
\node (v4e) at (1,2) {4e};
\node (v4f) at (2,2) {4f};
\node (v4g) at (3,2) {4g};
\node (v8a) at (-2,3) {8a};
\node (v8b) at (-1,3) {8b};
\node (v8c) at (0,3) {8c};
\node (v8e) at (1,3) {8e};
\node (v8d) at (2,3) {8d};
\node (v16a) at (-2,4) {16a};
\node (v16d) at (-1,4) {16d};
\node (v16b) at (0,4) {16b};
\node (v16c) at (1,4) {16c};
\draw (v2a) edge (v1a);
\draw (v2b) edge (v1a);
\draw (v4a) edge (v2a);
\draw (v4b) edge (v2a);
\draw (v4d) edge (v2a);
\draw (v4c) edge (v2a);
\draw (v4e) edge (v2b);
\draw (v4f) edge (v2b);
\draw (v4g) edge (v2b);
\draw (v8a) edge (v4c);
\draw (v8b) edge (v4c);
\draw (v8c) edge (v4c);
\draw (v8e) edge (v4e);
\draw (v8d) edge (v4f);
\draw (v16a) edge (v8b);
\draw (v16d) edge (v8b);
\draw (v16b) edge (v8c);
\draw (v16c) edge (v8c);
\end{tikzpicture}
\end{center}

Noting that $|\Out(S)|=3$, one can use this graph to argue that distinct conjugacy classes of elements of $S$ whose order lies in $\{1,2,8,16\}$ span distinct $\Aut(S)$-orbits, and that no two of the conjugacy classes 4a, 4c, 4e, 4f, 4g span the same $\Aut(S)$-orbit. It follows that the number of $\Aut(S)$-orbits consisting of unipotent elements is at least $17$. Now, for each non-unipotent element order $o\in\Ord(S)$, write the number of conjugacy classes of order $o$ elements in $S$ as $3\cdot q_o+r_o$ with $q_o\in\IN$ and $r_o\in\{0,1,2\}$. Then, as for $S=\leftidx{^2}G_2(27)$ above, since $|\Out(S)|=3$, $\omega_o(S)\geqslant q_o+r_o$. Denoting by $\Ord'(S)$ the set of non-unipotent element orders in $S$, it follows that
\[
\omega(S)\geqslant 17+\sum_{o\in\Ord'(S)}{(q_o+r_o)}=53>52,
\]
as asserted.
\item For $S=\leftidx{^3}D_4(8)$: By \cite[Theorem 3.2]{GZ16a}, $\omicron(S)=14$, and in order to deduce that $\epsilon_{\q}(S)>\epsilon_{\q}(\M)$, it is enough to know that $\omega(S)\geqslant15$, i.e., that $S$ is not an AT-group. But this is clear by Zhang's characterisation of AT-groups \cite[Theorem 3.1]{Zha92a}. Alternatively, one can use the extended character table of $S$ available in GAP \cite{GAP4} to see that $S$ has $20$ distinct conjugacy class lengths. In particular, $\omega(S)\geqslant 20>15$.
\item For $S=\leftidx{^2}E_6(2^2)$: By \cite[Theorem 1]{But13a}, $\omicron(S)=27$, and in order to get that $\epsilon_{\q}(S)>\epsilon_{\q}(\M)$, it is enough to know that $\omega(S)\geqslant48$. Using the extended character table of $S$ available in GAP \cite{GAP4}, we see that $S$ has $77$ distinct conjugacy class lengths, whence $\omega(S)\geqslant 77>48$.
\end{itemize}

This concludes our discussion of exceptional $S$, and we may assume from now on that $S=\leftidx{^t}X_d(q^t)$ is a classical finite simple group of Lie type. Note that by \cite[Table 2, p.~90]{Hum90a}, the Coxeter number $h(X_d)$ is at most $2d$. Moreover, recall the notation $\k_{\tor}(\Inndiag(S))$ for the number of conjugacy classes of maximal tori in the inner diagonal automorphism group of $S$. We proceed in several steps, showing successively stronger statements:
\begin{enumerate}
\item \emph{$\epsilon_{\q}(S)>\epsilon_{\q}(\M)$ if $d\geqslant1012$, or $q\geqslant3$ and $d\geqslant180$.} By \cite[Section 3]{Car81a}, \cite[pp.~437f.]{Erd42a} and the fact that the function $x\mapsto \sqrt{x}+\sqrt{d-x}$, has maximum value $\sqrt{2d}$ on the domain $\left[0,d\right]$, we have the following, where $p(k)$ denotes the number of ordered integer partitions of $k\in\IN$:
\begin{itemize}
\item If $\leftidx{^t}X \in \{A,\leftidx{^2}A\}$, then $\k_{\tor}(\Inndiag(S))=p(d+1)\leqslant\exp(\frac{2\pi}{\sqrt{6}}\sqrt{d+1})$.
\item If $\leftidx{^t}X \in \{B,C,D,\leftidx{^2}D\}$, then
\begin{align*}
\k_{\tor}(\Inndiag(S)) &=\sum_{i=0}^d{p(i)p(d-i)}\leqslant\sum_{i=0}^d{\exp(\frac{2\pi}{\sqrt{6}}(\sqrt{i}+\sqrt{d-i}))} \\
&\leqslant(d+1)\exp(\frac{2\pi}{\sqrt{3}}\sqrt{d}).
\end{align*}
\end{itemize}
Set
\[
g_1(d):=(d+1)\exp(\frac{2\pi}{\sqrt{3}}\sqrt{d})
\]
and $h_1(d,q):=2(q+1)^{d/2}$. Then by the observations from the beginning of this subsection (see Formula (\ref{omicronLieEq})) and letting $T$ range over the maximal tori in $\Inndiag(S)$,
\begin{align*}
\omicron(S)  &\leqslant\k_{\tor}(\Inndiag(S))\cdot\max_T{\omicron(T)}\cdot(1+\lceil\log_p(h(X_d))\rceil) \\
&\leqslant g_1(d)\cdot h_1(d,q)\cdot(1+\lceil\log_p(2d)\rceil).
\end{align*}
On the other hand, analogously to Formula (\ref{omegaLieEq}) from the proof of Theorem \ref{mainTheo2}(2), we have
\[
\omega(S)\geqslant q^{d-\log_q{(\min\{d+1,q+1\}^2\cdot 2f)}};
\]
since we can use $2f$ instead of $6f$ in the exponent because the number of graph-field automorphisms of $S$ is at most $2f$ (as $6f$ only occurs when $d=4$). Combining these bounds on $\omicron(S)$ and $\omega(S)$, we get that
\begin{align}\label{qOfSBoundEq}
\notag \q(S) &\geqslant q^{d-\log_q(\min\{d+1,q+1\}^2\cdot 2f\cdot(1+\lceil\log_p(2d)\rceil)\cdot(d+1)\exp(2\pi\sqrt{d/3})\cdot 2(q+1)^{d/2})} \\
&\geqslant q^{(1-\frac{\log{(q+1)}}{2\log{q}})d-\frac{\frac{2\pi}{\sqrt{3}}\sqrt{d}+3\log{(d+1)}+\log{(2+\frac{\log{(2d)}}{\log{2}})}+\log{4}}{\log{q}}-\frac{\log{f}}{\log{q}}}.
\end{align}
In view of $|S|\leqslant q^{4d^2}$, this implies that
\begin{equation}\label{epsilonQRoughEq}
\epsilon_{\q}(S)\geqslant\frac{\log{((1-\frac{\log{(q+1)}}{2\log{q}})d-\frac{\frac{2\pi}{\sqrt{3}}\sqrt{d}+3\log{(d+1)}+\log{(2+\frac{\log{(2d)}}{\log{2}})}+\log{4}}{\log{q}}-\frac{\log{f}}{\log{q}})}+\log\log{q}}{\log{(4d^2)}+\log\log{q}}.
\end{equation}
For $q=2$, the lower bound in Formula (\ref{epsilonQRoughEq}) becomes
\[
\frac{\log{((1-\frac{\log{3}}{\log{4}})d-\frac{\frac{2\pi}{\sqrt{3}}\sqrt{d}+3\log{(d+1)}+\log{(2+\frac{\log{2d}}{\log{2}})}+\log{4}}{\log{2}})}+\log\log{2}}{\log{(4d^2)}+\log\log{2}},
\]
which is indeed larger than $\epsilon_{\q}(\M)$ if $d\geqslant1012$. So we may henceforth assume that $q\geqslant3$; in particular, $\log\log{q}>0$. Moreover,
\[
\frac{\log{f}}{\log{q}}=\frac{\log{f}}{f\log{p}}\leqslant\frac{1}{\e\log{2}}
\]
since the function $x\mapsto\frac{\log{x}}{x}$, assumes its maximum on $\left(0,\infty\right)$ at $x=\e$. Combining this with Formula (\ref{epsilonQRoughEq}), we conclude that
\begin{align*}
\epsilon_{\q}(S) &\geqslant\frac{\log{((1-\frac{\log{(q+1)}}{2\log{q}})d-\frac{\frac{2\pi}{\sqrt{3}}\sqrt{d}+3\log{(d+1)}+\log{(2+\frac{\log{(2d)}}{\log{2}})}+\log{4}}{\log{q}}-\frac{\log{f}}{\log{q}})}}{\log{(4d^2)}} \\
&\geqslant\frac{\log{((1-\frac{\log{4}}{\log{9}})d-\frac{\frac{2\pi}{\sqrt{3}}\sqrt{d}+3\log{(d+1)}+\log{(2+\frac{\log{(2d)}}{\log{2}})}+\log{4}}{\log{3}}-\frac{1}{\e\log{2}})}}{\log{(4d^2)}},
\end{align*}
which is larger than $\epsilon_{\q}(\M)$ if $d\geqslant180$, as required.
\item \emph{$\epsilon_{\q}(S)>\epsilon_{\q}(\M)$ if $d\geqslant91$, or $q\geqslant3$ and $d\geqslant54$.} We will use the main results of \cite{BG07a}, which provide information on the cyclic structure of maximal tori of $S$ (\emph{not} just $\Inndiag(S)$). We need two new terminologies:
\begin{itemize}
\item Call a set $M\subseteq\{1,\ldots,(q+1)^d\}$ \emph{sufficient} if and only if for every maximal torus $T$ of $\Inndiag(S)$, there is an $o\in M$ such that the group exponent $\Exp(T\cap S)$ divides $o$. For every sufficient set $M$ of positive integers, the set of semisimple element orders in $S$ is just the union of the sets of divisors of the $o\in M$.
\item If $\lambda$ is an ordered integer partition, say $\lambda\vdash n$, then we denote by $\overline{\lambda}$ the unique ordered integer partition such that
\begin{enumerate}
\item $\overline{\lambda}\vdash n$;
\item for each positive integer $k>1$, we have that $k$ is a part of $\overline{\lambda}$ if and only if it is a part of $\lambda$; and
\item each positive integer $k>1$ occurs with multiplicity at most $1$ in $\overline{\lambda}$.
\end{enumerate}
For example, $\overline{(4,3,3,2,2,2,1)}=(4,3,2,1,1,1,1,1,1,1,1)$. An ordered integer partition $\lambda$ will be called \emph{reduced} if and only if $\lambda=\overline{\lambda}$.
\end{itemize}
We noted earlier that the number of semisimple element orders of $S$ can be bounded from above by the product of
\begin{itemize}
\item the number $\k_{\tor}(\Inndiag(S))$ of conjugacy classes of maximal tori of $\Inndiag(S)$ with
\item the maximum number of divisors of a given positive integer between $1$ and $(q+1)^d$.
\end{itemize}
However, in such a count, many conjugacy classes $C$ of maximal tori are usually redundant in the sense that there are lots of other conjugacy classes $C'$ such that the exponent of an element of $C'$ is a multiple of the exponent of an element of $C$. It is therefore more efficient to replace $\k_{\tor}(\Inndiag(S))$ by a sufficiently good upper bound on the size of a sufficient set of positive integers for $S$. Hence our goal is to find a \enquote{small} sufficient set $M(S)$ of positive integers (and an upper bound on $|M(S)|$). The notion of a reduced partition as well as the results of \cite{BG07a} will help us achieve this goal. More precisely:
\begin{itemize}
\item Assume first that $S=\leftidx{^t}A_d(q^t)$ for some $t\in\{1,2\}$. In the notation of \cite{BG07a}, $S=\PSL^{\epsilon}_{d+1}(q)$ for some $\epsilon\in\{+,-\}$. Then the conjugacy classes of maximal tori of $S$ are in bijection with ordered integer partitions $\lambda\vdash d+1$. Denote by $T_{\lambda}$ any fixed maximal torus of $\Inndiag(S)$ whose conjugacy class corresponds to $\lambda$. As we will now explain, the results \cite[Theorems 2.1 and 2.2]{BG07a} allow us to determine $\Exp(T_{\lambda}\cap S)$ in terms of $\lambda=(n_1,\ldots,n_t)$ (where $n_1\geqslant n_2\geqslant\cdots\geqslant n_t$). One can check that the cyclic decompositions $a_1\times\cdots\times a_t$ of $T_{\lambda}\cap S$ given in \cite[Theorems 2.1 and 2.2]{BG07a} are canonical in the sense that $a_t\mid a_{t-1}\mid\cdots\mid a_2\mid a_1$ (note, however, that some of the later $a_i$ may be $1$). It follows that the exponent of $T_{\lambda}\cap S$ is always the order of the first (i.e., left-most) cyclic direct factor in the decomposition from \cite[Theorems 2.1 and 2.2]{BG07a}. Hence, setting
\[
d_1(\lambda):=d_1^{\epsilon}(\lambda,q):=\lcm(q^{n_1}-(\epsilon1)^{n_1},\ldots,q^{n_t}-(\epsilon1)^{n_t}),
\]
we have the following:
\begin{itemize}
\item If $t>2$, then $\Exp(T_{\lambda}\cap S)=d_1(\lambda)$.
\item If $t=2$, then $\Exp(T_{\lambda}\cap S)=\frac{d_1(\lambda)}{\gcd((d+1)/\gcd(n_1,n_2),q-\epsilon1)}$.
\item If $t=1$, then $\Exp(T_{\lambda}\cap S)=\frac{d_1(\lambda)}{\gcd(d+1,q-\epsilon1)(q-\epsilon1)}=\frac{q^{d+1}-(\epsilon1)^{d+1}}{\gcd(d+1,q-\epsilon1)(q-\epsilon1)}$.
\end{itemize}
From this, it is immediate that $\Exp(T_{\lambda}\cap S)$ divides $\Exp(T_{\overline{\lambda}}\cap S)$. Therefore, the set of all positive integers of the form $d_1^{\epsilon}(\lambda,q)$ where $\lambda$ is a \emph{reduced} partition of $d+1$ is a sufficient set of positive integers for $S=\PSL^{\epsilon}_{d+1}(q)$; we define $M(S)$ to be this set. The number of reduced partitions of $d+1$, and thus the cardinality of $M(S)$, is at most $\sum_{i=0}^{d+1}{s(i)}$ where, as in Subsection \ref{subsec2P2}, $s(i)$ denotes the number of partitions of $i$ into pairwise distinct parts.
\item Now assume that $S$ is orthogonal or symplectic. Then the conjugacy classes of maximal tori of $\Inndiag(S)$ are in bijection with (certain, depending on the case) conjugacy classes of signed permutations of $\{\pm1,\ldots,\pm d\}$, and hence they are in bijection with (certain) ordered pairs $\bm{\lambda}=(\lambda_+,\lambda_-)$ of ordered integer partitions (corresponding to the multisets of lengths of positive, respectively negative, cycles) such that if $\lambda_+\vdash d_+$ and $\lambda_-\vdash d_-$, then $d_++d_-=d$. For each such pair $\bm{\lambda}=(\lambda_+,\lambda_-)$, write $\lambda_{\epsilon}=(n^{\epsilon}_1,\ldots,n^{\epsilon}_{t_{\epsilon}})$ for $\epsilon\in\{+,-\}$, and set
\begin{align*}
&d_1(\bm{\lambda}):= \\
&\begin{cases}\lcm(\{q^{n^+_{i_+}}-1,q^{n^-_{i_-}}+1\mid i_{\epsilon}=1,\ldots,t_{\epsilon}\}), & \text{if }t_++t_->1\text{ or }p=2, \\ \frac{1}{2}\lcm(\{q^{n^+_{i_+}}-1,q^{n^-_{i_-}}+1\mid i_{\epsilon}=1,\ldots,t_{\epsilon}\}), & \text{if }t_++t_-=1\text{ and }p>2.\end{cases}
\end{align*}
Moreover, set $\overline{\bm{\lambda}}:=(\overline{\lambda_+},\overline{\lambda_-})$. Then $d_1(\bm{\lambda})\mid d_1(\overline{\bm{\lambda}})$. Moreover, every pair $\bm{\lambda}$ corresponds to a conjugacy class $C_{\bm{\lambda}}$ of maximal tori of $\Inndiag(C_d(q))$, and by \cite[Theorem 3]{BG07a}, $d_1(\bm{\lambda})$ is the exponent of any torus in $C_{\bm{\lambda}}$, whence $d_1(\bm{\lambda})\leqslant(q+1)^d$. Finally, by \cite[Theorems 3--7]{BG07a}, for each $S=\leftidx{^t}X_d(q^t)$ with $X\in\{B,C,D\}$, if $\bm{\lambda}$ corresponds to a conjugacy class $C_{\bm{\lambda}}$ of maximal tori of $\Inndiag(S)$, then the exponent of any representative of $C_{\bm{\lambda}}$ divides $d_1(\bm{\lambda})$. Combining these facts, it follows that for any symplectic or orthogonal group $S=\leftidx{^t}X_d(q^t)$, the set $M(S)$ of all numbers of the form $d_1(\bm{\lambda})$ such that both entries of $\bm{\lambda}$ are reduced partitions is sufficient, and
\[
|M(S)|\leqslant\sum_{d_++d_-=d}\left(\sum_{i_+=0}^{d_+}{\sum_{i_-=0}^{d_-}{s(i_+)s(i_-)}}\right),
\]
where, again, $s(i)$ denotes the number of ordered integer partitions of $i$ with pairwise distinct parts.
\end{itemize}
It is not difficult to check that for each $d\geqslant1$,
\[
\sum_{i=0}^{d+1}{s(i)}\leqslant\sum_{d_++d_-=d}\left(\sum_{i_+=0}^{d_+}{\sum_{i_-=0}^{d_-}{s(i_+)s(i_-)}}\right)=:g_2(d),
\]
and hence every classical finite simple group of Lie type $S$ of untwisted rank $d$ admits a sufficient set $M(S)$ of positive integers of size at most $g_2(d)$.

It will also be necessary to use sharper upper bounds on
\[
\max\{\tau(1),\ldots,\tau((q+1)^d)\},
\]
where $\tau(n)$ denotes the number of (positive) divisors of $n$, exploiting that by assumption, $(q+1)^d$ is \enquote{large}. Assume first that $d\geqslant91$ (and $q=2$). Then
\[
(q+1)^{0.311\cdot d}=3^{0.311\cdot d}\geqslant 3^{0.311\cdot 91}\geqslant \exp(\exp(0.311^{-1}\cdot 1.538\cdot\log{2})),
\]
which (by taking logarithms twice) is equivalent to
\[
\frac{1.538\log{2}}{\log\log{(q+1)^{0.311\cdot d}}}\leqslant0.311.
\]
Hence by \cite[Th{\'e}or{\`e}me 1]{NR83a}, for every positive integer $k\geqslant(q+1)^{0.311\cdot d}$, we have $\tau(k)\leqslant k^{0.311}$. We claim that this implies that
\[
\max\{\tau(1),\ldots,\tau((q+1)^d)\}\leqslant(q+1)^{0.311d}.
\]
Indeed, let $k$ be a positive integer with $1\leqslant k\leqslant(q+1)^d$. If $k\geqslant(q+1)^{0.311\cdot d}$, then by the above, $\tau(k)\leqslant k^{0.311}\leqslant(q+1)^{0.311\cdot d}$. And if $k<(q+1)^{0.311\cdot d}$, then $\tau(k)\leqslant k<(q+1)^{0.311\cdot d}$.

Now, following the argument from the case (1) and replacing $\min\{d+1,q+1\}$ in Formula (\ref{qOfSBoundEq}) by $q+1=3$ instead of $d+1$, we obtain that
\[
\epsilon_{\q}(S)\geqslant\frac{\log{((1-0.311\cdot\frac{\log{3}}{\log{2}})d-\frac{\log{g_2(d)}+2\log{3}+\log{(2+\frac{\log{(2d)}}{\log{2}})}+\log{2}}{\log{2}})}+\log\log{2}}{\log{(4d^2)}+\log\log{2}},
\]
and one can check this lower bound to be greater than $\epsilon_{\q}(\M)$ for $d=91,\ldots,1011$ (we use the table of values of $s(k)$ from \cite{OEIS} to compute $g_2(d)$). This concludes the discussion of this case for $q=2$, so we may henceforth assume that $q\geqslant3$ and $d\in\{54,\ldots,179\}$. Then we repeat the argument for $q=2$ with $\frac{1}{3}$ instead of $0.311$, obtaining that
\[
\epsilon_{\q}(S)\geqslant\frac{\log{((1-\frac{\log{4}}{\log{27}})d-\frac{\log{g_2(d)}+2\log{(d+1)}+\log{(2+\frac{\log{(2d)}}{\log{2}})}+\log{2}}{\log{3}}-\frac{1}{\e\log{2}})}}{\log{(4d^2)}},
\]
which can also be checked to be larger than $\epsilon_{\q}(\M)$ for each $d\in\{54,\ldots,179\}$.

\item \emph{$\epsilon_{\q}(S)>\epsilon_{\q}(\M)$ if $d\geqslant54$.} Compared to the previous case, the only groups $S$ additionally included here are those where $q=2$ and $d\in\{54,\ldots,90\}$. Observe that the Schur cover of $\leftidx{^t}X_d(2^t)$ embeds into $\leftidx{^t}X_{d+1}(2^t)$; for example, $\SL_{d+1}(2)$ embeds into $\PSL_{d+2}(2)$ via
\[
M\mapsto\overline{\begin{pmatrix}M & 0 \\ 0 & 1\end{pmatrix}},
\]
and similar arguments work for the other types of classical groups. In particular,
\begin{equation}\label{omicronssEq}
\omicron_{\ssrm}(\leftidx{^t}X_d(2^t))\leqslant\omicron_{\ssrm}(\leftidx{^t}X_{d+1}(2^t)).
\end{equation}
Now let $S=\leftidx{^t}X_d(2^t)$ with $d\in\{54,\ldots,90\}$. Then by Formula (\ref{omicronssEq}),
\begin{equation}\label{omicronssEq2}
\omicron(S)\leqslant\omicron_{\ssrm}(S)\cdot(1+\lceil\log_2(h(X_d))\rceil)\leqslant\omicron_{\ssrm}(\leftidx{^t}X_{90}(2^t))\cdot(1+\lceil\log_2(h(X_d))\rceil).
\end{equation}
Define $\omicron_0(S)$ to be the upper bound on $\omicron(S)$ from Formula (\ref{omicronssEq2}). In order to make use of it, one needs upper bounds on $\omicron_{\ssrm}(\leftidx{^t}X_{90}(2^t))$ for each of the six Lie classes of finite simple classical groups. The authors computed such bounds using GAP \cite{GAP4}; let us briefly explain how this was done.

Consider the following case-dependent notion of an admissible partition or partition pair:
\begin{itemize}
\item For $\leftidx{^t}X\in\{A,\leftidx{^2}A\}$, a partition $\lambda$ is called \emph{$\leftidx{^t}X$-admissible} if and only if $\lambda$ is reduced.
\item For $\leftidx{^t}X\in\{B,C\}$, a partition pair $\bm{\lambda}=(\lambda_+,\lambda_-)$ is called \emph{$\leftidx{^t}X$-admissible} if and only if both entries of $\bm{\lambda}$ are reduced.
\item For $t\in\{1,2\}$, a partition pair $\bm{\lambda}=(\lambda_+,\lambda_-)$ is called \emph{$\leftidx{^t}D$-admissible} if and only if $\lambda_+$ is reduced, the number of parts of $\lambda_-$ is congruent to $t+1$ modulo $2$, and $\lambda_-$ is \enquote{almost reduced} in the sense that if some positive integer $k>1$ occurs $m>1$ times as a part of $\lambda_-$, then $(k,m)=(2,2)$.
\end{itemize}
Then every admissible partition or partition pair with the \enquote{right} sum of parts corresponds to a conjugacy class of maximal tori of $\Inndiag(S)$, and for every maximal torus $T_1$ of $\Inndiag(S)$, there is an admissible partition or partition pair with corresponding maximal torus $T_2$ such that $\Exp(T_1\cap S)$ divides $\Exp(T_2\cap S)$. Therefore, $\omicron_{\ssrm}(S)$ is bounded from above by the sum of the numbers of divisors of the exponents of the groups $T\cap S$ where $T$ ranges over a set of representatives for the conjugacy classes of maximal tori of $\Inndiag(S)$ corresponding to admissible partitions or partition pairs. It is this sum which we computed for each of the five (taking $B_{90}(2)\cong C_{90}(2)$ into account) rank $90$ classical groups using GAP, and we list the computed values in the following table:

\begin{center}
\begin{longtable}[H]{|c|c|}
\caption{Upper bounds for rank $90$ groups.}
\label{ssBoundsTable}\\
\hline
$\leftidx{^t}X$ & upper bound on $\omicron_{\ssrm}(\leftidx{^t}X_{90}(2^t))$ \\ \hline
$A$ & 4235078858 \\ \hline
$\leftidx{^2}A$ & 3178257722 \\ \hline
$B,C$ & 22293229392 \\ \hline
$D$ & 15931588348 \\ \hline
$\leftidx{^2}D$ & 12297818620 \\ \hline
\end{longtable}

\end{center}

Observe that by \cite[Theorem 1.1(1)]{FG12a}, we have
\begin{equation}\label{ksEq}
\k(S)\geqslant\frac{\k(\Inndiag(S))}{|\Inndiag(S):S|}\geqslant\frac{2^d}{|\Inndiag(S):S|}=\begin{cases}2^d, & \text{if }\leftidx{^t}X\not=\leftidx{^2}A, \\ \frac{2^d}{\gcd(3,d+1)}, & \text{if }\leftidx{^t}X=\leftidx{^2}A.\end{cases}
\end{equation}
Set $\omega_0(S):=\lceil\k_0(S)/|\Out(S)|\rceil$ where $\k_0(S)$ is the lower bound on $\k(S)$ from Formula (\ref{ksEq}). Then one can check that
\[
\epsilon_{\q}(S)\geqslant\frac{\log\log{(\omega_0(S)/\omicron_0(S)+3)}}{\log\log{|S|}}>\epsilon_{\q}(\M),
\]
as required.

\item \emph{$\epsilon_{\q}(S)>\epsilon_{\q}(\M)$ for all classical finite simple groups of Lie type $S$ except possibly those from an explicit, finite list (given below).} Note that by case (3), we only need to consider classical groups $S$ of untwisted Lie rank at most $53$. The goal is to explicitly determine, for each $d=1,\ldots,53$, a prime power $q_0(d)$ such that for all prime powers $q\geqslant q_0(d)$, we have $\epsilon_{\q}(\leftidx{^t}X_d(q^t))>\epsilon_{\q}(\M)$, and then use a computer to sort out most of the remaining finitely many groups in order to arrive at a short, explicit list of potential exceptions that will need to be checked with more careful methods.

Let us start with $d=1$, which is discussed separately because the method used for $d\geqslant2$ would produce too large a value for $q_0(1)$. So we need to consider $S=A_1(q)=\PSL_2(q)$. By \cite[Table 4, p.~43]{Mac81a},
\[
\k(\PSL_2(q))=\begin{cases}q+1, & \text{if }2\mid q, \\ \frac{q+5}{2}, & \text{if }2\nmid q,\end{cases}
\]
and so in any case, $\k(\PSL_2(q))\geqslant\frac{q+1}{2}$. Moreover,
\[
|\Out(\PSL_2(q))|=\gcd(2,q-1)\cdot f,
\]
so that
\[
\omega(\PSL_2(q))\geqslant\frac{q+1}{4f}\geqslant\frac{q+1}{4\log_2(q)}.
\]
Finally, by \cite[Theorem 2.1]{BG07a}, the maximal tori of $\PSL_2(q)$ are cyclic of orders $\frac{q\pm1}{\gcd(2,q-1)}$, and since
\[
\left\{\overline{\begin{pmatrix}1 & a \\ 0 & 1\end{pmatrix}} \mid a\in\IF_q\right\}
\]
is a Sylow $p$-subgroup of $\PSL_2(q)$ which is of exponent $p$ and equal to the centraliser in $\PSL_2(q)$ of any of its nontrivial elements, the only non-semisimple element order in $\PSL_2(q)$ is $p$. It follows that $\Ord(\PSL_2(q))$ is just the set of positive integers dividing at least one of the numbers $\frac{q+1}{2}$, $\frac{q-1}{2}$ or $p$, so that
\[
\omicron(\PSL_2(q))\leqslant 2+2\sqrt{\frac{q+1}{2}}-1+2\sqrt{\frac{q-1}{2}}-1=2(\sqrt{\frac{q+1}{2}}+\sqrt{\frac{q-1}{2}}).
\]
It follows that
\[
\epsilon_{\q}(\PSL_2(q))\geqslant\frac{\log\log{(\frac{q+1}{8\log_2(q)\cdot(\sqrt{(q+1)/2}+\sqrt{(q-1)/2})})}}{\log\log{(q(q^2-1))}},
\]
which can be checked to be larger than $\epsilon_{\q}(\M)$ for $q\geqslant1100543=:q_0(1)$, as recorded in Table \ref{q0Table}.

Let us now discuss how to handle $d=2,\ldots,53$. Set
\[
c:=c(d):=\begin{cases}2, & \text{if }d\not=4, \\ 6, & \text{if }d=4.\end{cases}
\]
Then
\[
\omega(S)\geqslant\frac{q^d}{\min\{d+1,q+1\}^2\cdot cf}\geqslant\frac{q^d}{c\log_2(q)\cdot\min\{d+1,q+1\}^2}
\]
and
\[
\omicron(S)\leqslant g_2(d)\cdot 2(q+1)^{d/2}\cdot(1+\lceil\log_2(2d)\rceil).
\]
It follows that
\[
\epsilon_{\q}(S)\geqslant\frac{\log\log{\frac{q^d}{2c\cdot\log_2(q)\cdot\min\{d+1,q+1\}^2\cdot g_2(d)(1+\lceil\log_2(2d)\rceil)\cdot(q+1)^{d/2}}}}{\log\log{q^{4d^2}}},
\]
which can be checked to be greater than $\epsilon_{\q}(\M)$ for all $q\geqslant q_0(d)$ with $q_0(d)$ as in the following table:

\begin{center}
\begin{longtable}[H]{|c|c||c|c|}
\caption{Values of $q_0(d)$.}
\label{q0Table}\\
\hline
$d$ & $q_0(d)$ & $d$ & $q_0(d)$ \\ \hline
$1$ & $1100543$ & $13$ & $25$ \\ \hline
$2$ & $62753$ & $14$ & $23$ \\ \hline
$3$ & $4903$ & $15$ & $19$ \\ \hline
$4$ & $1801$ & $16,17$ & $16$ \\ \hline
$5$ & $401$ & $18$ & $13$ \\ \hline
$6$ & $197$ & $19,20,21$ & $11$ \\ \hline
$7$ & $121$ & $22$ & $9$ \\ \hline
$8$ & $79$ & $23,24,25$ & $8$ \\ \hline
$9$ & $59$ & $26,\ldots,34$ & $7$ \\ \hline
$10$ & $43$ & $35,\ldots,45$ & $5$ \\ \hline
$11$ & $37$ & $46,\ldots,53$ & $4$ \\ \hline
$12$ & $29$ & & \\ \hline
\end{longtable}

\end{center}

So, at this point, we are down to an explicit, finite list of potential exceptions $S$ to $\epsilon_{\q}(S)>\epsilon_{\q}(\M)$. We can still reduce this list further by checking, for each $S=\leftidx{^t}X_d(q^t)$ with $d\in\{1,\ldots,53\}$ and $q<q_0(d)$, whether certain sharper lower bounds on $\epsilon_{\q}(S)$ are larger than $\epsilon_{\q}(\M)$. More precisely,
\begin{itemize}
\item with $T$ ranging over a complete set of representatives of the conjugacy classes of maximal tori of $\Inndiag(S)$ corresponding to an admissible partition or partition pair as defined in the argument for the previous case, (3),
\begin{itemize}
\item observe that $\omicron_{\ssrm}(S)$ is just the cardinality of the set of positive integers dividing one of the numbers $\Exp(T\cap S)$, and
\item let $\overline{\omicron_{\ssrm}}(S)$ denote the sum of the numbers $\tau(\Exp(T\cap S))$, where $\tau(k)$ denotes the number of divisors of $k$.
\end{itemize}
Then $\omicron_{\ssrm}(S)\leqslant\overline{\omicron_{\ssrm}}(S)$. Moreover, set
\[
\overline{\omicron}(S):=\omicron_{\ssrm}(S)\cdot(1+\lceil\log_p(h(X_d))\rceil)
\]
and
\[
\overline{\overline{\omicron}}(S):=\overline{\omicron_{\ssrm}}(S)\cdot(1+\lceil\log_p(h(X_d))\rceil).
\]
We have $\omicron(S)\leqslant\overline{\omicron}(S)\leqslant\overline{\overline{\omicron}}(S)$.
\item consider the following lower bounds $\underline{\k}(S)$ and $\underline{\underline{\k}}(S)$ on $\k(S)$, which satisfy $\k(S)\geqslant\max\{\underline{\k}(S),\underline{\underline{\k}}(S)\}$:
\begin{itemize}
\item Assume that at least one of the following holds:
\begin{itemize}
\item $\leftidx{^t}X\in\{A,\leftidx{^2}A,B\}$;
\item $\leftidx{^t}X\in\{D,\leftidx{^2}D\}$ and $S$ is its own Schur cover;
\item $p=2$.
\end{itemize}
Then set $\underline{\k}(S):=\k(S)$, to be computed according to \cite[Formulas (5.2) and (6.13)]{Mac81a} for $\leftidx{^t}X\in\{A,\leftidx{^2}A\}$, or \cite[Theorems 3.19(1), 3.13(1), 3.16(1,2) and 3.22(1,2)]{FG12a} for $\leftidx{^t}X\in\{B,C,D,\leftidx{^2}D\}$.
\item If $\leftidx{^t}X=C$ and $p>2$, set $\underline{\k}(S):=\lceil\frac{\k(\Sp_{2d}(q))}{2}\rceil$, to be computed according to \cite[Subsection 2.6, Case (B), statement (iii), p.~36]{Wal63a}.
\item If $\leftidx{^t}X\in\{D,\leftidx{^2}D\}$, $p>2$ and the Schur cover $\tilde{S}=\Omega^{\pm}_{2d}(q)$ of $S$ has nontrivial centre, set $\underline{\k}(S):=\lceil\frac{\k(\tilde{S})}{2}\rceil$, to be computed according to \cite[Theorem 3.18(1)]{FG12a}.
\item Set
\[
\underline{\underline{\k}}(S):=
\begin{cases}
\underline{\k}(S)=\k(S), & \text{if }\leftidx{^t}X\in\{A,\leftidx{^2}A\}, \\
\lceil\frac{q^d}{\gcd(2,q-1)}\rceil, & \text{if }\leftidx{^t}X\in\{B,C\}, \\
\lceil\frac{q^d}{\gcd(2,q-1)^2}\rceil, & \text{if }\leftidx{^t}X\in\{D,\leftidx{^2}D\}.
\end{cases}
\]
\end{itemize}
The fact that $\k(S)\geqslant\max\{\underline{\k}(S),\underline{\underline{\k}}(S)\}$ follows from the above definitions and the bound $\k(\Inndiag(S))\geqslant q^d$, see \cite[Theorem 1.1(1)]{FG12a}. Using these lower bounds on $\k(S)$, set
\[
\underline{\omega}(S):=\lceil\underline{\k}(S)/|\Out(S)|\rceil
\]
and
\[
\underline{\underline{\omega}}(S):=\lceil\underline{\underline{\k}}(S)/|\Out(S)|\rceil,
\]
so that $\omega(S)\geqslant\max\{\underline{\omega}(S),\underline{\underline{\omega}}(S)\}$.
\end{itemize}

Note that, while $\overline{\omicron}(S)$ is a better upper bound on $\omicron(S)$ than $\overline{\overline{\omicron}}(S)$, it takes longer to compute, and similarly, $\underline{\omega}(S)$ (which seems to be larger than $\underline{\underline{\omega}}(S)$ by empirical evidence) takes longer to compute than $\underline{\underline{\omega}}(S)$. In order to minimise the computation time for our checks, we proceed as follows:

For each pair $(d,q)$ with $d\in\{1,\ldots,53\}$ and $q$ being a prime power less than $q_0(d)$, we go through the finite simple classical Lie type groups $S=\leftidx{^t}X_d(q^t)$, and, for each of them, we check first whether
\[
\frac{\log\log{(\underline{\underline{\omega}}(S)/\overline{\overline{\omicron}}(S)+3)}}{\log\log{|S|}}>\epsilon_{\q}(\M).
\]
If this fails, we check if
\[
\frac{\log\log{(\underline{\omega}(S)/\overline{\overline{\omicron}}(S)+3)}}{\log\log{|S|}}>\epsilon_{\q}(\M),
\]
and if this also fails, we check whether
\[
\frac{\log\log{(\underline{\omega}(S)/\overline{\omicron}(S)+3)}}{\log\log{|S|}}>\epsilon_{\q}(\M).
\]
If this third check also fails and $d>1$, we take note of $S$, progressively building a (hopefully short) list of exceptions that require further inspection. When $d=1$ (and so $S=A_1(q)=\PSL_2(q)$), we know that $\omicron(S)=\omicron_{\ssrm}(S)+1$ and perform one more check, namely whether
\[
\frac{\log\log{(\underline{\omega}(S)/(\omicron_{\ssrm}(S)+1)+3)}}{\log\log{|S|}}>\epsilon_{\q}(\M),
\]
and only if $S$ fails this check as well do we add it to the list. We note that the GAP algorithm written by the authors to perform these checks prints a warning several times, stating that the calculations are carried out under the assumption that $37644053098601$ is a prime. This, however is not a problem, as one can actually check with GAP's built-in primality testing algorithm that this number is indeed a prime (GAP's primality testing algorithm is deterministic for inputs less than $10^{18}$, see \cite[Subsection 14.4-2: {\tt IsPrimeInt}]{GAPM}).

It turns out that only the following $68$ groups fail each of these refined checks:
\begin{itemize}
\item $A_d(q)$ with $(d,q)$ from the set
\begin{align*}
&\{(2,2),(2,4),(2,7),(2,8),(2,9),(2,13),(2,16),(2,19),(2,25),(2,49),(2,64), \\
&(3,2),(3,3),(3,4),(3,5),(3,7),(3,9),(3,13),(4,2),(4,3),(4,4),(4,16),(5,2), \\
&(5,3),(5,4),(6,2)\}.
\end{align*}
\item $\leftidx{^2}A_d(q^2)$ with $(d,q)$ from the set
\begin{align*}
&\{(2,2),(2,4),(2,5),(2,8),(2,11),(2,17),(2,23),(2,29),(2,32),(3,3),(3,4), \\
&(3,5),(3,7),(3,9),(3,11),(4,4),(4,9),(5,2),(5,3),(5,5),(7,3),(8,2)\}.
\end{align*}
\item $B_d(q)$ with $(d,q)$ from the set
\[
\{(2,4),(2,8),(2,9),(3,3),(3,4),(3,5)\}.
\]
\item None of the groups $C_d(q)$ with $d\geqslant3$ except for $C_3(4)\cong B_3(4)$.
\item $D_d(q)$ with $(d,q)$ from the set
\[
\{(4,2),(4,3),(4,4),(4,5),(4,7),(4,9),(5,2),(5,3),(5,5),(6,3),(7,3)\}.
\]
\item $\leftidx{^2}D_d(q^2)$ with $(d,q)$ from the set
\[
\{(4,2),(4,3),(5,3)\}.
\]
\end{itemize}

\item \emph{$\epsilon_{\q}(S)>\epsilon_{\q}(\M)$ for all finite simple classical groups of Lie type $S$.} We need to deal with the $67$ remaining groups listed above. The authors implemented simple algorithms in GAP \cite{GAP4} for computing $\omicron(S)$ and $\omega(S)$ (and thus $\epsilon_{\q}(S)$) exactly. These algorithms require one to first compute the set of conjugacy classes of $S$, for which GAP has the built-in command {\tt ConjugacyClasses}. This allows one to deal with $46$ out of the $68$ groups, as listed in the following table:
\begin{center}
\begin{longtable}[H]{|c|c|c|c||c|c|c|c|}
\caption{Exact computation of $\epsilon_{\q}(S)$.}
\label{exactComputTable}\\
\hline
$S$ & $\omega(S)$ & $\omicron(S)$ & $\epsilon_{\q}(S)\approx$ & $S$ & $\omega(S)$ & $\omicron(S)$ & $\epsilon_{\q}(S)\approx$ \\ \hline
$A_2(2)$ & $5$ & $5$ & $0.199907$ & $\leftidx{^2}A_2(2^2)$ & $4$ & $4$ & $0.224773$ \\ \hline
$A_2(4)$ & $6$ & $6$ & $0.142406$ & $\leftidx{^2}A_2(4^2)$ & $9$ & $8$ & $0.145146$ \\ \hline
$A_2(7)$ & $15$ & $10$ & $0.152856$ & $\leftidx{^2}A_2(5^2)$ & $10$ & $9$ & $0.140543$ \\ \hline
$A_2(8)$ & $17$ & $10$ & $0.155376$ & $\leftidx{^2}A_2(8^2)$ & $10$ & $9$ & $0.126245$ \\ \hline
$A_2(9)$ & $32$ & $17$ & $0.160853$ & $\leftidx{^2}A_2(11^2)$ & $30$ & $15$ & $0.164402$ \\ \hline
$A_2(13)$ & $39$ & $15$ & $0.183387$ & $\leftidx{^2}A_2(17^2)$ & $62$ & $21$ & $0.188453$ \\ \hline
$A_2(16)$ & $20$ & $12$ & $0.141745$ & $\leftidx{^2}A_3(3^2)$ & $14$ & $10$ & $0.145173$ \\ \hline
$A_2(19)$ & $75$ & $23$ & $0.19498$ & $\leftidx{^2}A_3(4^2)$ & $35$ & $14$ & $0.175921$ \\ \hline
$A_2(25)$ & $72$ & $21$ & $0.193765$ & $\leftidx{^2}A_3(5^2)$ & $64$ & $21$ & $0.186343$ \\ \hline
$A_2(49)$ & $237$ & $31$ & $0.253006$ & $\leftidx{^2}A_3(7^2)$ & $76$ & $23$ & $0.18362$ \\ \hline
$A_3(2)$ & $12$ & $8$ & $0.177958$ & $\leftidx{^2}A_4(4^2)$ & $34$ & $19$ & $0.129955$ \\ \hline
$A_3(3)$ & $21$ & $12$ & $0.161363$ & $\leftidx{^2}A_5(2^2)$ & $44$ & $18$ & $0.167802$ \\ \hline
$A_3(4)$ & $36$ & $16$ & $0.16688$ & $B_2(2)$ & $12$ & $9$ & $0.145857$ \\ \hline
$A_3(5)$ & $34$ & $16$ & $0.157277$ & $B_2(8)$ & $21$ & $14$ & $0.134539$ \\ \hline
$A_3(7)$ & $137$ & $30$ & $0.210503$ & $B_2(9)$ & $41$ & $16$ & $0.176644$ \\ \hline
$A_3(9)$ & $85$ & $26$ & $0.175965$ & $B_3(3)$ & $52$ & $16$ & $0.195259$ \\ \hline
$A_3(13)$ & $358$ & $36$ & $0.260237$ & $B_3(4)$ & $75$ & $22$ & $0.183849$ \\ \hline
$A_4(2)$ & $20$ & $13$ & $0.14886$ & $B_3(5)$ & $136$ & $27$ & $0.209901$ \\ \hline
$A_4(3)$ & $72$ & $24$ & $0.178591$ & $D_4(2)$ & $27$ & $12$ & $0.181326$ \\ \hline
$A_4(4)$ & $110$ & $32$ & $0.177528$ & $D_4(3)$ & $38$ & $16$ & $0.161802$ \\ \hline
$A_5(2)$ & $44$ & $18$ & $0.166564$ & $D_5(2)$ & $84$ & $24$ & $0.189461$ \\ \hline
$A_5(4)$ & $169$ & $40$ & $0.176772$ & $\leftidx{^2}D_4(2^2)$ & $39$ & $15$ & $0.1844$ \\ \hline
$A_6(2)$ & $77$ & $27$ & $0.163159$ & $\leftidx{^2}D_4(3^2)$ & $100$ & $25$ & $0.195833$ \\ \hline
\end{longtable}

\end{center}

For $S=D_4(4)$, one can compute the set of conjugacy classes of $S$ with GAP and use this to determine $\omicron(S)$ exactly as well as to provide a certain lower bound on $\omega(S)$. This lower bound, the derivation of which is explained in detail after Table \ref{conjClassTable}, is larger than $\underline{\omega}(S)$ and (unlike $\underline{\omega}(S)$) is sufficient to conclude that $\epsilon_{\q}(S)>\epsilon_{\q}(\M)$:

\begin{center}
\begin{longtable}[H]{|c|c|c|c|}
\caption{A conjugacy class argument.}
\label{conjClassTable}\\
\hline
$S$ & $\omega(S)\geqslant$ & $\omicron(S)$ & $\epsilon_{\q}(S)\geqslant$ \\ \hline
$\D_4(4)$ & $51$ & $23$ & $0.143317$ \\ \hline
\end{longtable}

\end{center}

Indeed, consider the set $\C$ of conjugacy classes of $S$. View $\C$ as a multiset, and consider the multiset obtained from it by replacing each element $C\in\C$ by the order of a representative of $C$. The following is the said multiset (we write $x_n$ shorthand for $n$ copies of $x$):
\begin{align*}
&\{1_1,2_5,3_5,4_6,5_{19},6_{14},7_1,8_2,9_3,10_{43},12_7,13_2,15_{43},17_{24},20_{18},21_8,30_{42},34_{12}, \\
&51_{12},63_{18},65_{48},85_{24},255_{48}\}.
\end{align*}
In particular, $\omicron(S)=23$, and since $|\Out(S)|=12$, for each element order $o$ in $S$, $\omega_o(S)$ is bounded from below by the shortest length of an integer partition of $o$ into divisors of $12$. Hence
\begin{itemize}
\item $\omega_o(S)\geqslant2$ for $o\in\{2,3,6,12,17,20,21,63,85\}$,
\item $\omega_5(S)\geqslant3$,
\item $\omega_o(S)\geqslant4$ for $o\in\{30,65,255\}$, and
\item $\omega_o(S)\geqslant5$ for $o\in\{10,15\}$.
\end{itemize}
It follows that $\omega(S)\geqslant\omicron(S)+28=51$, as asserted.

In order to deal with the remaining $20$ possibilities for $S$, which are all of Lie type $A$, $\leftidx{^2}A$, $D$ or $\leftidx{^2}D$ (and those of types $D$ or $\leftidx{^2}D$ all have odd defining characteristic), it will be useful to have an algorithm for computing an upper bound on $\omicron(S)$, and it is our next goal to describe such an algorithm.

Let $S=\leftidx{^t}X_d(p^{ft})$. Recall from the beginning of this subsection that by \cite[Corollary 0.5]{Tes95a}, the largest power of $p$ dividing $\Exp(S)$ is $p^{\lceil\log_p(h(X_d))\rceil}$ where $h(X_d)$ is the Coxeter number of the root system $X_d$. Recall also from earlier in this proof that we already described and used an algorithm for computing $\omicron_{\ssrm}(S)$, the number of semisimple element orders in $S$. This latter algorithm essentially consists of looping over certain conjugacy classes of maximal tori $T$ of $\Inndiag(S)$ and joining the divisor sets $\Div(\Exp(T\cap S))$, based on the formulas in \cite{BG07a}. Now, for $e=0,\ldots,\lceil\log_p(h(X_d))\rceil$, denote by $\omicron_{p,e}(S)$ the number of element orders $o$ in $S$ such that the largest power of $p$ dividing $o$ is $p^e$. Hence, $\omicron_{p,0}(S)$ is just $\omicron_{\ssrm}(S)$, $\omicron_{p,1}(S)$ is the number of element orders in $S$ that are sharply divisible by $p$, and so on. Our algorithm for computing an upper bound on $\omicron(S)$ proceeds by computing a certain upper bound $\overline{\omicron_{p,e}}(S)$ on $\omicron_{p,e}(S)$ for each $e=0,\ldots,\lceil\log_p(h(X_d))\rceil$ and then adding those upper bounds.

We set $\overline{\omicron_{p,0}}(S):=\omicron_{\ssrm}(S)$, to be computed as described above. Hence we assume that $e\geqslant1$. The following theoretical results are the basis of our argument:

\begin{enumerate}
\item Let $S=A_d(q)=\PSL_{d+1}(q)$, and let $o\in\Ord(S)$ with $p^e\mid\mid o$, for some given $e\in\IN^+$ (in particular, $p^{e-1}+1\leqslant h(A_d)=d+1$). Then:
\begin{enumerate}
\item There is an element $g\in\GL_{d+1}(q)$ such that $o\mid\ord(g)$ and the following hold:
\begin{itemize}
\item the rational canonical form of $g$ has exactly one non-semisimple block, which is of the form $\Comp((X-a)^{p^{e-1}+1})$ for some $a\in\IF_q^{\ast}$; and
\item the semisimple blocks of the rational canonical form of $g$ form the rational canonical form of a (semisimple) element of $\GL_{d-p^{e-1}}(q)$.
\end{itemize}
In particular, $\frac{o}{p^e}$ divides a number of the form $\lcm(q-1,o')$ for some $o'\in\Ord_{\ssrm}(\GL_{d-p^{e-1}}(q))$.
\item If $p^{e-1}+1=d+1=h(A_d)$, then $o=p^e$.
\item If $p^{e-1}+1=d=h(A_d)-1$, then $\frac{o}{p^e}$ divides $\frac{q-1}{\gcd(d+1,q-1)}$.
\end{enumerate}
\item Let $S=\leftidx{^2}A_d(q^2)=\PSU_{d+1}(q)$, and let $o\in\Ord(S)$ with $p^e\mid\mid o$, for some given $e\in\IN^+$ (in particular, $p^{e-1}+1\leqslant h(A_d)=d+1$). Then:
\begin{enumerate}
\item There is an element $g\in\GU_{d+1}(q)\leqslant\GL_{d+1}(q^2)$ such that $o\mid\ord(g)$ and the following hold:
\begin{itemize}
\item the rational canonical form of $g$ has exactly one non-semisimple block, which is of the form $\Comp((X-a)^{p^{e-1}+1})$ for some $a\in\IF_{q^2}^{\ast}$ with $\ord(a)\mid q+1$; and
\item the semisimple blocks of the rational canonical form of $g$ form the rational canonical form of a (semisimple) element of $\GU_{d-p^{e-1}}(q)$.
\end{itemize}
In particular, $\frac{o}{p^e}$ divides a number of the form $\lcm(q+1,o')$ for some $o'\in\Ord_{\ssrm}(\GU_{d-p^{e-1}}(q))$.
\item If $p^{e-1}+1=d+1=h(A_d)$, then $o=p^e$.
\item If $p^{e-1}+1=d=h(A_d)-1$, then $\frac{o}{p^e}$ divides $\frac{q+1}{\gcd(d+1,q+1)}$.
\end{enumerate}
\item Let $S=\POmega_{2d}^{\epsilon}(q)$ with $\epsilon\in\{+,-\}$ and $q$ odd, and let $o\in\Ord(S)$ with $p^e\mid\mid o$, for some given $e\in\IN^+$ (in particular, $p^{e-1}+1\leqslant h(D_d)=2d-2$). Then:
\begin{enumerate}
\item There is an element $g\in\GO_{2d}^{\epsilon}(q)$ such that $o\mid\ord(g)$ and the following hold:
\begin{itemize}
\item the rational canonical form of $g$ has exactly one non-semisimple block, which is of one of the two forms $\Comp((X+1)^{p^{e-1}+2})$ or $\Comp(P(X)^{p^{e-1}+1})$ for some monic quadratic irreducible polynomial $P(X)\in\IF_q[X]$ such that $\ord(P(X))\mid q+1$;
\item if the unique non-semisimple block of $g$ is of the form $\Comp((X+1)^{p^{e-1}+2})$, then the semisimple blocks of the rational canonical form of $g$ form the rational canonical form of a (semisimple) element of $\GO_{2d-p^{e-1}-2}(q)$; and
\item if the unique non-semisimple block of $g$ is of the form $\Comp(P(X)^{p^{e-1}+1})$, then the semisimple blocks of the rational canonical form of $g$ form the rational canonical form of a (semisimple) element of $\GO_{2d-2(p^{e-1}+1)}^{\epsilon}(q)$.
\end{itemize}
In particular, $\frac{o}{p^e}$ divides a number of one of the two forms $\lcm(2,o')$ for some $o'\in\Ord_{\ssrm}(\GO_{2d-p^{e-1}-2}(q))$, or $\lcm(q+1,o'')$ for some $o''\in\Ord_{\ssrm}(\GO_{2d-2(p^{e-1}+1)}^{\epsilon}(q))$.
\end{enumerate}
\item If $p^{e-1}+1=2d-2=h(D_d)$, then $o=p^e$.
\end{enumerate}

These results can be deduced from the classification of rational canonical forms of a given finite classical isometry group due to Wall \cite[Case (A), p.~34; Case (C), pp.~38f.]{Wal63a}. We only exemplarily prove the result for $S=\leftidx{^2}A_d(q^2)=\PSU_{d+1}(q)$ and leave the remaining two cases as exercises to the inclined reader.

Fix an element $s\in S$ with $\ord(s)=o$, and let $\tilde{s}\in\SU_{d+1}(q)$ be a lift of $s$. By \cite[Case (A), p.~34]{Wal63a}, the rational canonical form of $\tilde{s}$ has the property that its multiset of (Frobenius) blocks is closed under the involutory operation $\Comp(P(X)^k)\mapsto\Comp(\tilde{P}(X)^k)$, where $\tilde{P}(X)$ is the minimal polynomial over $\IF_{q^2}$ of $\xi^{-q}$, for any root $\xi\in\overline{\IF_{q^2}}$ of $P(X)$. Since $p$ does not divide the centre order $|\zeta\SU_{d+1}(q)|$, we also have $p^e\mid\mid\ord(\tilde{s})$, and so all Frobenius blocks $\Comp(P(X)^k)$ of $\tilde{s}$ have the property that $k\leqslant p^e$, and there is at least one block with $k\geqslant p^{e-1}+1$. Choose one copy of a Frobenius block $\Comp(P(X)^k)$ of $\tilde{s}$ for some $P(X)$ with $k\geqslant p^{e-1}+1$, mark it, and if $P(X)\not=\tilde{P}(X)$, additionally mark one copy of $\Comp(\tilde{P}(X)^k)$. Now apply the following transformations, in the given order, to the rational canonical form of $\tilde{s}$:
\begin{itemize}
\item Replace each unmarked block $\Comp(Q(X)^{\ell})$ by one copy of $\Comp(Q(X))$ and $\deg(Q(X))\cdot(\ell-1)$ copies of the trivial block $I_1$.
\item Replace each of the (at most two) marked blocks $\Comp(R(X)^k)$ by one marked copy of $\Comp(R(X)^{p^{e-1}+1})$ and $\deg{R(X)}\cdot (k-p^{e-1}-1)$ unmarked copies of the trivial block $I_1$.
\item If there are now two distinct marked blocks $\Comp(P(X)^{p^{e-1}+1})$ and $\Comp(\tilde{P}(X)^{p^{e-1}+1})$, then replace them by one unmarked copy of each of $\Comp(P(X))$ and $\Comp(\tilde{P}(X))$, as well as one unmarked copy of $\Comp((X-1)^{p^{e-1}+1})$ and $2\deg(P(X))\cdot(p^{e-1}+1)-2\deg(P(X))-(p^{e-1}+1)$ unmarked copies of the trivial block $I_1$. If, on the other hand, there is exactly one marked block, namely $\Comp(P(X)^{p^{e-1}+1})$, then proceed as follows: If $P(X)$ is linear, just remove the mark and leave the block itself unchanged. If $P(X)$ is not linear, replace the block by one unmarked copy of $\Comp(P(X))$, one unmarked copy of $\Comp((X-1)^{p^{e-1}+1})$ and $\deg(P(X))\cdot(p^{e-1}+1)-\deg(P(X))-(p^{e-1}+1)$ unmarked copies of the trivial block $I_1$.
\end{itemize}
These transformations result in a rational canonical form in $\GL_{d+1}(q^2)$ whose multiset of Frobenius blocks is still closed under the operation $\Comp(Q(X)^{\ell})\mapsto\Comp(\tilde{Q}(X)^{\ell})$, whence by \cite[Case (A), p.~34]{Wal63a}, the constructed rational canonical form is attained by some element $g\in\GU_{d+1}(q)$. Moreover, it is clear by construction that the rational canonical form of $g$ has exactly one non-semisimple block, which is of the form $\Comp(L(X)^{p^{e-1}+1})$ for some linear polynomial $L(X)\in\IF_{q^2}[X]$ with $L(X)=\tilde{L}(X)$, which forces $L(X)$ to be of the form $X-a$ with $\ord(a)$ divides $q+1$. The multiset of semisimple blocks of the rational canonical form of $g$ is closed under $\Comp(Q(X)^{\ell})\mapsto\Comp(\tilde{Q}(X)^{\ell})$ and thus by \cite[Case (A), p.~34]{Wal63a} forms the rational canonical form of an element of $\GU_{(d+1)-(p^{e-1}+1)}(q)=\GU_{d-p^{e-1}}(q)$. Finally, none of the three transformations described above which lead from the rational canonical form of $\tilde{s}$ to the one of $g$ change the order of the form, whence $o=\ord(s)$ divides $\ord(\tilde{s})=\ord(g)$, as required. This concludes the proof of statement (i) for $S=\leftidx{^2}A_d(q^2)$.

As for statement (ii), note that any lift $\tilde{s}\in\SU_{d+1}(q)$ of $s$ must have a Frobenius block of the form $\Comp(P(X)^k)$ with $k\geqslant p^{e-1}+1=d+1$. It follows that the rational canonical form of $\tilde{s}$ consists of exactly one Frobenius block, of the form $\Comp((X-a)^{d+1})$ with $a\in\IF_{q^2}^{\ast}$ of order dividing $q+1$. Hence $(\tilde{s}-aI_{d+1})^{d+1}=0$, and consequently (raising both sides to the power $p^e$) $(\tilde{s}^{p^e}-a^{p^e}I_{d+1})^{d+1}=0$. However, $\tilde{s}^{p^e}$ is semisimple, and so the minimal polynomial of $\tilde{s}^{p^e}$ divides $(X-a^{p^e})^{d+1}$. It follows that the minimal polynomial of $\tilde{s}^{p^e}$ is $X-a^{p^e}$, whence $\tilde{s}^{p^e}$ is the scalar matrix $a^{p^e}I_{d+1}$. Denoting by $\pi$ the canonical projection $\SU_{d+1}(q)\rightarrow\PSU_{d+1}(q)=S$, it follows that
\[
s^{p^e}=\pi(\tilde{s})^{p^e}=\pi(\tilde{s}^{p^e})=\pi(a^{p^e}I_{d+1})=1_S,
\]
whence $o=p^e$, as asserted.

Finally, as for statement (iii), let, again, $\tilde{s}\in\SU_{d+1}(q)$ be a lift of $s$. Then $\tilde{s}$ has a Frobenius block of the form $\Comp(P(X)^k)$ with $k\geqslant p^{e-1}+1=d$, so either $\tilde{s}$ is as in the arugment for statement (ii), which implies that $\ord(s)=p^e$, or $\tilde{s}$ has two Frobenius blocks, one of the form $\Comp((X-a)^{p^{e-1}+1})=\Comp((X-a)^d)$ for some $a\in\IF_{q^2}^{\ast}$ with $\ord(a)$ divides $q+1$, and the other block is $\Comp(X-a^{-d})$ (taking into account that $\det(\tilde{s})=1$). Similarly to the argument for statement (ii), we find that $\tilde{s}^{p^e}$ is similar to the diagonal matrix which has the eigenvalues $a^{p^e}$, with multiplicity $d$, and $a^{-dp^e}$, with multiplicity $1$. Modulo the scalars $\zeta\GU_{d+1}(q)$, the said diagonal matrix is congruent to the one which has eigenvalues $1$, with multiplicity $d$, and $a^{-(d+1)p^e}$. It follows that
\[
\frac{o}{p^e}=\ord(s^{p^e})\textrm{ divides }\ord(a^{-(d+1)p^e})\textrm{ divides }\frac{q+1}{\gcd(q+1,d+1)},
\]
as required. This concludes our proof of result (b), exemplary for the entire proof of results (a)--(c).

Now let $e\in\{1,\ldots,\lceil\log_p(h(X_d))\rceil\}$. By the above results (a)--(c), each of the following numbers $\overline{\omicron_{p,e}}(S)$ is an upper bound on $\omicron_{p,e}(S)$:
\begin{enumerate}
\item If $S=A_d(q)$, set
\[
\overline{\omicron_{p,e}}(S):=\begin{cases}1, & \text{if }p^{e-1}+1=d+1, \\ \tau(\frac{q-1}{\gcd(d+1,q-1)}), & \text{if }p^{e-1}+1=d, \\ |\bigcup_{o\in\Ord_{\ssrm}(\GL_{d-p^{e-1}}(q))}{\Div(\lcm(q-1,o))}|, & \text{otherwise.}\end{cases}
\]
\item If $S=\leftidx{^2}A_d(q^2)$, set
\[
\overline{\omicron_{p,e}}(S):=\begin{cases}1, & \text{if }p^{e-1}+1=d+1, \\ \tau(\frac{q+1}{\gcd(d+1,q+1)}), & \text{if }p^{e-1}+1=d, \\ |\bigcup_{o\in\Ord_{\ssrm}(\GU_{d-p^{e-1}}(q))}{\Div(\lcm(q+1,o))}|, & \text{otherwise.}\end{cases}
\]
\item If $S=\POmega^{\epsilon}_{2d}(q)$, set
\[
U_1:=\bigcup_{o\in\Ord_{\ssrm}(\GO_{2d-p^{e-1}-2}(q))}{\Div(\lcm(2,o))}
\]
and
\[
U_2:=\bigcup_{o\in\Ord_{\ssrm}(\GO_{2d-2(p^{e-1}+1)}^{\epsilon}(q))}{\Div(\lcm(q+1,o))}.
\]
Then set
\[
\overline{\omicron_{p,e}}(S):=\begin{cases}1, & \text{if }p^{e-1}+1=2d-2, \\ |U_1 \cup U_2|, & \text{otherwise.}\end{cases}
\]
\end{enumerate}

In order to compute these upper bounds $\overline{\omicron_{p,e}}(S)$, we need to be able to compute the set of semisimple element orders in the groups $\GL_n(q)$, $\GU_n(q)$, $\GO_n(q)$ for $n$ odd, and $\GO_n^{\epsilon}(q)$ for $n$ even. This is similar to (and actually easier than) the computation of $\omicron_{\ssrm}(S)$ following \cite{BG07a}, and it uses the following well-known facts:
\begin{enumerate}
\item The conjugacy classes of maximal tori of $\GL_n(q)$ are in a natural bijection with the ordered integer partitions $\lambda$ of $n$. Moreover, for any maximal torus $T\leqslant\GL_n(q)$ in the class corresponding to $\lambda=(\lambda_1,\ldots,\lambda_s)$, we have $\Exp(T)=\lcm(q^{\lambda_1}-1,\ldots,q^{\lambda_s}-1)$.
\item The conjugacy classes of maximal tori of $\GU_n(q)$ are in a natural bijection with the ordered integer partitions $\lambda$ of $n$. Moreover, for any maximal torus $T\geqslant\GU_n(q)$ in the class corresponding to $\lambda=(\lambda_1,\ldots,\lambda_s)$, we have $\Exp(T)=\lcm(q^{\lambda_1}-(-1)^{\lambda_1},\ldots,q^{\lambda_s}-(-1)^{\lambda_s})$.
\item Let $n\in\IN^+$ be odd. The conjugacy classes of maximal tori of $\GO_n(q)$ are in a natural bijection with the ordered pairs $(\lambda_+,\lambda_-)$ of ordered integer partitions such that the total sum of the parts of $\lambda_+$ and $\lambda_-$ is $\frac{n-1}{2}$. Moreover, for any maximal torus $T\leqslant\GO_n(q)$ in the class corresponding to $(\lambda_+,\lambda_-)$ with $\lambda_+=(\lambda^{(1)}_+,\ldots,\lambda^{(s_+)}_+)$ and $\lambda_-=(\lambda^{(1)}_-,\ldots,\lambda^{(s_-)}_-)$, we have $\Exp(T)=\lcm(q^{\lambda^{(1)}_+}-1,\ldots,q^{\lambda^{(s_+)}_+}-1,q^{\lambda^{(1)}_-}+1,\ldots,q^{\lambda^{(s_-)}_-}+1)$.
\item Let $n\in\IN^+$ be even. The conjugacy classes of maximal tori of $\GO_n^{\epsilon}(q)$, with $\epsilon\in\{+,-\}$, are in a natural bijection with the ordered pairs $(\lambda_+,\lambda_-)$ of ordered integer partitions such that the total sum of the parts of $\lambda_+$ and $\lambda_-$ is $\frac{n}{2}$ and the number of parts of $\lambda_-$ is even when $\epsilon=+$ and odd when $\epsilon=-$. Moreover, for any maximal torus $T\leqslant\GO_n^{\epsilon}(q)$ in the class corresponding to $(\lambda_+,\lambda_-)$ with $\lambda_+=(\lambda^{(1)}_+,\ldots,\lambda^{(s_+)}_+)$ and $\lambda_-=(\lambda^{(1)}_-,\ldots,\lambda^{(s_-)}_-)$, we have $\Exp(T)=\lcm(q^{\lambda^{(1)}_+}-1,\ldots,q^{\lambda^{(s_+)}_+}-1,q^{\lambda^{(1)}_-}+1,\ldots,q^{\lambda^{(s_-)}_-}+1)$.
\end{enumerate}
Recall the notion of a $\leftidx{^t}X$-admissible partition (pair) from case (3) earlier in this proof. The above results on exponents of maximal tori of $\GL_n(q)$ and $\GU_n(q)$ respectively imply that for $G\in\{\GL_n(q),\GU_n(q)\}$ and for each partition $\lambda$ of $n$, there is a reduced (i.e., admissible) partition $\mu$ of $n$ (namely $\mu=\overline{\lambda})$ such that the exponent of any maximal torus of $G$ in the class corresponding to $\lambda$ is equal to the exponent of any maximal torus of $G$ in the class corresponding to $\mu$. Hence $\Ord_{\ssrm}(G)$ can be computed by looping only over conjugacy classes of maximal tori corresponding to admissible (i.e., reduced) partitions of $n$ and joining the sets of divisors of the exponents of such maximal tori. Similarly, $\Ord_{\ssrm}(G)$  for $G$ an orthogonal group, can be computed by looping only over conjugacy classes of maximal tori corresponding to admissible partition pairs and joining the sets of divisors of the exponents of such maximal tori. This concludes our description of the algorithm for computing $\Ord_{\ssrm}(G)$.

We now start to apply this algorithm, which the authors have implemented in GAP \cite{GAP4}. For the $10$ groups listed in Table \ref{onlyOmicronTable}, the direct computation of $\omega(S)$ (and for some also of $\omicron(S)$) via a computation of the conjugacy classes of $S$ was either not possible at all or just too costly. However, combining either
\begin{itemize}
\item the exact value of $\omicron(S)$ (where it can be computed) or
\item the upper bound $\overline{\omicron}(S)$ on it as computed by the above described algorithm
\end{itemize}
with the lower bound $\underline{\omega}(S)$ on $\omega(S)$ from the end of case (4) above is enough for those $9$ groups $S$ to show that $\epsilon_{\q}(S)>\epsilon_{\q}(\M)$.
\begin{center}
\begin{longtable}[H]{|c|c|c|c|}
\caption{Computation of $\underline{\omega}(S)$ and of $\omicron(S)$ or $\overline{\omicron(S)}$ is sufficient.}
\label{onlyOmicronTable}\\
\hline
$S$ & $\underline{\omega}(S)$ & $\omicron(S)$ & $\epsilon_{\q}(S)\geqslant$ \\ \hline
$A_2(64)$ & $39$ & $=26$ & $0.11759$ \\ \hline
$A_4(16)$ & $350$ & $\leqslant86$ & $0.1607$ \\ \hline
$A_5(3)$ & $51$ & $=33$ & $0.114388$ \\ \hline
$\leftidx{^2}A_2(23^2)$ & $32$ & $\leqslant20$ & $0.13303$ \\ \hline
$\leftidx{^2}A_2(29^2)$ & $49$ & $\leqslant25$ & $0.14480$ \\ \hline
$\leftidx{^2}A_3(9^2)$ & $56$ & $\leqslant27$ & $0.13961$ \\ \hline
$\leftidx{^2}A_4(9^2)$ & $76$ & $\leqslant41$ & $0.11623$ \\ \hline
$\leftidx{^2}A_5(3^2)$ & $66$ & $\leqslant35$ & $0.12715$ \\ \hline
$D_5(5)$ & $166$ & $\leqslant81$ & $0.11635$ \\ \hline
$D_7(3)$ & $557$ & $\leqslant113$ & $0.16116$ \\ \hline
\end{longtable}
\end{center}

The group $S=\leftidx{^2}A_2(32^2)=\PSU_3(32)$ can be dealt with similarly; the lower bound $\underline{\omega}(S)=12$ fails us by $1$, but this can be easily fixed:

\begin{center}
\begin{longtable}[H]{|c|c|c|c|}
\caption{A group with similar treatment.}
\label{singleGroupTable}\\
\hline
$S$ & $\omega(S)\geqslant$ & $\omicron(S)$ & $\epsilon_{\q}(S)\geqslant$ \\ \hline
$\leftidx{^2}A_2(32^2)$ & $13$ & $\leqslant10$ & $0.11502$ \\ \hline
\end{longtable}
\end{center}

Indeed, in order to get $\omega(S)\geqslant13$, we note that $\k(S)=356$. This leads to the  lower bound
\[
\underline{\omega}(S)=\lceil\frac{356}{|\Out(S)|}\rceil=\lceil\frac{356}{30}\rceil=12,
\]
and we can improve this by $1$ by noting that the trivial conjugacy class of $S$ is certainly fixed under the action of $\Out(S)$, so that
\[
\omega(S)\geqslant 1+\lceil\frac{355}{30}\rceil=1+12=13,
\]
as asserted.

Only the $10$ groups $S$ listed in Table \ref{remainingTable} below remain to be dealt with. For each of them, we proceed as follows: First, we compute the upper bound $\overline{\omicron}(S)$ on $\omicron(S)$ listed in Table \ref{remainingTable} with our algorithm. Once this is done, we consider certain element orders $o\in\Ord(S)$, for which we provide nontrivial (i.e., greater than $1$) lower bounds on $\omega_o(S)$; essentially, we do so by specifying various distinct rational canonical forms of elements in the Schur cover of $S$ which project to order $o$ elements in $S$, but in light of graph-field automorphisms, the correspondence between rational canonical forms and $\Aut(S)$-orbits of order $o$ elements is not always injective, so in general, we need to divide the number of normal forms by a certain number to get a lower bound on $\omega_o(S)$. This is particularly cumbersome when $S=D_4(q)$, due to the large number of graph-field automorphisms, though we will be able to use \cite[Proposition 3.55(i)]{Bur07a} to at least get better bounds when $o=p$. In any case, this approach allows us to show that $\omega(S)\geqslant\omicron(S)+W(S)$ for a certain number $W(S)$, also listed in Table \ref{remainingTable}. We then conclude that $\q(S)=\frac{\omega(S)}{\omicron(S)}\geqslant\frac{\overline{\omicron}(S)+W(S)}{\overline{\omicron}(S)}$, and this will be sufficient to get that $\epsilon_{\q}(S)>\epsilon_{\q}(\M)$.

\begin{center}
\begin{longtable}[H]{|c|c|c|c|}
\caption{The remaining groups.}
\label{remainingTable}\\
\hline
$S$ & $W(S)$ & $\overline{\omicron}(S)$ & $\epsilon_{\q}(S)\geqslant$ \\ \hline
$\leftidx{^2}A_3(11^2)$ & $29$ & $38$ & $0.12567$ \\ \hline
$\leftidx{^2}A_5(5^2)$ & $65$ & $70$ & $0.11677$ \\ \hline
$\leftidx{^2}A_7(3^2)$ & $84$ & $76$ & $0.11593$ \\ \hline
$\leftidx{^2}A_8(2^2)$ & $48$ & $48$ & $0.11922$ \\ \hline
$D_4(5)$ & $17$ & $31$ & $0.11483$ \\ \hline
$D_4(7)$ & $35$ & $47$ & $0.11611$ \\ \hline
$D_4(9)$ & $38$ & $47$ & $0.11459$ \\ \hline
$D_5(3)$ & $30$ & $44$ & $0.1153$ \\ \hline
$D_6(3)$ & $69$ & $61$ & $0.11808$ \\ \hline
$\leftidx{^2}D_5(3^2)$ & $23$ & $29$ & $0.1161$ \\ \hline
\end{longtable}
\end{center}

As for the details of computing $W(S)$, we start by discussing the four remaining groups $S$ of type $\leftidx{^2}A$. Let $S=\leftidx{^2}A_d(q^2)=\PSU_{d+1}(q)$. The Schur cover of $S$ is $\SU_{d+1}(q)$, sitting inside $\GU_{d+1}(q)\leqslant\GL_{d+1}(q^2)$. The rational canonical forms in $\GL_{d+1}(q^2)$ that are attained by elements of $\GU_{d+1}(q)$ are characterised by Wall in \cite[Subsection 2.6, Case (a), p.~34]{Wal63a} to be those where the multiset of (Frobenius) blocks is closed under the involutory operation $\Comp(P(X)^k)\mapsto\Comp(\tilde{P}(X)^k)$, where $\tilde{P}(X)$ is the minimal polynomial over $\IF_{q^2}$ of $\xi^{-q}$, for any root $\xi\in\overline{\IF_{q^2}}$ of $P(X)$. Henceforth, we will refer to rational canonical forms in $\GL_{d+1}(q^2)$ satisfying Wall's criterion as \emph{admissible}.

It is easy to see that a monic irreducible polynomial $P(X)\in\IF_{q^2}[X]$ satisfies $P(X)=\tilde{P}(X)$ if and only if $\deg{P(X)}$ is odd and $\ord(P(X))$ divides $q^{\deg{P(X)}}+1$, in which case we call either of
\begin{itemize}
\item the positive $p'$-integer $\ord(P(X))$,
\item the polynomial $P(X)$, or
\item the Frobenius block $\Comp(P(X)^k)$
\end{itemize}
\emph{$\U$-economic}, or just \emph{economic} for short (the $\U$ stresses the fact that this notion of \enquote{economy} is specific for the treatment of the unitary case; for orthogonal groups, we will use a different notion of economic objects, see below). Positive $p'$-integers, monic irreducible polynomials in $\IF_{q^2}[X]$ or Frobenius blocks over $\IF_{q^2}$ which are not economic will be called \emph{uneconomic}. Note that these notions of (un)economic objects depend on $q$, which we view as fixed.

It is not difficult to show that a positive $p'$-integer $o$ is economic if and only if it satisfies either of the following, equivalent conditions:
\begin{itemize}
\item $o$ divides some number of the form $q^{2k+1}+1$ with $k\in\IN$.
\item The \emph{$q^2$-degree of $o$}, denoted by $\deg_{q^2}(o)$, which is defined as the smallest positive integer $t$ such that $o\mid q^{2t}-1$, is odd, and $o\mid q^{\deg_{q^2}(o)}+1$.
\end{itemize}

Now, given an element order $o$ in $S=\PSU_{d+1}(q)$, our goal is to specify a set $\F_o$ of order $o$ admissible rational canonical forms in $\GL_{d+1}(q^2)$ such that
\begin{enumerate}
\item all forms in $\F_o$ have determinant $1$ (so that they actually represent elements in $\SU_{d+1}(q)$, not just $\GU_{d+1}(q)$);
\item all forms in $\F_o$ have order $o$ \enquote{modulo the scalars $\zeta\GU_{d+1}(q)$} (i.e., for each form in $\F_o$, $o$ is the smallest positive integer $m$ such that the $m$-th power of the form is in $\zeta\GU_{d+1}(q)$); and
\item no two matrix similarity classes represented by distinct forms in $\F_o$ are fused under multiplication by scalars in $\zeta\GU_{d+1}(q)$.
\end{enumerate}

It is easy to check that each of these three properties, say (x), is implied by a certain other property, (x$'$), which we will now formulate:
\renewcommand{\labelenumii}{(\alph{enumii}$'$)}
\begin{enumerate}
\item $o$ is coprime to $\gcd(d+1,q+1)=|\zeta\GU_{d+1}(q)|$ (equivalently, every Frobenius block of every form in $\F_o$ has order coprime to $\gcd(d+1,q+1)$.
\item Either $o$ is coprime to $\gcd(d+1,q+1)$, or $p\mid o$ and each form in $\F_o$ has at least one unipotent block $\Comp((X-1)^k)$.
\item There is a divisor $o'$ of $o$ which is coprime to $\gcd(d+1,q+1)$ and such that each form in $\F_o$ has at least one block of order $o'$, but no form in $\F_o$ has a block of order a proper multiple of $o'$.
\end{enumerate}
\renewcommand{\labelenumii}{(\alph{enumii})}

Assume now that we have specified such a set $\F_o$ of rational canonical forms. Then different elements in $\F_o$ correspond to different conjugacy classes in $S$, and by property (c) above, the only fusion under the action of $\Aut(S)=\PGU_{d+1}(q)\rtimes\Phi_S$ of conjugacy classes in $S$ corresponding to different forms in $\F_o$ that can occur is under $\Phi_S$, the group of field automorphisms of $S$. In each of the four examples $S$ that we need to consider, $q=p$ is a prime, and thus $|\Phi_S|=2$. It follows that $\omega_o(S)\geqslant\lceil\frac{|\F_o|}{2}\rceil$, and we even get $\omega_o(S)\geqslant|\F_o|$ if we can argue that no fusion under field automorphisms can occur either (which is possible in some cases).

Each of the element orders $o$ in $S$ which we consider for the definition of $\F_o$ falls into exactly one of the following three categories:
\begin{itemize}
\item Category I: $o$ is semisimple and coprime to $\gcd(d+1,q+1)$. Then
\begin{itemize}
\item if $o$ is economic, we define $\F_o$ to consist of those (admissible) rational canonical forms with exactly one nontrivial block, which is of the form $\Comp(P(X))$ for some monic irreducible polynomial $P(X)\in\IF_{q^2}[X]$ of order $o$. Then $|\F_o|\geqslant\frac{\phi(o)}{\deg_{q^2}(o)}$, and $\omega_o(S)\geqslant\lceil\frac{|\F_o|}{2}\rceil$;
\item if $o$ is uneconomic, we define $\F_o$ to consist of those (admissible) rational canonical forms with exactly two nontrivial blocks, of the forms $\Comp(P(X))$ and $\Comp(\tilde{P}(X))$ for some monic irreducible polynomial $P(X)\in\IF_{q^2}[X]$ of order $o$. Then $|\F_o|\geqslant\lceil\frac{\phi(o)}{2\deg_{q^2}(o)}\rceil$ and $\omega_o(S)\geqslant\lceil\frac{|\F_o|}{2}\rceil$.
\end{itemize}
\item Category II: $o=p^e$ is unipotent. Then we let $\F_o$ consist of all unipotent rational canonical forms in $\GL_{d+1}(q^2)$ whose order is $p^e$ (note that unipotent rational canonical forms are always admissible). Then $|\F_o|$ is just $\pi(d+1,p,e)$, which is defined as the number of ordered integer partitions of $d+1$ all of whose parts are at most $p^e$ and which have at least one part strictly larger than $p^{e-1}$. Moreover, since all unipotent forms have coefficients in the prime field $\IF_p$, no fusion can occur under $\Phi_S$, whence $\omega_o(S)\geqslant|\F_o|$.
\item Category III: $o$ is neither semisimple nor unipotent. This only occurs in the following two cases:
\begin{itemize}
\item $S=\leftidx{^2}A_7(2^2)=\PSU_8(2)$ and $o=6=2^e\cdot 3$ with $e=1$;
\item $S=\leftidx{^2}A_8(2^2)=\PSU_9(2)$ and $o=2^e\cdot 3$ with $e\in\{1,2\}$.
\end{itemize}
In both cases, let $\xi\in\IF_{2^2}$ be a generator of $\IF_{2^2}^{\ast}$. We let $\F_o$ consist of those rational canonical forms in $\GL_{d+1}(2^2)$ whose nontrivial semisimple blocks are $\Comp(X-\xi)$ and $\Comp(X-\xi^{-1})$, occurring with the same multiplicity $m\in\{1,2,3\}$ (this is to ensure that the determinant is $1$) and whose other nontrivial blocks are all unipotent of maximal order $2^e$. It is not difficult to see that $|\F_o|=\sum_{m=1}^3{\pi(d+1-2m,2,e)}$ and $\omega_o(S)\geqslant|\F_o|$.
\end{itemize}

We are now ready to give the arguments for the lower bounds $W(S)$ on $\dfrak(S)=\omega(S)-\omicron(S)$ when $S$ is one of the four  $\leftidx{^2}A$ examples in compact, tabular form. Each table corresponds to one group $S$, and each row to an element order $o$ in $S$. One can then check that the value of $W(S)$ given in Table \ref{remainingTable} coincides with $\sum_o{(\underline{\omega_o}(S)-1)}$, where $o$ ranges over the element orders in $S$ listed in the respective table and $\underline{\omega_o}(S)$ is the lower bound on $\omega_o(S)$ given in the last column of the table.

\begin{center}
\begin{longtable}[H]{|c|c|c|c|c|}
\caption{$S={^2}A_3(11^2)=\PSU_4(11)$, $\gcd(d+1,q+1)=4$.}
\label{remainingTable1}\\
\hline
$o$ & Category of $o$ & Relevant info & $|\F_o|$ & $\omega_o(S)\geqslant$ \\ \hline
$305=5\cdot61$ & I & $o$ is uneco., $\deg_{11^2}(o)=2$ & $\geqslant\lceil\frac{\phi(305)}{4}\rceil=60$ & $\lceil\frac{60}{2}\rceil=30$ \\ \hline
\end{longtable}

\end{center}

\begin{center}
\begin{longtable}[H]{|c|c|c|c|c|}
\caption{$S={^2}A_5(5^2)=\PSU_6(5)$, $\gcd(d+1,q+1)=6$.}
\label{remainingTable2}\\
\hline
$o$ & Category of $o$ & Relevant info & $|\F_o|$ & $\omega_o(S)\geqslant$ \\ \hline
$217=7\cdot31$ & I & $o$ is uneco., $\deg_{5^2}(o)=3$ & $\geqslant\lceil\frac{\phi(217)}{2\cdot 3}\rceil=30$ & $\lceil\frac{30}{2}\rceil=15$ \\ \hline
$521$ & I & $o$ is eco., $\deg_{5^2}(o)=5$ & $=\frac{\phi(521)}{5}=104$ & $\lceil\frac{104}{2}\rceil=52$ \\ \hline
\end{longtable}

\end{center}

\begin{center}
\begin{longtable}[H]{|c|c|c|c|c|}
\caption{$S={^2}A_7(3^2)=\PSU_8(3)$, $\gcd(d+1,q+1)=4$.}
\label{remainingTable3}\\
\hline
$o$ & Category of $o$ & Relevant info & $|\F_o|$ & $\omega_o(S)\geqslant$ \\ \hline
$3$ & II & $|\F_o|=\pi(8,3,1)$ & $=9$ & $9$ \\ \hline
$6=2\cdot 3$ & III & $|\F_o|=\sum_{m=1}^3{\pi(8-2m,3,1)}$ & $=9$ & $9$ \\ \hline
$9=3^2$ & II & $|\F_o|=\pi(8,3,2)$ & $=12$ & $12$ \\ \hline
$61$ & I & $o$ is eco., $\deg_{3^2}(o)=5$ & $=\frac{\phi(61)}{5}=12$ & $\lceil\frac{12}{2}\rceil=6$ \\ \hline
$91=7\cdot13$ & I & $o$ is uneco., $\deg_{3^2}(o)=3$ & $\geqslant\lceil\frac{\phi(91)}{2\cdot3}\rceil=12$ & $\lceil\frac{12}{2}\rceil=6$ \\ \hline
$205=5\cdot41$ & I & $o$ is uneco., $\deg_{3^2}(o)=4$ & $\geqslant\lceil\frac{\phi(205)}{4}\rceil=20$ & $\lceil\frac{20}{2}\rceil=10$ \\ \hline
$547$ & I & $o$ is eco., $\deg_{3^2}(o)=7$ & $=\frac{\phi(547)}{7}=78$ & $\lceil\frac{78}{2}\rceil=39$ \\ \hline
\end{longtable}

\end{center}

\begin{center}
\begin{longtable}[H]{|c|c|c|c|c|}
\caption{$S={^2}A_8(2^2)=\PSU_9(2)$, $\gcd(d+1,q+1)=3$.}
\label{remainingTable4}\\
\hline
$o$ & Category of $o$ & Relevant info & $|\F_o|$ & $\omega_o(S)\geqslant$ \\ \hline
$2$ & II & $|\F_o|=\pi(9,2,1)$ & $=4$ & $4$ \\ \hline
$4=2^2$ & II & $|\F_o|=\pi(9,2,2)$ & $=13$ & $13$ \\ \hline
$6=2\cdot3$ & III & $|\F_o|=\sum_{m=1}^3{\pi(9-2m,2,1)}$ & $=6$ & $6$ \\ \hline
$8=2^3$ & II & $|\F_o|=\pi(9,2,3)$ & $=11$ & $11$ \\ \hline
$12=2^2\cdot3$ & III & $|\F_o|=\sum_{m=1}^3{\pi(9-2m,2,2)}$ & $=14$ & $14$ \\ \hline
$43$ & I & $o$ is eco., $\deg_{2^2}(o)=7$ & $=\frac{\phi(43)}{7}=6$ & $\lceil\frac{6}{2}\rceil=3$ \\ \hline
$85=5\cdot17$ & I & $o$ is uneco., $\deg_{2^2}(o)=4$ & $\geqslant\lceil\frac{\phi(85)}{8}\rceil=8$ & $\lceil\frac{8}{2}\rceil=4$ \\ \hline
\end{longtable}

\end{center}

We now turn to the remaining six groups $S$, all of which are of type $D$ or $\leftidx{^2}D$, i.e., of the form $\POmega_{2d}^{\epsilon}(q)$ with $\epsilon\in\{+,-\}$, and they have odd defining characteristic. The Schur cover of $S$ is $\Omega_{2d}^{\epsilon}(q)$, a subgroup of $\GO_{2d}^{\epsilon}(q)\leqslant\GL_{2d}(q)$. In \cite[Subsection 2.6, Case (C), pp.~38f.]{Wal63a}, Wall characterised those rational canonical forms in $\GL_{2d}(q)$ which are attained by an element of $\GO_{2d}^{\epsilon}(q)$. More precisely, these are just those rational canonical forms where
\begin{itemize}
\item the multiset of Frobenius blocks is closed under the involutory operation $\Comp(P(X)^k)\mapsto\Comp(P^{\ast}(X)^k)$ where $P^{\ast}(X)$ is the minimal polynomial over $\IF_q$ of $\xi^{-1}$ for any root $\xi\in\overline{\IF_q}$ of $P(X)$;
\item the multiplicity of each block of the form $\Comp((X\pm1)^{2k})$ is even; and
\item denoting by $\mu(P(X)^k)$ the multiplicity of the block $\Comp(P(X)^k)$ in the form: if there are no blocks of the form $\Comp((X\pm1)^{2k+1})$, then
\[
\sum_{P(X),k}{k\mu(P(X)^k)}\equiv\begin{cases}0\Mod{2}, & \text{if }\epsilon=+, \\ 1\Mod{2}, & \text{if }\epsilon=-,\end{cases}
\]
where $P(X)$ ranges over all monic irreducible polynomials in $\IF_q[X]$ and $k$ ranges over all positive integers.
\end{itemize}
It is easy to see that a monic irreducible polynomial $P(X)\in\IF_q[X]$ satisfies $P(X)=P^{\ast}(X)$ if and only if $\deg{P(X)}$ is even and $\ord(P(X))\mid q^{\deg(P(X))/2}+1$, in which case we call either of
\begin{itemize}
\item the positive $p'$-integer $\ord(P(X))$,
\item the polynomial $P(X)\in\IF_q[X]$, or
\item the Frobenius block $\Comp(P(X)^k)$
\end{itemize}
\emph{$\O$-economic} or just \emph{economic} for short (note that this is distinct from the notion of \emph{$\U$-economic} objects used in the unitary case above). Positive $p'$-integers, monic irreducible polynomials in $\IF_q[X]$, and Frobenius blocks over $\IF_q$ which are not economic will be called \emph{uneconomic}. Note that, as in the unitary case, these notions of (un)economic objects depend on $q$, which we view as fixed. Also, observe that a positive $p'$-integer $o$ is economic if and only if it satisfies either of the following, equivalent conditions:
\begin{itemize}
\item $o$ divides some number of the form $q^k+1$ with $k\in\IN^+$.
\item Either $o\leq2$, or the \emph{$q$-degree of $o$}, denoted by $\deg_q(o)$ and defined as the smallest positive integer $t$ such that $o$ divides $q^t-1$, is even, and $o\mid q^{\deg_q(o)/2}+1$.
\end{itemize}

Our basic strategy is to specify, for a given element order $o$ in $S=\POmega_{2d}^{\epsilon}(q)$, a (preferably large) set $\F_o$ of order $o$ rational canonical forms in $\GL_{2d}(q)$ such that the following hold:
\begin{enumerate}
\item all forms in $\F_o$ are attained by an element of $\POmega_{2d}^{\epsilon}(q)$ (not just $\GO_{2d}^{\epsilon}(q)$);
\item all forms in $\F_o$ have order $o$ \enquote{modulo the scalars $\zeta\Omega_{2d}^{\epsilon}$} (i.e., for each form in $\F_o$, $o$ is the smallest positive integer $m$ such that the $m$-th power of the form is a scalar matrix in $\zeta\Omega_{2d}^{\epsilon}$);
\item no two matrix similarity classes represented by distinct forms in $\F_o$ are fused under multiplication by scalars in $\zeta\Omega_{2d}^{\epsilon}(q)$.
\end{enumerate}
Each of these properties, say (x), is implied by a certain other property, (x$'$), which we will now list:
\renewcommand{\labelenumii}{(\alph{enumii}$'$)}
\begin{enumerate}
\item all forms in $\F_o$ are similar to the square of the rational canonical form of some element in $\GO_{2d}^{\epsilon}(q)$, as characterised by Wall;
\item all forms in $\F_o$ have at least one Frobenius block of odd order;
\item there is an odd divisor $o'$ of $o$ such that every form in $\F_o$ has at least one block of order $o'$, but no form in $\F_o$ has a block of order $2o'$.
\end{enumerate}
\renewcommand{\labelenumii}{(\alph{enumii})}
To see that (a$'$) implies (a), use that $\GO_{2d}^{\epsilon}(q)/\Omega_{2d}^{\epsilon}(q)$ is of exponent $2$, and to see that (b$'$) implies (b) and (c$'$) implies (c), use that $|\zeta\Omega_{2d}^{\epsilon}(q)|\leqslant2$. Since $D_5(3)=\Omega_{10}^+(3)$ is centreless, we do not need to worry about properties (b) and (c) at all when $S=D_5(3)$. Throughout the concrete discussion of the six remaining examples $S$ below, properties (a$'$) and (b$'$) (and thus (a) and (b)) will be satisfied for each of the element orders $o$ of $S$ that we will consider. Property (c) will also always be satisfied, but sometimes, property (c$'$) is not (for example, when $S=D_4(7)$ and $o=4$); in those cases, one needs to verify directly that no two of the given forms are fused under scalar multiplication (by $-I_{2d}$).

After specifying $\F_o$ and computing its cardinality, we need to consider potential fusion of forms in $\F_o$ under $\Aut(S)$. Note that $\Aut(S)$ contains the subgroup $\PCO_{2d}^{\epsilon}(q)$, which does not fuse distinct matrix similarity classes. If $S=\leftidx{^2}D_5(3^2)$, which is the only group of type $\leftidx{^2}D$ that we need to consider, then $\Aut(S)=\PCO_{2d}^-(3)$, and we may use $|\F_o|$ itself as a lower bound on $\omega_o(S)$. So assume for the rest of this paragraph that $S=D_d(q)=\POmega_{2d}^+(q)$. Then the fusion of forms in $\F_o$ under $\Aut(S)$ is controlled by the index $|\Aut(S):\PCO_{2d}^+(q)|$. If $d>4$, then $\Aut(S)=\PGammaO_{2d}^+(q)=\PCO_{2d}^+(q)\Phi_S$, and so we only need to worry about fusion under field automorphisms of $S$, of which there are $|\Phi_S|=f$, and we conclude that $\omega_o(S)\geqslant\frac{|\F_o|}{f}$. If, on the other hand, $d=4$, then $\PGammaO_{2d}^+(q)$ is an index $3$ subgroup of $\Aut(S)$, and $\Aut(S)=\PGammaO_{2d}^+(q)\langle\gamma\rangle$ where $\gamma$ is an order $3$ graph automorphism of $S$; it follows that $\omega_o(S)\geqslant\frac{|\F_o|}{3f}$. The only case where we obtain better lower bounds on $\omega_o(S)$ is when $d=4$ and $o=p$ (the defining characteristic), using \cite[Proposition 3.55(i)]{Bur07a}.

We now provide some more information relevant for reading Tables \ref{remainingTable5} to \ref{remainingTable9} below. Each table corresponds to one group $S$, and the rows correspond to the element orders $o$ in $S$ for which we specify a set $\F_o$ of rational canonical forms as explained above and subsequently compute a lower bound on $\omega_o(S)$.

When $o'$ is a semisimple element order in $S$, then the number of distinct monic irreducible polynomials in $\IF_q[X]$ of order $o'$ (all of which have degree $\deg_q(o)$) is exactly $\frac{\phi(o)}{\deg_q(o)}$. If $o'$ is uneconomic, then these polynomials come in at least $\lceil\frac{\phi(o)}{2\deg_q(o)}\rceil$ pairs $\{P(X),P^{\ast}(X)\}$. We use the notation $P_{o'}=P_{o'}(X)$ to denote an arbitary monic irreducible polynomial in $\IF_q[X]$ of order $o'$.

For describing the forms in $\F_o$, we specify the multiplicities of their Frobenius blocks $\Comp(P(X)^k)$, with the convention that blocks which are not mentioned occur with multiplicity $0$. When doing so, we identify $P^k=P(X)^k$ with $\Comp(P(X)^k)$ for brevity, and we use $I_k$ to denote the $(k\times k)$-identity matrix over $\IF_q$; in particular, $I_1$ is a copy of the trivial Frobenius block over $\IF_q$. We write $(B,a)$ shorthand for \enquote{the Frobenius block $B$ occurs with multiplicity $a$}, and we separate these multiplicity specifications for one type of form by commas, with a semicolon separating descriptions of forms of different shape in the same set $\F_o$ (such as for $S=D_4(5)$ and $o=13$). When the form may involve several companion blocks of polynomials of the same order $o'$ which may or may not be equal, we denote those polynomials by $P_{o'}^{(1)}$, $P_{o'}^{(2)}$, and so on.

In order to see that in those cases where $o$ is even, the specified forms in $\F_o$ are indeed similar to squares of forms of elements in $\GO_{2d}^{\epsilon}(q)$, we make the following observations:
\begin{itemize}
\item When $q=3$ (relevant for $S=D_5(3),\leftidx{^2}D_5(3^2),D_6(3)$): $-I_2=\Comp(P_4(X))^2$ for the unique order $4$ (quadratic) monic irreducible polynomial $P_4(X)\in\IF_3[X]$, which is economic.
\item When $q=5$ (relevant for $S=D_4(5)$): $-I_2$ is the square of the rational canonical form whose nontrivial blocks (both occurring with multiplicity $1$) are the companion matrices of the two order $4$ (linear) monic irreducible polynomials $P_4(X)$ and $P^{\ast}_4(X)$ in $\IF_5[X]$.
\item When $q=7$ (relevant for $S=D_4(7)$):
\begin{itemize}
\item $-I_2=\Comp(P_4(X))^2$ for the unique order $4$ (quadratic) monic irreducible polynomial $P_4(X)\in\IF_7[X]$, which is economic;
\item $\Comp(P_4(X))$ is similar to $\Comp(P_8(X))^2$ for any of the two order $8$ (quadratic) monic irreducible polynomials $P_8(X)\in\IF_7[X]$, which are economic.
\end{itemize}
\end{itemize}

Finally, for the counting of forms of unipotent elements in $S$, we denote by $\Pi'(2d,p,e)$ (resp.~$\pi'(2d,p,e)$) the set (resp.~number) of ordered integer partitions of $2d$ such that all parts are at most $p^e$, at least one part is larger than $p^{e-1}$, and all multiplicities of even parts are even. We identify elements of $\Pi'(2d,p,e)$ with unipotent rational canonical forms in $S$ by assigning to a partitition $\lambda=(\lambda_1,\ldots,\lambda_s)\in\Pi'(2d,p,e)$ the form with blocks $\Comp((X-1)^{\lambda_1}),\ldots,\Comp((X-1)^{\lambda_s})$, listed with multiplicities.

\begin{center}
\begin{longtable}[H]{|c|c|c|c|c|}
\caption{$S=D_4(5)=\POmega_8^+(5)$.}
\label{remainingTable5}\\
\hline
$o$ & relevant info & forms in $\F_o$ & $|\F_o|$ & $\omega_o(S)\geqslant$ \\ \hline
$3$ & \thead{$3$ is eco., $\deg_5(3)=2$, \\ $\#P_3=\frac{\phi(3)}{2}=1$} & \thead{$(P_3,a),(I_1,8-2a)$ \\ for $a\in\{1,2,3,4\}$} & $4$ & $\lceil\frac{4}{3}\rceil=2$ \\ \hline
$5$ & Use \cite[Proposition 3.55(i)]{Bur07a} & $\Pi'(8,5,1)\setminus\{(2,2,2,2),(4,4)\}$ & $\pi'(8,5,1)-2=5$ & $5$ \\ \hline
$6$ & \thead{$3$ is eco., $\deg_5(3)=2$, \\ $\#P_3=\frac{\phi(3)}{2}=1$} & \thead{$(P_3,a),(-I_1,2b),(I_1,8-2(a+b))$ \\ for $a,b\geqslant1$, $a+b\leqslant4$} & $6$ & $\lceil\frac{6}{3}\rceil=2$ \\ \hline
$10$ & none & \thead{$(-I_1,2a),\Pi'(8-2a,5,1)$ \\ for $a\in\{1,2\}$} & $6$ & $\lceil\frac{6}{3}=2\rceil$ \\ \hline
$13$ & \thead{$13$ is eco., $\deg_5(13)=4$, \\ $\#P_{13}=\frac{\phi(13)}{4}=3$} & \thead{$(P_{13},1),(I_1,4)$; \\ $(P_{13}^{(1)},1),(P_{13}^{(2)},1)$} & $3+3+{3\choose 2}=9$ & $\lceil\frac{9}{3}\rceil=3$ \\ \hline
$21$ & \thead{$21$ is eco., $\deg_5(21)=6$, \\ $\#P_{21}=\frac{\phi(21)}{6}=2$; \\ $3$ is eco., $\deg_5(3)=2$, \\ $\#P_3=\frac{\phi(3)}{2}=1$} & \thead{$(P_{21},1),(I_1,2)$; \\ $(P_{21},1),(P_3,1)$} & $2+2=4$ & $\lceil\frac{4}{3}\rceil=2$ \\ \hline
$26$ & \thead{$13$ is eco., $\deg_5(13)=4$ \\ $\#P_{13}=\frac{\phi(13)}{4}=3$} & \thead{$(P_{13},1),(-I_1,2a),(I_1,4-2a)$ \\ for $a\in\{1,2\}$} & $2\cdot 3=6$ & $\lceil\frac{6}{3}\rceil=2$ \\ \hline
$31$ & \thead{$31$ is uneco., $\deg_5(31)=3$, \\ $\#\{P_{31},P^{\ast}_{31}\}=\frac{\phi(31)}{2\cdot3}=5$} & $(P_{31},1),(P^{\ast}_{31},1),(I_1,2)$ & $5$ & $\lceil\frac{5}{3}\rceil=2$ \\ \hline
$62$ & \thead{$31$ is uneco., $\deg_5(31)=3$, \\ $\#\{P_{31},P^{\ast}_{31}\}=\frac{\phi(31)}{2\cdot3}=5$} & $(P_{31},1),(P^{\ast}_{31},1),(-I_1,2)$ & $5$ & $\lceil\frac{5}{3}=2$ \\ \hline
$63$ & \thead{$63$ is eco., $\deg_5(63)=6$, \\ $\#P_{63}=\frac{\phi(63)}{6}=6$; \\ $3$ is eco., $\deg_5(3)=2$, \\ $\#P_3=\frac{\phi(3)}{2}=1$} & \thead{$(P_{63},1),(I_1,2)$; \\ $(P_{63},1),(P_3,1)$} & $6+6=12$ & $\lceil\frac{12}{3}\rceil=4$ \\ \hline
$126$ & \thead{$63$ is eco., $\deg_5(63)=6$, \\ $\#P_{63}=\frac{\phi(63)}{6}=6$; \\ $3$ is eco., $\deg_5(3)=2$, \\ $\#P_3=\frac{\phi(3)}{2}=1$} & $(P_{63},1),(-I_1,2)$ & $6$ & $\lceil\frac{6}{3}\rceil=2$ \\ \hline
\end{longtable}
\end{center}

\begin{center}
\begin{longtable}[H]{|c|c|c|c|c|}
\caption{$S=D_4(7)=\POmega^+_8(7)$.}
\label{remainingTable6}\\
\hline
$o$ & relevant info & forms in $\F_o$ & $|\F_o|$ & $\omega_o(S)\geqslant$ \\ \hline
$3$ & \thead{$3$ is uneco., $\deg_7(3)=1$, \\ $\#\{P_3,P^{\ast}_3\}=\frac{\phi(3)}{2\cdot1}=1$} & \thead{$(P_3,a),(P^{\ast}_3,a),(I_1,8-2a)$ \\ for $a\in\{1,2,3,4\}$} & $4$ & $\lceil\frac{4}{3}\rceil=2$ \\ \hline
$4$ & \thead{$4$ is eco., $\deg_7(4)=2$, \\ $\#P_4=\frac{\phi(4)}{2}=1$} & \thead{$(P_4,a),(-I_1,2b),(I_1,8-2(a+b))$ \\ for $1\leqslant a\leqslant 3,b\geqslant0,a+2b\leqslant4$} & $5$ & $\lceil\frac{5}{3}\rceil=2$ \\ \hline
$6$ & \thead{$3$ is uneco., $\deg_7(3)=1$, \\ $\#\{P_3,P^{\ast}_3\}=\frac{\phi(3)}{2\cdot1}=1$} & \thead{$(P_3,a),(P^{\ast}_3,a),(-I_1,2b),(I_1,8-2(a+b))$ \\ for $a,b\geqslant1,a+b\leqslant4$} & 6 & $\lceil\frac{6}{3}\rceil=2$ \\ \hline
$7$ & Use \cite[Proposition 3.55(i)]{Bur07a} & $\Pi'(8,7,1)\setminus\{(2,2,2,2),(4,4)\}$ & $6$ & $6$ \\ \hline
$12$ & \thead{$3$ is uneco., $\deg_7(3)=1$, \\ $\#\{P_3,P^{\ast}_3\}=\frac{\phi(3)}{2\cdot1}=1$; \\ $4$ is eco., $\deg_7(4)=2$, \\ $\#P_4=\frac{\phi(4)}{2}=1$} & \thead{$(P_3,a),(P^{\ast}_3,a),(P_4,b),(-I_1,2c),(I_1,8-2(a+b+c))$ \\ for $a,b\geqslant1,c\geqslant0,a+b+2c\leqslant4$, \\ and if $c=0$ and $a+b=4$, then $2\mid b$} & $8$ & $\lceil\frac{8}{3}\rceil=3$ \\ \hline
$25$ & \thead{$25$ is eco., $\deg_7(25)=4$, \\ $\#P_{25}=\frac{\phi(25}{4}=5$; \\ $5$ is eco., $\deg_7(5)=4$, \\ $\#P_5=\frac{\phi(5)}{4}=1$} & \thead{$(P_{25},1),(I_1,4)$; \\ $(P_{25},1),(P_5,1)$; \\ $(P_{25}^{(1)},1),(P_{25}^{(2)},1)$} & $25$ & $\lceil\frac{25}{3}\rceil=9$ \\ \hline
$43$ & \thead{$43$ is eco., $\deg_7(43)=6$, \\ $\#P_{43}=\frac{\phi(43)}{6}=7$} & $(P_{43},1),(I_1,2)$ & $7$ & $\lceil\frac{7}{3}\rceil=3$ \\ \hline
$50$ & \thead{$25$ is eco., $\deg_7(25)=4$, \\ $\#P_{25}=\frac{\phi(25}{4}=5$; \\ $4$ is eco., $\deg_7(4)=2$, \\ $\#P_4=\frac{\phi(4)}{2}=1$} & $(P_{25},1),(P_4,1),(I_1,2)$ & $10$ & $\lceil\frac{10}{3}\rceil=4$ \\ \hline
$57$ & \thead{$57$ is uneco., $\deg_7(57)=3$, \\ $\#\{P_{57},P^{\ast}_{57}\}=\frac{\phi(57)}{2\cdot3}=6$} & $(P_{57},1),(I_1,2)$ & $6$ & $\lceil\frac{6}{3}\rceil=2$ \\ \hline
$75$ & \thead{$25$ is eco., $\deg_7(25)=4$, \\ $\#P_{25}=\frac{\phi(25}{4}=5$; \\ $3$ is uneco., $\deg_7(3)=1$, \\ $\#\{P_3,P^{\ast}_3\}=\frac{\phi(3)}{2\cdot1}=1$} & $(P_{25},1),(P_3,1),(P^{\ast}_3,1),(I_1,2)$ & $5$ & $\lceil\frac{5}{3}\rceil=2$ \\ \hline
$86$ & \thead{$43$ is eco., $\deg_7(43)=6$, \\ $\#P_{43}=\frac{\phi(43)}{6}=7$} & $(P_{43},1),(-I_1,2)$ & $7$ & $\lceil\frac{7}{3}\rceil=3$ \\ \hline
$100$ & \thead{$25$ is eco., $\deg_7(25)=4$, \\ $\#P_{25}=\frac{\phi(25}{4}=5$; \\ $4$ is eco., $\deg_7(4)=2$, \\ $\#P_4=\frac{\phi(4)}{2}=1$} & $(P_{25},1),(P_4,1),(I_1,2)$ & $5$ & $\lceil\frac{5}{3}\rceil=2$ \\ \hline
$171$ & \thead{$171$ is uneco., $\deg_7(171)=3$, \\ $\#\{P_{171},P^{\ast}_{171}\}=\frac{\phi(171)}{2\cdot3}=18$} & $(P_{171},1),(P^{\ast}_{171},1),(I_1,2)$ & $18$ & $\lceil\frac{18}{3}\rceil=6$ \\ \hline
$172$ & \thead{$43$ is eco., $\deg_7(43)=6$, \\ $\#P_{43}=\frac{\phi(43)}{6}=7$; \\ $4$ is eco., $\deg_7(4)=2$, \\ $\#P_4=\frac{\phi(4)}{2}=1$} & $(P_{43},1),(P_4,1)$ & $7$ & $\lceil\frac{7}{3}\rceil=3$ \\ \hline
\end{longtable}
\end{center}

\begin{center}
\begin{longtable}[H]{|c|c|c|c|c|}
\caption{$S=D_4(9)=\POmega^+_8(9)$.}
\label{remainingTable7}\\
\hline
$o$ & relevant info & forms in $\F_o$ & $|\F_o|$ & $\omega_o(S)\geqslant$ \\ \hline
$3$ & Use \cite[Proposition 3.55(i)]{Bur07a} & $\Pi'(8,3,1)\setminus\{(2,2,2,2),(4,4)\}$ & $4$ & $4$ \\ \hline
$5$ & \thead{$5$ is eco., $\deg_9(5)=2$, \\ $\#P_5=\frac{\phi(5)}{2}=2$} & \thead{$(P_5^{(1)},a),(P_5^{(2)},b),(I_1,8-2(a+b))$ \\ for $P_5^{(1)}\not=P_5^{(2)}$, $a,b\geqslant0$, $1\leqslant a+b\leqslant4$} & $14$ & $\lceil\frac{14}{6}\rceil=3$ \\ \hline
$41$ & \thead{$41$ is eco., $\deg_9(41)=4$, \\ $\#P_{41}=\frac{\phi(41)}{4}=10$} & \thead{$(P_{41},1),(I_1,4)$; \\ $(P_{41}^{(1)},1),(P_{41}^{(2)},1)$} & $65$ & $\lceil\frac{65}{6}\rceil=11$ \\ \hline
$365$ & \thead{$365$ is eco., $\deg_9(365)=6$, \\ $\#P_{365}=\frac{\phi(365)}{6}=48$; \\ $5$ is eco., $\deg_9(5)=2$, \\ $\#P_5=\frac{\phi(5)}{2}=2$} & \thead{$(P_{365},1),(I_1,2)$; $(P_{365},1),(P_5,1)$} & $144$ & $\lceil\frac{144}{6}\rceil=24$ \\ \hline
\end{longtable}
\end{center}

\begin{center}
\begin{longtable}[H]{|c|c|c|c|c|}
\caption{$S=D_5(3)=\POmega^+_{10}(3)=\Omega^+_{10}(3)$.}
\label{remainingTable8}\\
\hline
$o$ & relevant info & forms in $\F_o$ & $|\F_o|$ & $\omega_o(S)\geqslant$ \\ \hline
$2$ & none & \thead{$(-I_1,2a),(I_1,10-2a)$ \\ for $a\in\{1,\ldots,5\}$} & $5$ & $5$ \\ \hline
$3$ & none & $\Pi'(10,3,1)$ & $7$ & $7$ \\ \hline
$5$ & \thead{$5$ is eco. $\deg_3(5)=4$, \\ $\#P_5=\frac{\phi(4)}{4}=1$} & \thead{$(P_4,a),(I_1,10-4a)$ \\ for $a\in\{1,2\}$} & $2$ & $2$ \\ \hline
$6$ & none & \thead{$(-I_1,2a),\Pi'(10-2a,3,1)$ \\ for $a\in\{1,2,3\}$} & $10$ & $10$ \\ \hline
$9$ & none & $\Pi'(10,3,2)$ & $8$ & $8$ \\ \hline
$10$ & \thead{$5$ is eco. $\deg_3(5)=4$, \\ $\#P_5=\frac{\phi(4)}{4}=1$} & \thead{$(P_5,1),(-I_1,2a),(I_1,6-2a)$ \\ for $a\in\{1,2,3\}$} & $4$ & $4$ \\ \hline
\end{longtable}
\end{center}

\begin{center}
\begin{longtable}[H]{|c|c|c|c|c|}
\caption{$S={^2}D_5(3^2)=\POmega^-_{10}(3)$.}
\label{remainingTable9}\\
\hline
$o$ & relevant info & forms in $\F_o$ & $|\F_o|$ & $\omega_o(S)\geqslant$ \\ \hline
$2$ & none & \thead{$(-I_1,2a),(I_1,10-2a)$ \\ for $a\in\{1,2\}$} & $2$ & $2$ \\ \hline
$3$ & none & $\Pi'(10,3,1)$ & $7$ & $7$ \\ \hline
$5$ & \thead{$5$ is eco. $\deg_3(5)=4$, \\ $\#P_5=\frac{\phi(4)}{4}=1$} & \thead{$(P_4,a),(I_1,10-4a)$ \\ for $a\in\{1,2\}$} & $2$ & $2$ \\ \hline
$6$ & none & \thead{$(-I_1,2a),\Pi'(10-2a,3,1)$ \\ for $a\in\{1,2,3\}$} & $10$ & $10$ \\ \hline
$10$ & \thead{$5$ is eco. $\deg_3(5)=4$, \\ $\#P_5=\frac{\phi(4)}{4}=1$} & \thead{$(P_5,1),(-I_1,2a),(I_1,6-2a)$ \\ for $a\in\{1,2,3\}$} & $4$ & $4$ \\ \hline
\end{longtable}
\end{center}

\begin{center}
\begin{longtable}[H]{|c|c|c|c|c|}
\caption{$S=D_6(3)=\POmega^+_{12}(3)$.}
\label{remainingTable10}\\
\hline
$o$ & relevant info & forms in $\F_o$ & $|\F_o|$ & $\omega_o(S)\geqslant$ \\ \hline
$2$ & none & \thead{$(-I_1,2a),(I_1,12-2a)$ \\ for $a\in\{1,2,3\}$} & $3$ & $3$ \\ \hline
$3$ & none & $\Pi'(12,3,1)$ & $10$ & $10$ \\ \hline
$6$ & none & \thead{$(-I_1,2a),\Pi'(12-2a,3,1)$ \\ for $a\in\{1,2,3,4\}$} & $17$ & $17$ \\ \hline
$9$ & none & $\Pi'(12,3,2)$ & $15$ & $15$ \\ \hline
$41$ & \thead{$41$ is eco., $\deg_3(41)=8$, \\ $\#P_{41}=\frac{\phi(41)}{8}=5$} & $(P_{41},1),(I_1,4)$ & $5$ & $5$ \\ \hline
$61$ & \thead{$61$ is eco., $\deg_3(61)=10$, \\ $\#P_{61}=\frac{\phi(61)}{10}=6$} & $(P_{61},1),(I_1,2)$ & $6$ & $6$ \\ \hline
$82$ & \thead{$41$ is eco., $\deg_3(41)=8$, \\ $\#P_{41}=\frac{\phi(41)}{8}=5$} & \thead{$(P_{41},1),(-I_1,2a),(I_1,4-2a)$ \\ for $a\in\{1,2\}$} & $10$ & $10$ \\ \hline
$91$ & \thead{$91$ is uneco., $\deg_3(91)=6$, \\ $\#\{P_{91},P^{\ast}_{91}\}=\frac{\phi(91)}{2\cdot6}=6$} & $(P_{91},1),(P^{\ast}_{91},1)$ & $6$ & $6$ \\ \hline
$122$ & \thead{$61$ is eco., $\deg_3(61)=10$, \\ $\#P_{61}=\frac{\phi(61)}{10}=6$} & $(P_{122},1),(-I_1,2)$ & $6$ & $6$ \\ \hline
\end{longtable}
\end{center}\qedhere
\end{enumerate}
\end{proof}

\numberwithin{equation}{section}
\section{Proof of Theorem \ref{mainTheo1}(1)}\label{sec3}

We start with the following lemma, which provides upper bounds for $\dfrak(N)$ and $\m(G/N)$ in terms of $\dfrak(G)$ (see Definition \ref{mainDef1}(4,a)), where $N$ is a characteristic subgroup of $G$:

\begin{lemma}\label{charQuotLem}
Let $G$ be a finite group, and let $N$ be a characteristic subgroup of $G$. Then

\begin{enumerate}
\item $\dfrak(N)\leqslant\dfrak(G)$.
\item $\m(G/N)\leqslant 2^{\dfrak(G)}+\dfrak(G)$.
\end{enumerate}
\end{lemma}

We note that the special case $\dfrak(G)=0$ (i.e., when $G$ is an AT-group) in Lemma \ref{charQuotLem} is just \cite[Lemma 1.1]{Zha92a}, and the proof of Lemma \ref{charQuotLem} is also a generalisation of the proof of \cite[Lemma 1.1]{Zha92a}.

\begin{proof}[Proof of Lemma \ref{charQuotLem}]
For statement (1): If two elements of $N$ are $\Aut(G)$-conjugate, then they are also $\Aut(N)$-conjugate (or, equivalently, $\Aut(N)$-orbits on $N$ are unions of $\Aut(G)$-orbits on $N$). In particular, $\omega_o(G)\geqslant\omega_o(N)$ for each $o\in\Ord(N)\subseteq\Ord(G)$. It follows that
\[
\dfrak(G)=\sum_{o\in\Ord(G)}{(\omega_o(G)-1)}\geqslant\sum_{o\in\Ord(N)}{(\omega_o(G)-1)}\geqslant\sum_{o\in\Ord(N)}{(\omega_o(N)-1)}=\dfrak(N),
\]
as required.

For statement (2): For a group $H$ and a positive integer $o$, we denote the set of order $o$ elements in $H$ by $H_o$. By definition, $\m(G/N)$ is the maxium value of $\omega_{\overline{o}}(G/N)$ where $\overline{o}$ ranges over the element orders of $G/N$. So the goal will be to show that $\omega_{\overline{o}}(G/N)\leqslant 2^{\dfrak(G)}+\dfrak(G)$ for all $\overline{o}\in\Ord(G/N)$. Consider the following two conditions on such an $\overline{o}$:
\begin{enumerate}
\item There is a set $M_{\overline{o}}$ of $\Aut(G)$-orbits on $G$ with $|M_{\overline{o}}|\leqslant 2^{\dfrak(G)}+\dfrak(G)$ such that for each $\overline{x}\in(G/N)_{\overline{o}}$, there is a lift $x$ of $\overline{x}$ in $G$ such that $x$ lies in one of the orbits from $M_{\overline{o}}$.
\item There is a set $N_{\overline{o}}$ of positive integers with $|N_{\overline{o}}|\leqslant 2^{\dfrak(G)}$ such that each $\overline{x}\in(G/N)_{\overline{o}}$ admits a lift $x$ in $G$ such that $\ord(x)\in N_{\overline{o}}$.
\end{enumerate}
Since $N$ is characteristic in $G$, if two elements of $G/N$ have lifts in the same $\Aut(G)$-orbit on $G$, then they are $\Aut(G/N)$-conjugate, and so the first condition implies that $\omega_{\overline{o}}(G/N)\leqslant 2^{\dfrak(G)}+\dfrak(G)$ (which is what we want to show). Moreover, the second condition implies the first, by letting $M_{\overline{o}}$ be the set of all $\Aut(G)$-orbits on $G$ consisting of elements whose order lies in $N_{\overline{o}}$ -- then by definition of $\dfrak(G)$, $|M_{\overline{o}}|\leqslant|N_{\overline{o}}|+\dfrak(G)\leqslant 2^{\dfrak(G)}+\dfrak(G)$.

Hence we will aim at verifying that the second condition holds for all $\overline{o}\in\Ord(G/N)$. So, fix such an $\overline{o}$, say with prime power factorisation $\overline{o}=p_1^{f_1}\cdots p_s^{f_s}$. Denote by $\pi$ the canonical projection $G\rightarrow G/N$, and consider the following function $\lambda_{\overline{o}}$, which maps the set $\pi^{-1}[(G/N)_{\overline{o}}]$, of all lifts in $G$ of order $\overline{o}$ elements of $G/N$, into itself: For $x\in\pi^{-1}[(G/N)_{\overline{o}}]$, let $x=x_1\cdots x_r$ be the unique (up to reordering the factors) factorisation of $x$ into pairwise commuting elements of pairwise coprime prime-power orders. Since $\pi(x)$ has order $\overline{o}$, we have that $p_i\mid\ord(x)$ for $i=1,\ldots,s$, so $r\geqslant s$, and we may assume w.l.o.g.~that $\ord(x_i)=p_i^{k_i}$ for some $k_i\in\IN^+$ for $i=1,\ldots,s$. Set $\lambda_{\overline{o}}(x):=x_1\cdots x_s$. Note that $x_{s+1},\ldots,x_r\in N$, and we have
\begin{equation}\label{lambdaEq}
\pi(\lambda_{\overline{o}}(x))=\pi(x),
\end{equation}
which shows in particular that $\lambda_{\overline{o}}(x)\in\pi^{-1}[(G/N)_{\overline{o}}]$, as asserted. We let $X_{\overline{o}}$ denote the set of orders of elements in the image of $\lambda_{\overline{o}}$. By Formula (\ref{lambdaEq}), each $\overline{x}\in(G/N)_{\overline{o}}$ has a lift in $\im(\lambda_{\overline{o}})$, and thus a lift with order in $X_{\overline{o}}$. It remains to show that $|X_{\overline{o}}|\leqslant 2^{\dfrak(G)}$.

For $i\in\{1,\ldots,s\}$, let $t_i\in\IN^+$ be maximal subject to $p_i^{t_i}\in\Ord(N)$, and let $u_i$ be the number of distinct $p_i$-adic valuations of $\ord(x^{p_i^{f_i}})$ where $x$ ranges over $\im(\lambda_{\overline{o}})$; hence by definition, $|X_{\overline{o}}|\leqslant\prod_{i=1}^s{u_i}$. Note that $u_i\leqslant t_i+1$, since the $p_i$-part of $x^{p_i^{f_i}}$ is (by definition of $f_i$) always an element of $N$. Observe also that for a fixed $i$, as long as each of the subsets $G_1,G_{p_i},G_{p_i^2},\ldots,G_{p_i^{t_i}}\subseteq G$ is a single $\Aut(G)$-orbit (which must hold for all but at most $\dfrak(G)$ of the indices $i\in\{1,\ldots,s\}$), then the argument in \cite[proof of Lemma 1.1]{Zha92a} gives that the $p_i$-adic valuation of $\ord(x)$ is $f_i+t_i$ for all $x\in\im(\lambda_{\overline{o}})$, and so $u_i=1$. Let the number of indices $i\in\{1,\ldots,s\}$ for which this is not the case be $e$, and let these $e$ \enquote{exceptional} indices be w.l.o.g.~just $1,\ldots,e$. Note that if $e=0$, then $u_i=1$ for all $i\in\{1,\ldots,s\}$, so that
\[
\prod_{i=1}^s{u_i}=1\leqslant 2^{\dfrak(G)},
\]
as required. We may thus assume that $e\geqslant1$. Moreover, we claim that
\begin{equation}\label{uEq}
\sum_{i=1}^s{(u_i-1)}=\sum_{i=1}^e{(u_i-1)}\leqslant\dfrak(G).
\end{equation}
Indeed, the equality in Formula (\ref{uEq}) is clear by the above remark that $i\in\{1,\ldots,s\}$ not being among the $e$ exceptional indices $1,\ldots,e$ implies that $u_i=1$. As for the inequality in Formula (\ref{uEq}), we will argue as follows: Consider the set $\Mcal$, of all pairs $(i,m)$ such that
\begin{itemize}
\item $i\in\{1,\ldots,e\}$,
\item $m\in\{0,\ldots,t_i-1\}$, and
\item there exists $x\in\im(\lambda_{\overline{o}})$ such that $\nu_{p_i}(\ord(x^{p_i^{f_i}}))=m$.
\end{itemize}

Note that if the second condition in the definition of $\Mcal$ was replaced by \enquote{$m\in\{0,\ldots,t_i\}$}, then by definition of $t_i$ and $u_i$, the cardinality of $\Mcal$ would be $\sum_{i=1}^e{u_i}$; excluding the possibility $m=t_i$ removes at most one pair $(i,m)$ for each $i$, and so the actual cardinality of $\Mcal$ is bounded from below by $\sum_{i=1}^e{(u_i-1)}$. Consider the injective function $f:\Mcal\rightarrow\IN^+$, $(i,m)\mapsto p_i^{m+1}$. Observe that the image of $f$ consists of element orders $o\in\Ord(G)$ such that $G$ contains elements of order $o$ both inside and outside of $N$ (the former since $m+1\leqslant t_i$, and the latter by considering the $p_i$-part of $x^{p_i^{f_i-1}}$ where $x$ is as in the third bullet point of the definition of $\Mcal$ above). In particular, $\omega_o(G)\geqslant 2$ for each such $o$, and thus
\[
\dfrak(G)=\sum_{o\in\Ord(G)}{(\omega_o(G)-1)}\geqslant|\im(f)|\geqslant|\Mcal|\geqslant\sum_{i=1}^e{(u_i-1)},
\]
as asserted. Using the now established Formula (\ref{uEq}) and the inequality of arithmetic and geometric means, we deduce that
\begin{equation}\label{uEq2}
|X_{\overline{o}}|\leqslant\prod_{i=1}^s{u_i}=\prod_{i=1}^e{u_i}\leqslant(\frac{\sum_{i=1}^e{u_i}}{e})^e\leqslant(\frac{\dfrak(G)+e}{e})^e=(\frac{\dfrak(G)}{e}+1)^e.
\end{equation}
Now for each real number $y\geqslant1$, we have $2^y\geqslant y+1$. Applied with $y:=\dfrak(G)/e$ (using that $e\leqslant\dfrak(G)$ by definition of $e$), we get that
\[
\frac{\dfrak(G)}{e}+1\leqslant 2^{\dfrak(G)/e},
\]
or equivalently,
\[
(\frac{\dfrak(G)}{e}+1)^e\leqslant 2^{\dfrak(G)},
\]
which together with Formula (\ref{uEq2}) implies that $|X_{\overline{o}}|\leqslant 2^{\dfrak(G)}$ and thus concludes the proof.
\end{proof}

Note that by applying Lemma \ref{charQuotLem}(2) with $N:=\Rad(G)$, the soluble radical of $G$, we get in particular that $\m(G/\Rad(G))\leqslant 2^{\dfrak(G)}+\dfrak(G)$. Since we want to bound the index $|G:\Rad(G)|$, i.e., the order of the group $G/\Rad(G)$, in terms of $\dfrak(G)$, and since $G/\Rad(G)$ is always semisimple (i.e., has no nontrivial soluble normal subgroups, see \cite[pp.~89 and 122]{Rob96a}), our next goal will be to bound the order of a finite semisimple group $H$ in terms of $\m(H)$. Consider the following simple bound:

\begin{lemma}\label{charSubLem}
Let $G$ be a finite group, and let $N$ be a characteristic subgroup of $G$. Then $\m(N)\leqslant\m(G)$.
\end{lemma}

\begin{proof}
For each $o\in\Ord(N)$ we have $\omega_o(N)\leqslant\omega_o(G)$ (as was already observed in the proof of Lemma \ref{charQuotLem}(1)), and so
\begin{align*}
\m(G)&=\max\{\omega_o(G)\mid o\in\Ord(G)\}\geqslant\max\{\omega_o(G)\mid o\in\Ord(N)\} \\
&\geqslant\max\{\omega_o(N)\mid o\in\Ord(N)\}=\m(N),
\end{align*}
as required.
\end{proof}

By Lemma \ref{charSubLem}, $\m(\Soc(H))\leqslant\m(H)$. Hence if we can bound $|\Soc(H)|$ by a monotonically increasing function in $\m(\Soc(H))$, then $|\Soc(H)|$ is also bounded in terms of $\m(H)$, and this implies that $|H|$ is bounded in terms of $\m(H)$, because $H$ embeds into $\Aut(\Soc(H))$ (see e.g.~\cite[Lemma 1.1]{Ros75a}). Since we know by \cite[3.3.18, p.~89]{Rob96a} that $\Soc(H)$ is isomorphic to a direct product of nonabelian finite simple groups, the following will be useful:

\begin{lemma}\label{socLem}
Let $S_1,\ldots,S_r$ be pairwise nonisomorphic nonabelian finite simple groups, and let $n_1,\ldots,n_r\in\IN^+$. Then $\m(S_1^{n_1}\times\cdots\times S_r^{n_r})\geqslant\prod_{i=1}^r{(n_i\cdot\m(S_i))}$.
\end{lemma}

\begin{proof}
First note that if two elements in $S_i^{n_i}$ have a different number of nontrivial entries, then they lie in different orbits of $\Aut(S_i^{n_i})=\Aut(S_i)\wr\Sym(n_i)$. Thus for each element order $o_i$ of $S_i$, there are at least $n_i\cdot\omega_{o_i}(S_i)$ many $\Aut(S_i^{n_i})$-orbits on the set of elements of $S_i^{n_i}$ of order $o_i$. Now for each $i\in\{1,\ldots,r\}$, let $o_i$ be an element order of $S_i$ such that $\omega_{o_i}(S_i)$ is as large as possible, that is, $\omega_{o_i}(S_i)=\m(S_i)$. Observe that
\[
\Aut(S_1^{n_1}\times\cdots\times S_r^{n_r})=\Aut(S_1^{n_1})\times\cdots\times\Aut(S_r^{n_r}),
\]
and so if there is an $i\in\{1,\ldots,r\}$ such that the projections of two elements $g,h\in S_1^{n_1}\times\cdots\times S_r^{n_r}$ to the $i$-th component $S_i^{n_i}$ lie in different $\Aut(S_i^{n_i})$-orbits, then $g,h$ lie in different $\Aut(S_1^{n_1}\times\cdots\times S_r^{n_r})$-orbits. Thus letting $o:=\lcm(o_1,\ldots,o_r)$, we see that $o\in\Ord(S_1^{n_1}\times\cdots\times S_r^{n_r})$ and
\[
\omega_o(S_1^{n_1}\times\cdots\times S_r^{n_r})\geqslant\prod_{i=1}^r{(n_i\m(S_i))}.
\]
Hence the lower bound in the statement holds.
\end{proof}

We are now ready for the

\begin{proof}[Proof of Theorem \ref{mainTheo1}(1)]
Let $G$ be an arbitrary finite group.  By Lemma \ref{charQuotLem}(2), applied with $N:=\Rad(G)$, we find that $\m(G/\Rad(G))\leqslant 2^{\dfrak(G)}+\dfrak(G)$. Set $H:=G/\Rad(G)$. Write $\Soc(H)=S_1^{n_1}\times\cdots\times S_r^{n_r}$ where $S_1,\ldots,S_r$ are pairwise nonisomorphic nonabelian finite simple groups and $n_1,\ldots,n_r\in\IN^+$. Then by combining the above and Lemmas \ref{charSubLem} and \ref{socLem}, we get
\begin{equation}\label{firstIneqChain}
2^{\dfrak(G)}+\dfrak(G)\geqslant\m(H)\geqslant\m(\Soc(H))\geqslant\prod_{i=1}^r{(n_i\cdot\m(S_i))}\geqslant\max\{n_i,\m(S_i)\mid i=1,\ldots,r\}.
\end{equation}
Hence for each $i\in\{1,\ldots,r\}$, we have
\begin{equation}\label{niEq}
n_i\leqslant 2^{\dfrak(G)}+\dfrak(G),
\end{equation}
and, setting $c':=\frac{\log\log{(413/73)}}{\log\log{|\M|}}\approx0.11404$ as in Theorem \ref{mainTheo2}(5), where $\M$ is the Fischer-Griess Monster group, we have
\begin{equation}\label{ineqChain}
\exp(\log^{c'}{|S_i|})-3\leqslant\q(S_i)\leqslant\m(S_i)\leqslant\m(\Soc(H))\leqslant 2^{\dfrak(G)}+\dfrak(G).
\end{equation}
Indeed, for the first inequality in Formula (\ref{ineqChain}), note that by Theorem \ref{mainTheo2}(5),
\[
\epsilon_{\q}(S_i)\geqslant\epsilon_{\q}(\M)=c',
\]
where $\epsilon_{\q}(S)$ is as defined in Formula (\ref{qsPlus3Eq}). Hence, using the said definition of $\epsilon_{\q}$,
\[
\frac{\log\log{(\q(S_i)+3)}}{\log\log{|S_i|}}\geqslant c',
\]
or equivalently,
\[
\log\log{(\q(S_i)+3)}\geqslant c'\log\log{|S_i|},
\]
and by applying $\exp$ to both sides twice, one obtains the first inequality in Formula (\ref{ineqChain}). The second inequality in Formula (\ref{ineqChain}) follows from the definitions of $\q$ and $\m$, see Definition \ref{mainDef1}(4) and also the first sentence after Definition \ref{mainDef1}(4). The third inequality in Formula (\ref{ineqChain}) is by Lemma \ref{socLem}, and the last inequality in Formula (\ref{ineqChain}) follows from the first two inequalities in Formula (\ref{firstIneqChain}).

Using Formula (\ref{ineqChain}), and noting that the value of $c$ from the statement of Theorem \ref{mainTheo1} is just $1/c'$, we conclude that
\begin{equation}\label{siEq}
|S_i|\leqslant\exp(\log^c{(2^{\dfrak(G)}+\dfrak(G)+3)}).
\end{equation}
In view of Formula (\ref{siEq}) and Kohl's bound $|\Out(S_i)|\leqslant\log_2{|S_i|}$ from \cite{Koh03a} already used at the beginning of Subsection \ref{subsec2P3}, we deduce that
\begin{equation}\label{autsiEq}
|\Aut(S_i)|\leqslant|S_i|\cdot\log_2{|S_i|}\leqslant\exp(\log^c{(2^{\dfrak(G)}+\dfrak(G)+3)})\cdot\frac{\log^c{(2^{\dfrak(G)}+\dfrak(G)+3)}}{\log{2}}
\end{equation}
Combining Formulas (\ref{niEq}) and (\ref{autsiEq}), we obtain, still for all $i=1,\ldots,r$,
\begin{align}\label{autsiniEq}
\notag &|\Aut(S_i)^{n_i}|\leqslant \\
&\exp((2^{\dfrak(G)}+\dfrak(G))\log^c{(2^{\dfrak(G)}+\dfrak(G)+3)})\cdot(\log^{-1}{2}\cdot(2^{\dfrak(G)}+\dfrak(G)+3))^{2^{\dfrak(G)}+\dfrak(G)}.
\end{align}
Recall from above that $S_1,\ldots,S_r$ are pairwise nonisomorphic nonabelian finite simple groups. For each $m\in\IN^+$, there are at most $m$ isomorphism types of nonabelian finite simple groups of order at most $m$, because all nonabelian finite simple groups are of even order, and for each given $k\in\IN^+$, there are at most two nonisomorphic nonabelian finite simple groups of order $k$. In particular, in view of Formula (\ref{siEq}), we have
\begin{equation}\label{rEq}
r\leqslant\exp(\log^c{(2^{\dfrak(G)}+\dfrak(G)+3)}).
\end{equation}
Formulas (\ref{autsiniEq}) and (\ref{rEq}) yield
\begin{align}\label{permPartCompEq}
\notag &|H\cap(\Aut(S_1)^{n_1}\times\cdots\times\Aut(S_r)^{n_r})|\leqslant|\Aut(S_1)^{n_1}\times\cdots\times\Aut(S_r)^{n_r}|\leqslant \\
\notag &\exp((2^{\dfrak(G)}+\dfrak(G))\log^c{(2^{\dfrak(G)}+\dfrak(G)+3)}\exp(\log^c{(2^{\dfrak(G)}+\dfrak(G)+3)})) \\
&\cdot(\log^{-1}{2}\cdot(2^{\dfrak(G)}+\dfrak(G)+3))^{(2^{\dfrak(G)}+\dfrak(G))\exp(\log^c{(2^{\dfrak(G)}+\dfrak(G)+3)})}.
\end{align}
Moreover, since $H/(H\cap(\Aut(S_1)^{n_1}\times\cdots\times\Aut(S_r)^{n_r}))$ embeds into $\Sym(n_1)\times\cdots\times\Sym(n_r)$, Formulas (\ref{niEq}) and (\ref{rEq}) imply that
\begin{equation}\label{permPartEq}
|H:(H\cap(\Aut(S_1)^{n_1}\times\cdots\times\Aut(S_r)^{n_r}))|\leqslant((2^{\dfrak(G)}+\dfrak(G))!)^{\exp(\log^c{(2^{\dfrak(G)}+\dfrak(G)+3)})}.
\end{equation}
Together, Formulas (\ref{permPartCompEq}) and (\ref{permPartEq}) yield that
\begin{align*}
&|G:\Rad(G)|=|H|\leqslant \\
&\exp((2^{\dfrak(G)}+\dfrak(G))\log^c{(2^{\dfrak(G)}+\dfrak(G)+3)}\exp(\log^c{(2^{\dfrak(G)}+\dfrak(G)+3)})) \\
&\cdot(\log^{-1}{2}\cdot(2^{\dfrak(G)}+\dfrak(G)+3))^{(2^{\dfrak(G)}+\dfrak(G))\exp(\log^c{(2^{\dfrak(G)}+\dfrak(G)+3)})} \\
&\cdot((2^{\dfrak(G)}+\dfrak(G))!)^{\exp(\log^c{(2^{\dfrak(G)}+\dfrak(G)+3)})},
\end{align*}
which is what we needed to show.
\end{proof}

\numberwithin{equation}{subsection}
\section{Proof of Theorem \ref{mainTheo1}(2)}\label{sec4}

\subsection{Reduction to semisimple groups}\label{subsec4P1}

We first make the following observation, which allows us to restrict our attention to finite \emph{semisimple} groups (recall that these are by definition groups without nontrivial soluble normal subgroups, or, equivalently, with trivial soluble radical, see \cite[pp.~89 and 122]{Rob96a}):

\begin{remmark}\label{semisimpleRem}
We claim that the following are equivalent:
\begin{enumerate}
\item The existence of a function $f_2:\left[0,\infty\right)^2\rightarrow\left[1,\infty\right)$ that is monotonically increasing in both variables and such that $|G:\Rad(G)|\leqslant f_2(\q(G),\omicron(\Rad(G)))$ for all finite groups $G$, as asserted by Theorem \ref{mainTheo1}(2).
\item The existence of a monotonically increasing function $\g:\left[1,\infty\right)\rightarrow\left[1,\infty\right)$ such that $|H|\leqslant \g(\q(H))$ for all finite semisimple groups $H$.
\end{enumerate}
Indeed, assuming the first statement and aiming at deriving the second, just observe that for each finite semisimple group $H$, since $\Rad(H)=\{1_H\}$,
\[
|H|=|H:\Rad(H)|\leqslant f_2(\q(H),\omicron(\Rad(H)))=f_2(\q(H),1),
\]
so one may choose $\g(x):=f_2(x,1)$ in the second statement.

On the other hand, assuming the second statement, we can infer the first as follows: Let $G$ be an arbitrary finite group. Observe that
\[
\omicron(G)\leqslant\omicron(\Rad(G))\cdot\omicron(G/\Rad(G)),
\]
and
\[
\omega(G)\geqslant\omega(G/\Rad(G)).
\]
It follows that
\[
\q(G)=\frac{\omega(G)}{\omicron(G)}\geqslant\frac{\omega(G/\Rad(G))}{\omicron(\Rad(G))\omicron(G/\Rad(G))}=\frac{\q(G/\Rad(G))}{\omicron(\Rad(G))},
\]
or equivalently,
\[
\q(G/\Rad(G))\leqslant\q(G)\cdot\omicron(\Rad(G)).
\]
Applying the assumed second statement with $H:=G/\Rad(G)$ (which is semisimple, as noted in \cite[p.~122]{Rob96a}), we get that 
\[
|G:\Rad(G)|=|G/\Rad(G)|\leqslant\g(\q(G/\Rad(G)))\leqslant\g(\q(G)\cdot\omicron(\Rad(G))).
\]
Hence (and since $\min\{\q(G),\omicron(\Rad(G))\}\geqslant1$ for all finite groups $G$),
\[
f_2(x,y):=\begin{cases}1, & \text{if }\min\{x,y\}<1, \\ \g(x\cdot y), & \text{if }\min\{x,y\}\geqslant1\end{cases}
\]
is a suitable choice for the function in the first statement. This proves the claim.
\end{remmark}

In the rest of this section, we will be concerned with proving the second statement in Remark \ref{semisimpleRem}, so we will primarily be concerned with finite \emph{semisimple} groups only.

\subsection{Two lemmas for working with partitions}\label{subsec4P2}

Given a finite group $G$, rather than determining $\omega(G)$ and bounding $\omicron(G)$ directly, it may be easier to determine corresponding parameters for each subset $M\subseteq G$ belonging to a suitable, fixed partition $\P$ of $G$ (i.e., to a family of \emph{nonempty}, pairwise disjoint subsets of $G$ that cover $G$). In this subsection, we present two simple, but important lemmas for deriving information on $\q(G)$ using such an approach. First, we extend the notations $\omega(G)$ and $\omicron(G)$ to subsets of $G$:

\begin{nottation}\label{subsetNot}
Let $G$ be a finite group and $M\subseteq G$.

\begin{enumerate}
\item We denote by $\omega_G(M)$ the number of $\Aut(G)$-orbits on $G$ whose intersection with $M$ is nonempty.
\item We denote by $\Ord_G(M)$ (or simply $\Ord(M)$, see below) the set of distinct orders of elements of $M$ and we define $\omicron_G(M)$ (also usually simplified to $\omicron(M)$, see below) as $|\Ord_G(M)|$.
\item We set $\q_G(M):=\frac{\omega_G(M)}{\omicron_G(M)}$.
\end{enumerate}
\end{nottation}

We note that while the concept $\Ord_G(M)$ (and, likewise, $\omicron_G(M)$) does depend on $G$ to the extent that $G$ provides the algebraic structure (which $M$ itself, being only a set, is lacking) to make talking about the \enquote{order of an element of $M$} meaningful, it does have the property that if $M\subseteq G_1\leqslant G_2$ for a finite group $G_2$, then $\Ord_{G_1}(M)=\Ord_{G_2}(M)$. So as long as the context of discussion provides a \enquote{natural} smallest finite group into which the given finite set $M$ embeds (which will always be the case in our paper), we can and will omit the subscript $G$ in $\Ord_G(M)$ and $\omicron_G(M)$. On the other hand, the subscript $G$ will \emph{always} be included in the notations $\omega_G(M)$ and $\q_G(M)$ for the sake of necessity.

\begin{lemmma}\label{partitionLem}
Let $G$ be a finite group, let $M\subseteq G$, and let $\P$ be a partition of $M$ into $\Aut(G)$-invariant subsets. Then $\q_G(M)\geqslant\min\{\q_G(N)\mid N\in\P\}$.
\end{lemmma}

\begin{proof}
As each $N\in\P$ is $\Aut(G)$-invariant, we have $\omega_G(M)=\sum_{N\in\P}{\omega_G(N)}$. Moreover, $\omicron(M)\leqslant\sum_{N\in\P}{\omicron(N)}$, since the $N\in\P$ cover $M$. Setting $c:=\min\{\q_G(N)\mid N\in\P\}$, it follows that
\[
\q_G(M)=\frac{\omega_G(M)}{\omicron(M)}\geqslant\frac{\sum_{N\in\P}{\omega_G(N)}}{\sum_{N\in\P}{\omicron(N)}}\geqslant\frac{\sum_{N\in\P}{c\omicron(N)}}{\sum_{N\in\P}{\omicron(N)}}=c,
\]
as required.
\end{proof}

We will be applying Lemma \ref{partitionLem} in the special case where $G$ is a finite semisimple group. With $M:=G$, Lemma \ref{partitionLem} says in particular that if we can find a partition of $G$ into $\Aut(G)$-invariant subsets each of which has \enquote{large} $\q_G$-value, then $\q_G(G)=\q(G)$ will be large. However, sometimes it is easier to consider partitions where not every partition member has large $q_G$-value, forcing us to distinguish between \enquote{good} and \enquote{bad} partition members. The following lemma basically says that as long as the total number of element orders in the \enquote{bad} partition members  is suitably bounded from above, one may still produce a useful lower bound on $\q(G)$ from such a \enquote{mixed} partition:

\begin{lemmma}\label{mixedPartitionLem}
Let $G$ be a finite group, let $M\subseteq G$, and let $\P=\{M_{\good},M_{\bad}\}$ be a partition of $M$ into two distinct (nonempty) $\Aut(G)$-invariant subsets. Then
\[
\q_G(M)\geqslant\frac{\q_G(M_{\good})}{1+\omicron(M_{\bad})}.
\]
\end{lemmma}

\begin{proof}
Since $M_{\good}\not=\varnothing$, we have $\omicron(M_{\good})\geqslant1$, and therefore
\[
\omicron(M)\leqslant\omicron(M_{\good})+\omicron(M_{\bad})\leqslant(1+\omicron(M_{\bad}))\omicron(M_{\good}).
\]
Furthermore,
\[
\omega_G(M)=\omega_G(M_{\good})+\omega_G(M_{\bad})\geqslant\omega_G(M_{\good}),
\]
and so
\[
\q_G(M)=\frac{\omega_G(M)}{\omicron(M)}\geqslant\frac{\omega_G(M_{\good})}{(1+\omicron(M_{\bad}))\omicron(M_{\good})}=\frac{\q_G(M_{\good})}{1+\omicron(M_{\bad})}.\qedhere
\]
\end{proof}

When using Lemma \ref{partitionLem} to study $\q(H)$ for finite semisimple groups $H$, an important partition of $H$ to consider is $\P_H$, defined as follows: Say $\Soc(H)=S_1^{n_1}\times\cdots\times S_r^{n_r}$ where $S_1,\ldots,S_r$ are pairwise nonisomorphic nonabelian finite simple groups and $n_1,\ldots,n_r\in\IN^+$ (see e.g.~\cite[3.3.18, p.~89]{Rob96a}). Note that $H$ may be viewed (via its conjugation action on $\Soc(H)$) as a subgroup of
\[
\Aut(\Soc(H))=(\Aut(S_1)\wr\Sym(n_1))\times\cdots\times(\Aut(S_r)\wr\Sym(n_r)),
\]
so that each coset of $\Soc(H)$ in $H$ can be written as $\Soc(H)\alphabv\psiv$ where $\alphabv\in\Aut(S_1)^{n_1}\times\cdots\times\Aut(S_r)^{n_r}$ and $\psiv\in\Sym(n_1)\times\cdots\times\Sym(n_r)$. Using these notational conventions, we set
\begin{align*}
\P_H:=\{(\Soc(H)\alphabv\psiv)^{\Aut(H)}\mid &\alphabv\in\Aut(S_1)^{n_1}\times\cdots\times\Aut(S_r)^{n_r}, \\
&\psiv\in\Sym(n_1)\times\cdots\times\Sym(n_r), \\
&\alphabv\psiv\in H\}.
\end{align*}
Equivalently, $\P_H$ is the unique finest partition of $H$ into subsets that are both $\Aut(H)$-invariant and unions of cosets of $\Soc(H)$. By applying Lemma \ref{partitionLem} with $G:=H$ and $\P:=\P_H$, we will be able to show that $\q(H)$ is large if $\max\{\tilde{\q}(S_i)\mid i=1,\ldots,r\}$ is large (see Lemma \ref{qTildeLem}(5)), where $\tilde{\q}(S)$ is a certain parameter associated with each nonabelian finite simple group $S$, which will be introduced and studied in the next subsection.

\subsection{Some auxiliary results on finite simple groups}\label{subsec4P3}

We begin with the following definition, most of which is taken from \cite[Definition 2.2.1]{Bor19a}:

\begin{deffinition}\label{sTypeDef}
Let $S$ be a nonabelian finite simple group, and let $\pi:\Aut(S)\rightarrow\Out(S)$ be the canonical projection.
\begin{enumerate}
\item The term \emph{$S$-type} is a synonym for \enquote{$\Out(S)$-conjugacy class}.
\item For each $\Aut(S)$-conjugacy class $c$, we call the element-wise image of $c$ under $\pi$ the \emph{$S$-type of $c$}.
\item For each $\alpha\in\Aut(S)$, the \emph{$S$-type of $\alpha$} is defined as the $S$-type of $\alpha^{\Aut(S)}$.
\end{enumerate}
\end{deffinition}

$S$-types played an important role in the first author's result \cite[Lemma 2.2.5(2)]{Bor19a}, which gave an upper bound on the size of a conjugacy class in a finite semisimple group and which is based on James and Kerber's characterisation of conjugacy in wreath products of the form $G\wr\Sym(n)$ \cite[Theorem 4.2.8, p.~141]{JK81a}. Likewise, we will use James and Kerber's result to give, for each finite semisimple group $H$ and each coset $C$ of $\Soc(H)$ in $H$, bounds on $\omega_H(C)$, $\omicron(C)$ and $\q_H(C)$, see Lemma \ref{qTildeLem}(2,4,5). These bounds will also involve $S$-types, and one of the (lower) bounds on $\q_H(C)$ from Lemma \ref{qTildeLem}(5) will also involve the parameters $\tilde{q}(S)$ for nonabelian finite simple groups $S$, defined as follows:

\begin{nottation}\label{tildeNot}
Let $S$ be a nonabelian finite simple group.
\begin{enumerate}
\item We set $\tilde{\omega}(S):=\min\{\omega_{\Aut(S)}(S\alpha)\mid \alpha\in\Aut(S)\}$ and $\tilde{\q}(S):=\min\{\q_{\Aut(S)}(S\alpha)\mid \alpha\in\Aut(S)\}$.
\item For each $S$-type $\tau$, denote by $\alpha(\tau)$ some fixed automorphism of $S$ such that $\alpha(\tau)^{\Aut(S)}$ has $S$-type $\tau$, and set $\omega(\tau):=\omega_{\Aut(S)}(S\alpha(\tau))$ and $\omicron(\tau):=\omicron(S\alpha(\tau))$ (note that $\omega(\tau)$ and $\omicron(\tau)$ do not depend on the choice of $\alpha(\tau)$).
\end{enumerate}
\end{nottation}

For later reference, we note the following:

\begin{lemmma}\label{omegaTildeLem}
For every nonabelian finite simple group $S$, $\tilde{\omega}(S)\geqslant 2$.
\end{lemmma}

\begin{proof}
By \cite[Lemma 2.4.2(1)]{Bor19a}, for each $\alpha\in\Aut(S)$, the size of the intersection of $S\alpha$ with any $\Aut(S)$-conjugacy class is at most $\frac{18}{19}|S|$; in particular, $S\alpha$ is never fully contained in a single $\Aut(S)$-conjugacy class. Hence $\omega_{\Aut(S)}(S\alpha)\geqslant 2$ for all $\alpha\in\Aut(S)$, which by definition of $\tilde{\omega}(S)$ entails that $\tilde{\omega}(S)\geqslant 2$.
\end{proof}

We note that the bound in Lemma \ref{omegaTildeLem} is optimal, as
\[
\tilde{\omega}(\Alt(6))=\omega_{\Aut(\Alt(6))}(\M_{10}\setminus\Alt(6))=2.
\]
As noted above, the parameter $\tilde{q}(S)$ from Notation \ref{tildeNot} will appear in a lower bound in Lemma \ref{qTildeLem}(5), and thus we will be interested in knowing for which nonabelian finite simple groups $S$ this parameter is large. To state a corresponding asymptotic result (see Lemma \ref{qTildeLem2} below), we need some more preparation, including the following notation, which is motivated by \cite[Propositions 4.1 and 4.2]{Har92a} and part of which already appeared in \cite{Har92a}:

\begin{nottation}\label{gNot}
Let $S$ be a nonabelian finite simple group, and let $\alpha\in\Aut(S)$. We introduce the following numerical parameters $f(S)$ and $g(\alpha)$:
\begin{enumerate}
\item If $S$ is isomorphic to some alternating or sporadic finite simple group, we set $f(S):=1$ and $g(\alpha):=1$.
\item If $S$ is not isomorphic to any alternating or sporadic finite simple group, then $S$ is in particular of Lie type, so as in Section \ref{secNot}, we can write $S=\O^{p'}(\overline{S}_{\sigma})$ where $\overline{S}=X_d(\overline{\IF_p})$ is a simple linear algebraic group of adjoint type and $\sigma$ is a Lang-Steinberg endomorphism of $\overline{S}$. We then set $f(S):=6f(\sigma)$, where $f(\sigma)$ is as in the paragraph on simple Lie type groups in Section \ref{secNot} (i.e., $f(\sigma)$ is the $f$ in the notation $S=\leftidx{^t}X_d(p^{ft})$). As for $g(\alpha)$:
\begin{enumerate}
\item Assume that $X_d\notin\{B_2,F_4,G_2\}$. Then, as explained at the end of Section \ref{secNot}, we can write $\alpha=s\phi\delta$ where $s$ is the inner diagonal, $\phi$ is the field and $\delta$ is the graph component of $\alpha$, and we set $g(\alpha):=\ord(\phi)$.
\item Assume that $X_d\in\{B_2,F_4,G_2\}$. Then we can write $\alpha=s\phi$ where $s$ is the inner diagonal and $\phi$ is the graph-field component of $\alpha$, and we set $g(\alpha):=\ord(\phi)$.
\end{enumerate}
\end{enumerate}
Moreover, for each $S$-type $\tau$, we set $g(\tau):=g(\alpha(\tau))$ (which is independent of the choice of $\alpha(\tau)$ as in Notation \ref{tildeNot}(2)).
\end{nottation}

Note that by definition, $f(S)$ and $g(\alpha)$ are always positive integers, and one has that $g(\alpha)\mid f(S)$. The following lemma provides some restrictions, in terms of $g(\alpha)$, on the possible orders of elements of a coset $S\alpha$ where $S$ is a finite simple group of Lie type and $\alpha\in\Aut(S)$ (this will be useful for studying $\tilde{q}(S)$):

\begin{lemmma}\label{gAlphaLem}
Let $S=\leftidx{^t}X_d(p^{ft})$ be a finite simple group of Lie type, and let $\alpha\in\Aut(S)$. The following hold:
\begin{enumerate}
\item If $X_d\in\{B_2,F_4,G_2\}$, then
\[
\Ord(\Inndiag(S)\alpha)\subseteq g(\alpha)\cdot\Ord(\Inndiag(\leftidx{^t}X_d(p^{(f/g(\alpha))t}))).
\]
\item If $X_d\notin\{B_2,F_4,G_2\}$, then
\[
\Ord(\Inndiag(S)\alpha)\subseteq g(\alpha)\cdot\Ord(\Inndiag(\leftidx{^u}X_d(p^{(f/g(\alpha))\cdot v}))),
\]
where, denoting by $t'$ the order of the graph component of $\alpha$,
\[
(u,v)=\begin{cases}(1,1), & \text{if }t=t'=1, \\ (1,1), & \text{if }t=1,t'>1,t'\nmid g(\alpha), \\ (t',t'), & \text{if }t=1,t'>1,t'\mid g(\alpha), \\ (t,t), & \text{if }t>1,t\nmid g(\alpha), \\ (1,t), & \text{if }t>1,t\mid g(\alpha).\end{cases}
\]
\end{enumerate}
\end{lemmma}

\begin{proof}
Statement (1) as well as the first four cases in statement (2) follow from \cite[Proposition 2.4.3]{Bor19a} (more precisely, the properties of the Lang-Steinberg endomorphism $\mu$ mentioned there and which fixes $\alpha^g=\alpha^{g(\alpha)}$), which is really just a more detailed version of \cite[Propositions 4.1 and 4.2]{Har92a}.

In the last case in statement (2), i.e., when $t>1$ and $t\mid g(\alpha)$, neither \cite[Proposition 2.4.3]{Bor19a} nor \cite[Propositions 4.1 and 4.2]{Har92a} explicitly mention a Lang-Steinberg endomorphism fixing $\alpha^{g(\alpha)}$ (or a suitable other power of $\alpha$), so we resort to an argument from \cite[proof of Theorem 2.16, pp.~7678f.]{GMPS15a} to deal with this case.

More precisely, since $S$ is twisted, it has no (nontrivial) graph automorphisms. In particular, $\alpha$ does not involve any (nontrivial) graph automorphisms, and so we can write an arbitrary element of $\Inndiag(S)\alpha$ as $\delta\phi^{-1}$ where $\delta\in\Inndiag(S)$ and $\phi$ is a field automorphism of $S$ of order $g(\alpha)$. By \cite[proof of Proposition 4.1, Case 4]{Har92a}, $\phi$ is the restriction to $S$ of some untwisted Lang-Steinberg endomorphism of $X_d(\overline{\IF_p})$, which we, by abuse of notation, also denote by $\phi$, and which satisfies $q(\phi)=p^{ft/g(\alpha)}$ (see the paragraph on Lie type groups in Section \ref{secNot} for the notation $q(\mu)$ where $\mu$ is a Lang-Steinberg endomorphism of a simple linear algebraic group).

By Lang's theorem, there is an $\epsilon\in X_d(\overline{\IF_p})$ such that $\epsilon\epsilon^{-\phi}=\delta$. Set $\eta:=\epsilon^{-1}(\delta\phi^{-1})^{g(\alpha)}\epsilon$, and note that
\[
(\delta\phi^{-1})^{g(\alpha)}=\delta\delta^{\phi}\cdots\delta^{\phi^{g(\alpha)-2}}\delta^{\phi^{g(\alpha)-1}},
\]
which implies that
\begin{align*}
\eta^{\phi} &=\epsilon^{-\phi}(\delta^{\phi}\delta^{\phi^2}\cdots\delta^{\phi^{g(\alpha)-1}}\delta^{\phi^{g(\alpha)}})\epsilon^{\phi} \\
&=\epsilon^{-\phi}(\delta^{\phi}\delta^{\phi^2}\cdots\delta^{\phi^{g(\alpha)-1}}\delta)\epsilon^{\phi} \\
&=(\epsilon^{-\phi}\delta^{-1})(\delta\delta^{\phi}\cdots\delta^{\phi^{g(\alpha)-1}})(\delta\epsilon^{\phi}) \\
&=\epsilon^{-1}(\delta\phi^{-1})^{g(\alpha)}\epsilon=\eta.
\end{align*}
This shows that $\eta$, which is conjugate in $X_d(\overline{\IF_p})$ to $(\delta\phi^{-1})^{g(\alpha)}$ and thus has order $\frac{1}{g(\alpha)}\ord(\delta\phi^{-1})$, lies in $(X_d(\overline{\IF_p}))_{\phi}$, which, since $\phi$ is untwisted and has $q$-value $p^{ft/g(\alpha)}$, is isomorphic to $\Inndiag(X_d(p^{tf/g(\alpha)}))$. Since $\delta\phi^{-1}$ was an arbitrary element of $\Inndiag(S)\alpha$, we are done.
\end{proof}

We will not need the full level of detail of Lemma \ref{gAlphaLem}; in fact, the following weaker version of it will suffice for our purposes:

\begin{lemmma}\label{gAlphaLem2}
Let $S=\leftidx{^t}X_d(p^{ft})$ be a finite simple group of Lie type, and let $\alpha\in\Aut(S)$. Then there exists $t'\in\{1,2,3\}$ such that
\begin{align*}
&\Ord(\Inndiag(S)\alpha)\subseteq \\
&g(\alpha)\cdot\begin{cases}\Ord(\Inndiag(\leftidx{^{t'}}X_d(p^{(f/g(\alpha))t'}))), & \text{if }t=1\text{ or }X_d\in\{B_2,F_4,G_2\}\text{ or }t\nmid g(\alpha), \\
\Ord(\Inndiag(\leftidx{^{t'}}X_d(p^{(tf/g(\alpha))t'}))), & \text{else.}\end{cases}
\end{align*}\qed
\end{lemmma}

At last, we are now able to state and prove the following asymptotic result on the parameter $\tilde{\q}(S)$ defined in Notation \ref{tildeNot}(1):

\begin{lemmma}\label{qTildeLem2}
The following hold:

\begin{enumerate}
\item As $m\to\infty$,
\begin{enumerate}
\item $\tilde{\q}(\Alt(m))\to\infty$, and
\item $\min_{\alpha\in\Aut(\Alt(m))}{\frac{\log{\omega_{\Aut(\Alt(m))}(S\alpha)}}{\log{\omicron(S\alpha)}}}\to\infty$.
\end{enumerate}
\item Let $S=\O^{p'}(X_d(\overline{\IF_p})_{\sigma})=\leftidx{^{t(\sigma)}}{X_d(p^{f(\sigma)t(\sigma)})}$ be a finite simple group of Lie type, where $\sigma$ is a Lang-Steinberg endomorphism of $X_d(\overline{\IF_p})$, and let $\alpha\in\Aut(S)$. Then as $\max\{p,d,f(\sigma)/g(\alpha)\}\to\infty$,
\begin{enumerate}
\item $\q_{\Aut(S)}(S\alpha)\to\infty$, and
\item $\frac{\log{\omega_{\Aut(S)}(S\alpha)}}{\log{(\omicron(S\alpha)+1)}}\to\infty$.
\end{enumerate}
In particular, $\tilde{\q}(S)\to\infty$ as $\max\{p,d\}\to\infty$.
\end{enumerate}
\end{lemmma}

\begin{proof}
For statement (1): We may assume throughout that $m\geqslant7$, so that $\Aut(\Alt(m))=\Sym(m)$ and there are exactly two cosets of $\Alt(m)$ in $\Aut(\Alt(m))$. Note that as $m\to\infty$,
\[
\min_{\alpha\in\Sym(m)}{\omicron(\Alt(m)\alpha)}\to\infty,
\]
and so statement (1,a) follows from statement (1,b), because statement (1,b) implies that for sufficiently large $m$ and all $\alpha\in\Sym(m)$,
\[
\omega_{\Sym(m)}(\Alt(m)\alpha)\geqslant\omicron(\Alt(m)\alpha)^2,
\]
or equivalently,
\[
\q_{\Sym(m)}(\Alt(m)\alpha)\geqslant\omicron(\Alt(m)\alpha).
\]
We will thus restrict our attention to showing statement (1,b). It is clear by Theorem \ref{mainTheo2}(3) that
\[
\frac{\log{\omega(\Alt(m))}}{\log{\omicron(\Alt(m))}}\to\infty
\]
as $m\to\infty$, which deals with the case $\alpha\in S=\Alt(m)$, so consider the nontrivial coset $\Sym(m)\setminus\Alt(m)$. Recall from Section \ref{sec2} that for a finite group $G$, $\k(G)$ denotes the number of conjugacy classes of $G$. By \cite[Formula (1.5), p.~90]{DET69a},
\[
\k(\Alt(m))\sim\frac{1}{2}\k(\Sym(m)),
\]
so for sufficiently large $m$,
\[
\omega(\Alt(m))\leqslant\k(\Alt(m))\leqslant\frac{2}{3}\k(\Sym(m))=\frac{2}{3}\omega(\Sym(m)).
\]
Hence, using Formula (\ref{partitionAsymptoticsEq}) and recalling that $p(m)$ denotes the number of ordered integer partitions of $m$,
\[
\omega_{\Sym(m)}(\Sym(m)\setminus\Alt(m))\geqslant\frac{1}{3}\omega(\Sym(m))=\frac{1}{3}p(m)\sim\frac{1}{12\sqrt{3}m}\exp(\frac{2\pi}{\sqrt{6}}\sqrt{m}).
\]
On the other hand, recalling Formula (\ref{orderAsymptoticsEq}),
\[
\omicron(\Sym(m)\setminus\Alt(m))\leqslant\omicron(\Sym(m))=\exp(\frac{2\pi}{\sqrt{6}}\sqrt{\frac{m}{\log{m}}}+\O(\frac{\sqrt{m}\log\log{m}}{\log{m}})),
\]
so clearly, $\omega_{\Sym(m)}(\Sym(m)\setminus\Alt(m))$ grows faster than any power of $\omicron(\Sym(m)\setminus\Alt(m))$, which is just what we wanted to show.

For statement (2): For a finite group $G$, denote by
\[
\MCS(G):=\min\{|\C_G(g)|\mid g\in G\}
\]
the minimum size of an element centraliser in $G$. The bounds in \cite[Proposition 2.4.4]{Bor19a} (which are based on earlier work of Hartley and Kuzucuo\u{g}lu from \cite[proof of Theorem A1, pp.~319f.]{HK91a}) together with Fulman and Guralnick's bounds from \cite[Section 6]{FG12a} are strong enough to show that $\MCS(X_d(\overline{\IF_p})_{\mu})\geqslant q(\mu)^{d/8}$ for any Lang-Steinberg endomorphism $\mu$ of $X_d(\overline{\IF_p})$. But by \cite[Propositions 4.1 and 4.2]{Har92a}, there is a Lang-Steinberg endomorphism $\mu$ on $X_d(\overline{\IF_p})$ such that $q(\mu)\geqslant q(\sigma)^{1/g(\alpha)}$ and for all $s\in S$,
\[
|\Cent_{\Aut(S)}(s\alpha)|\geqslant g(\alpha)\MCS(X_d(\overline{\IF_p})_{\mu})\geqslant g(\alpha)q(\mu)^{d/8}\geqslant g(\alpha)q(\sigma)^{d/(8g(\alpha))}=g(\alpha)p^{\frac{d}{8}\cdot\frac{f(\sigma)}{g(\alpha)}},
\]
and so
\[
|(s\alpha)^{\Aut(S)}|\leqslant\frac{|\Aut(S)|}{g(\alpha)\cdot p^{\frac{d}{8}\cdot\frac{f(\sigma)}{g(\alpha)}}}\leqslant\frac{6df(\sigma)|S|}{g(\alpha)\cdot p^{\frac{d}{8}\cdot\frac{f(\sigma)}{g(\alpha)}}}=6\frac{df(\sigma)}{g(\alpha)}p^{-\frac{d}{8}\frac{f(\sigma)}{g(\alpha)}}|S|.
\]
Using that $\Aut(S)$ is a complete group, it follows that
\[
\omega_{\Aut(S)}(S\alpha)\geqslant(6\frac{df(\sigma)}{g(\alpha)}p^{-\frac{d}{8}\frac{f(\sigma)}{g(\alpha)}})^{-1}=\frac{p^{\frac{d}{8}\cdot\frac{f(\sigma)}{g(\alpha)}}}{6\frac{df(\sigma)}{g(\alpha)}}.
\]
In particular, if $\max\{p,d,f(\sigma)/g(\alpha)\}$ is large enough, then
\[
\omega_{\Aut(S)}(S\alpha)\geqslant p^{\frac{d}{16}\cdot\frac{f(\sigma)}{g(\alpha)}},
\]
in particular $\omega_{\Aut(S)}(S\alpha)\to\infty$ as $\max\{p,d,f(\sigma)/g(\alpha)\}\to\infty$.

What about $\omicron(S\alpha)$? In view of Lemma \ref{gAlphaLem2} and Formula (\ref{lieOmicronEq}) from Subsection \ref{subsec2P3}, as $\max\{p,d,f(\sigma)/g(\alpha)\}\to\infty$,
\[
\omicron(S\alpha)\leqslant p^{\omicron(1)df(\sigma)/g(\alpha)},
\]
so $\omicron(S\alpha)$ does indeed grow more slowly than any power of $\omega_{\Aut(S)}(S\alpha)$ as
\[
\max\{p,d,f(\sigma)/g(\alpha)\}\to\infty,
\]
which is just statement (2,b). Statement (2,a) follows from this since therefore, if $\max\{p,d,f(\sigma)/g(\alpha)\}$ is large enough,
\[
\q_{\Aut(S)}(S\alpha)\geqslant\sqrt{\omega_{\Aut(S)}(S\alpha)},
\]
which also converges to $\infty$ as $\max\{p,d,f(\sigma)/g(\alpha)\}\to\infty$.
\end{proof}

We conclude this subsection with the following lemma concerning $\q$-values of direct products of nonabelian finite simple groups, which will be used in the proof of Lemma \ref{hcalLem2}:

\begin{lemmma}\label{qsnLem}
Let $S,S_1,\ldots,S_r$ be nonabelian finite simple groups, $S_i\not\cong S_j$ for $1\leqslant i<j\leqslant r$, and let $n,n_1,\ldots,n_r\in\IN^+$. Set $T:=S_1^{n_1}\times\cdots\times S_r^{n_r}$. Then:
\begin{enumerate}
\item $\q(T)\geqslant\prod_{i=1}^r{\q(S_i^{n_i})}\geqslant\max\{\q(S_i^{n_i})\mid i=1,\ldots,r\}$.
\item $\q(S^n)\geqslant\q(S)$.
\item $\q(T)\to\infty$ as $|T|\to\infty$.
\end{enumerate}
\end{lemmma}

\begin{proof}
For statement (1): Since $\Aut(T)=\Aut(S_1^{n_1})\times\cdots\times\Aut(S_r^{n_r})$, it is clear that $\omega(T)=\prod_{i=1}^r{\omega(S_i^{n_i})}$, and since every element order in $T$ is a least common multiple over an $r$-tuple of one element order choice from each $S_i^{n_i}$, it is also clear that $\omicron(T)\leqslant\prod_{i=1}^r{\omicron(S_i^{n_i})}$. From this, the first inequality follows, and the second inequality holds because every $\q$-value is at least $1$.

For statement (2): It is clear from the fact that $\Aut(S^n)=\Aut(S)\wr\Sym(n)$ that the set of $\Aut(S^n)$-orbits on $S^n$ is in bijection with the set of cardinality $n$ multisets formed from $\Aut(S)$-orbits on $S$, whose total number is by definition $\omega(S)$. Hence
\begin{equation}\label{omegasnEq}
\omega(S^n)={n+\omega(S)-1 \choose \omega(S)-1}.
\end{equation}
On the other hand, each element order in $S^n$ can be written as a least common multiple over an $n$-tuple of element orders in $S$. In particular, $\omicron(S^n)$ is bounded from above by the number of cardinality $n$ multisets formed from element orders in $S$, whose total number is by definition $\omicron(S)$. It follows that
\begin{equation}\label{omicronsnEq}
\omicron(S^n)\leqslant{n+\omicron(S)-1 \choose \omicron(S)-1}.
\end{equation}
Combining Formulas (\ref{omegasnEq}) and (\ref{omicronsnEq}), we get
\begin{align*}
\q(S^n)&=\frac{\omega(S^n)}{\omicron(S^n)}\geqslant\frac{{n+\omega(S)-1 \choose \omega(S)-1}}{{n+\omicron(S)-1 \choose \omicron(S)-1}}=\frac{\frac{(n+\omega(S)-1)!}{(\omega(S)-1)!n!}}{\frac{(n+\omicron(S)-1)!}{(\omicron(S)-1)!n!}}=\frac{(n+\omega(S)-1)!(\omicron(S)-1)!}{(n+\omicron(S)-1)!(\omega(S)-1)!} \\
&=\frac{\omega(S)+n-1}{\omicron(S)+n-1}\cdot\frac{\omega(S)+n-2}{\omicron(S)+n-2}\cdot\cdots\cdot\frac{\omega(S)+1}{\omicron(S)+1}\cdot\frac{\omega(S)!(\omicron(S)-1)!}{\omicron(S)!(\omega(S)-1)!} \\
&\geqslant 1\cdot1\cdot\cdots\cdot 1\cdot\frac{\omega(S)}{\omicron(S)}=\q(S),
\end{align*}
as required.

For statement (3): By statements (1) and (2) and Theorem \ref{mainTheo2}(5), it is clear that $\q(T)$ is large if $T$ has a large (nonabelian) composition factor, so assume that the orders of the composition factors of $T$ are bounded. Then $T$ contains some small nonabelian finite simple group, say $S$, with large multiplicity, say $n$. As already noted in the proof of statement (2), every element order in $S^n$ is a least common multiple over an $n$-tuple of element orders in $S$, whose total number is $\omicron(S)$, and so $\omicron(S^n)\leqslant 2^{\omicron(S)}$, an upper bound which does not depend on $n$. On the other hand, using Formula (\ref{omegasnEq}) and that $\omega(S)\geqslant\omicron(S)\geqslant4$ by Burnside's $p^aq^b$-theorem,
\[
\omega(S^n)={n+\omega(S)-1 \choose \omega(S)-1}\geqslant{n+3 \choose 3},
\]
so that $\q(S^n)\to\infty$ as $n\to\infty$, and thus $\q(T)\to\infty$ by statement (1).
\end{proof}

\subsection{Gaining some control over socle cosets in finite semisimple groups}\label{subsec4P4}

Recall Notation \ref{subsetNot}. The main purpose of this subsection is to obtain some bounds on the parameters $\omega_H(C)$, $\omicron(C)$ and $\q_H(C)$, where $H$ is a finite semisimple group and $C$ is a coset of $\Soc(H)$ in $H$. This is achieved via Lemma \ref{qTildeLem} below, which was already announced in Subsection \ref{subsec4P3} and which will involve the concept of an $S$-type, see Definition \ref{sTypeDef}(1). Before being able to formulate Lemma \ref{qTildeLem}, we will need to introduce quite a few more notations, see Notations \ref{lambdaNot}, \ref{gammaNot}, \ref{heavyNot} and \ref{bcpNot} below:

\begin{nottation}\label{lambdaNot}
Let $I$ be a finite set, and let $\F=(M_i)_{i\in I}$ be a family of finite subsets of $\IN^+$ indexed by the elements of $I$.

\begin{enumerate}
\item We denote by $\Lambda(\F)$ the set of all numbers of the form $\lcm_{i\in I}{a_i}$ where $a_i\in M_i$ for $i\in I$.
\item We set $\lambda(\F):=|\Lambda(\F)|$.
\end{enumerate}
\end{nottation}

For example,
\[
\Lambda((\{2,3\},\{3,4\}))=\{\lcm(2,3),\lcm(2,4),\lcm(3,3),\lcm(3,4)\}=\{3,4,6,12\},
\]
and so $\lambda(\{2,3\},\{3,4\})=4$. Note that
\[
\lambda((M_i)_{i\in I})\leqslant\prod_{i\in I}{|M_i|},
\]
and that $\Lambda((M_i)_{i\in I})=\varnothing$ if and only if at least one of the sets $M_i$ is empty; as a small technicality, we note that if $I=\varnothing$, and thus $(M_i)_{i\in I}=\varnothing$, then by definition,
\[
\Lambda((M_i)_{i\in I})=\Lambda(\varnothing)=\{\lcm(\varnothing)\}=\{1\},
\]
which is also the only way to define $\Lambda(\varnothing)$ so that Lemma \ref{qTildeLem}(3) also applies when the semisimple group $H$ in its formulation is trivial.

\begin{nottation}\label{gammaNot}
Let $n\in\IN^+$ and $\psi\in\Sym(n)$. We denote by $\Gamma(\psi)$ the number of distinct cycles of $\psi$ on $\{1,\ldots,n\}$ including fixed points.
\end{nottation}

This notation will mainly be applied to an element $\psiv=(\psi_1,\ldots,\psi_r)$ of the direct product $\prod_{i=1}^r{\Sym(n_i)}$, and in order to make sense of this, we identify the abstract direct product $\prod_{i=1}^r{\Sym(n_i)}$ with the subgroup of the symmetric group over the size $\sum_{i=1}^r{n_i}$ set
\[
\bigcup_{i=1}^r{(\{i\}\times\{1,\ldots,n_i\})}
\]
which consists of all permutations of the form $(i,j)\mapsto (i,\sigma_i(j))$ for some given element $(\sigma_1,\ldots,\sigma_r)$ of the abstract direct product $\prod_{i=1}^r{\Sym(n_i)}$. Under this identification, $\psiv$ corresponds to the permutation $(i,j)\mapsto (i,\psi_i(j))$, so that $\Gamma(\psiv)=\sum_{i=1}^r{\Gamma(\psi_i)}$ and $\ord(\psiv)=\lcm(\ord(\psi_1),\ldots,\ord(\psi_r))$. Finally, if $C=\Soc(H)\alphabv\psiv$ is a socle coset in some finite semisimple group $H$ (with notation as introduced in the paragraph after the proof of Lemma \ref{mixedPartitionLem} above), we set $\Gamma(C):=\Gamma(\psiv)$.

\begin{nottation}\label{bcpNot}
Let $\Omega$ be a finite set, let $\psi\in\Sym(\Omega)$, and let $\zeta=(\gamma_1,\ldots,\gamma_{\ell})$ be an $\ell$-cycle of $\psi$ (possibly with $\ell=1$).
\begin{enumerate}
\item We set $\supp(\zeta):=\{\gamma_1,\ldots,\gamma_{\ell}\}$.
\item Assume now additionally that $S$ is a nonabelian finite simple group and $\vec{\alpha}=(\alpha_{\gamma})_{\gamma\in\Omega}\in\Aut(S)^{\Omega}$ is a family of automorphisms of $S$ labelled by the elements of $\Omega$. For each $\gamma\in\supp(\zeta)$, we set
\[
\bcp_{\gamma}(\psi,\vec{\alpha}):=\alpha_{\gamma}\alpha_{\psi^{-1}(\gamma)}\alpha_{\psi^{-2}(\gamma)}\cdots\alpha_{\psi^{-\ell+1}(\gamma)}\in\Aut(S),
\]
the \emph{index $\gamma$ backward cycle product associated with $\psi$ and $\vec{\alpha}$}. Moreover, we set
\[
\bcpc_{\zeta}(\vec{\alpha}):=\bcp_{\gamma}(\psi,\vec{\alpha})^{\Aut(S)}
\]
for any $\gamma\in\supp(\zeta)$, the \emph{backward cycle product class associated with $\zeta$ and $\vec{\alpha}$}.
\end{enumerate}
\end{nottation}

Note that the definition of $\bcpc_{\zeta}(\vec{\alpha})$ does not depend on the choice of $\gamma\in\supp(\zeta)$, because if $\gamma,\gamma'\in\supp(\zeta)$, then $\bcp_{\gamma}(\psi,\vec{\alpha})$ and $\bcp_{\gamma'}(\psi,\vec{\alpha})$ are cyclic shifts of each other; in particular, they are conjugate in $\Aut(S)$.

The parameters introduced in the following notation all play a role in Lemma \ref{qTildeLem}:

\begin{nottation}\label{heavyNot}
Let $r\in\IN^+$, $S_1,\ldots,S_r$ be pairwise nonisomorphic nonabelian finite simple groups, let $n_1,\ldots,n_r\in\IN^+$, and let $\psiv=(\psi_1,\ldots,\psi_r)\in\Sym(n_1)\times\cdots\times\Sym(n_r)$ and $\alphabv=(\vec{\alpha_1},\ldots,\vec{\alpha_r})\in\Aut(S_1)^{n_1}\times\cdots\times\Aut(S_r)^{n_r}$. We view $\alphabv$ as an array of automorphisms of nonabelian finite simple groups whose entries are labelled by pairs $(i,j)$ with $i\in\{1,\ldots,r\}$ and $j\in\{1,\ldots,n_i\}$.

\begin{enumerate}
\item For $i=1,\ldots,r$, we denote by $\overline{\psi_i}$ the permutation of $\{i\}\times\{1,\ldots,n_i\}$ mapping $(i,j)\mapsto(i,\psi_i(j))$.
\item We say that a triple $(i,\ell,\tau)$ where $i\in\{1,\ldots,r\}$, $\ell\in\{1,\ldots,n_i\}$ and $\tau$ is an $S_i$-type is \emph{$(\psiv,\alphabv)$-admissible} if and only if for some $\ell$-cycle $\zeta$ of $\overline{\psi_i}$, $\bcpc_{\zeta}(\alphabv)$ has $S_i$-type $\tau$, and each such cycle $\zeta$ is called an \emph{$(i,\ell,\tau)$-cycle of $(\psiv,\alphabv)$}.
\item We denote by $\Adm(\psiv,\alphabv)$ the set of all $(\psiv,\alphabv)$-admissible triples.
\item Assume that $(i,\ell,\tau)\in\Adm(\psiv,\alphabv)$.
\begin{enumerate}
\item We denote by $\Gamma_{i,\ell,\tau}(\psiv,\alphabv)$ the number of $(i,\ell,\tau)$-cycles of $(\psiv,\alphabv)$.
\item We denote by $\Omega_{i,\ell,\tau}(\psiv,\alphabv)$ the number of multisets with elements from $\{1,\ldots,\omega(\tau)\}$ and with cardinality $\Gamma_{i,\ell,\tau}(\psiv,\alphabv)$, where $\omega(\tau)$ is as in Notation \ref{tildeNot}(2).
\item We denote by $\Omicron_{i,\ell,\tau}(\psiv,\alphabv)$ the number of subsets of $\{1,\ldots,\omicron(\tau)\}$ of cardinality at most $\Gamma_{i,\ell,\tau}(\psiv,\alphabv)$, where $\omicron(\tau)$ is as in Notation \ref{tildeNot}(2).
\item We denote by $\F_{i,\ell,\tau}(\psiv,\alphabv)$ the tuple of length $\Gamma_{i,\ell,\tau}(\psiv,\alphabv)$ whose entries are all equal to the set $\Ord_{\Aut(S_i)}((S_i\alpha(\tau))^{\ord(\psiv)/\ell})$ of orders of $(\ord(\psiv)/\ell)$-th powers of elements of the coset $S_i\alpha(\tau)$.
\item We set $M_{i,\ell,\tau}(\psiv,\alphabv):=\Lambda(\F_{i,\ell,\tau}(\psiv,\alphabv))$, where $\Lambda$ is as in Notation \ref{lambdaNot}.
\item We set $\G(\psiv,\alphabv):=\Lambda((M_{i,\ell,\tau}(\psiv,\alphabv))_{(i,\ell,\tau)\in\Adm(\psiv,\alphabv)})$.
\end{enumerate}
\end{enumerate}
\end{nottation}

Observe that by definition,
\begin{equation}\label{omegailtEq}
\Omega_{i,\ell,\tau}(\psiv,\alphabv)={\Gamma_{i,\ell,\tau}(\psiv,\alphabv)+\omega(\tau)-1 \choose \omega(\tau)-1},
\end{equation}
and that $\Omicron_{i,\ell,\tau}(\psiv,\alphabv)$ is bounded from above by the number of multisets with elements from $\{1,\ldots,\omicron(\tau)\}$ and with cardinality $\Gamma_{i,\ell,\tau}(\psiv,\alphabv)$, which is
\[
{\Gamma_{i,\ell,\tau}(\psiv,\alphabv)+\omicron(\tau)-1 \choose \omicron(\tau)-1}.
\]
At last, we are able to formulate and prove Lemma \ref{qTildeLem}, providing bounds on $\omega_H$-, $\omicron$- and $\q_H$-values (see Notation \ref{subsetNot}) of socle cosets in finite semisimple groups:

\begin{lemmma}\label{qTildeLem}
Let $r\in\IN^+$, $S_1,\ldots,S_r$ be pairwise nonisomorphic nonabelian finite simple groups, let $n_1,\ldots,n_r\in\IN^+$, let $H$ be a finite semisimple group with $\Soc(H)=S_1^{n_1}\times\cdots\times S_r^{n_r}$, and let $\psiv=(\psi_1,\ldots,\psi_r)\in\Sym(n_1)\times\cdots\times\Sym(n_r)$ and $\alphabv=(\vec{\alpha_1},\ldots,\vec{\alpha_r})\in\Aut(S_1)^{n_1}\times\cdots\times\Aut(S_r)^{n_r}$ be such that $\alphabv\psiv\in H$. Let $C:=\Soc(H)\alphabv\psiv$ be the associated socle coset in $H$. Then:

\begin{enumerate}
\item $\omega_H(C)=\omega_H(C^{\Aut(H)})$ and $\omicron(C)=\omicron(C^{\Aut(H)})$.
\item $\omega_H(C)\geqslant\omega_{\Aut(\Soc(H))}(C)=\prod_{(i,\ell,\tau)\in\Adm(\psiv,\alphabv)}{\Omega_{i,\ell,\tau}(\psiv,\alphabv)}$.
\item $\Ord(C)=\ord(\psiv)\cdot\G(\psiv,\alphabv)$.
\item $\omicron(C)=|\G(\psiv,\alphabv)|\leqslant\prod_{(i,\ell,\tau)\in\Adm(\psiv,\alphabv)}{\Omicron_{i,\ell,\tau}(\psiv,\alphabv)}$.
\item \begin{align*}\q_H(C)&\geqslant\q_{\Aut(\Soc(H))}(C)\geqslant\prod_{(i,\ell,\tau)\in\Adm(\psiv,\alphabv)}{\frac{\Omega_{i,\ell,\tau}(\psiv,\alphabv)}{\Omicron_{i,\ell,\tau}(\psiv,\alphabv)}} \\ &\geqslant\prod_{(i,\ell,\tau)\in\Adm(\psiv,\alphabv)}{\q_{\Aut(S_i)}(S_i\alpha(\tau))}\geqslant\max\{\tilde{\q}(S_i)\mid i=1,\ldots,r\},\end{align*} where $\alpha(\tau)$ and $\tilde{\q}(S_i)$ are as in Notation \ref{tildeNot}.
\end{enumerate}
\end{lemmma}

\begin{proof}
For statement (1): The second equality (of the $\omicron$-values) is clear since group automorphisms preserve the orders of elements. The first equality holds because each coset of $\Soc(H)$ that is contained in the union of socle cosets $(\Soc(H)\alphabv\psiv)^{\Aut(H)}$ intersects the same $\Aut(H)$-orbits, namely those that are contained in $(\Soc(H)\alphabv\psiv)^{\Aut(H)}$.

For statement (2): The inequality $\omega_H(\Soc(H)\alphabv\psiv)\geqslant\omega_{\Aut(\Soc(H))}(\Soc(H)\alphabv\psiv)$ holds because $\Aut(H)$ embeds naturally into $\Aut(\Soc(H))$ (see e.g.~\cite[Lemma 1.1]{Ros75a}), which is complete.

The asserted formula for $\omega_{\Aut(\Soc(H))}(\Soc(H)\alphabv\psiv)$ is an immediate consequence of \cite[Lemma 2.2.5(1)]{Bor19a}, which is an equivalent reformulation of James and Kerber's characterisation of conjugacy in wreath products $G\wr\Sym(n)$ \cite[Theorem 4.2.8, p.~141]{JK81a}.

For statement (3): For the proof of this statement, view $\psiv$ as a permutation on
\[
\bigcup_{i=1}^r{(\{i\}\times\{1,\ldots,n_i\})}
\]
as explained after Notation \ref{gammaNot}. Observe that each element
\[
h\alphabv\psiv\in \Soc(H)\alphabv\psiv
\]
has order divisible by $\ord(\psiv)$ and that it thus suffices to show that the set of orders of the powers
\[
(h\alphabv\psiv)^{\ord(\psiv)}\in\prod_{i=1}^r{\Aut(S_i)^{n_i}},
\]
where $h$ ranges over $\Soc(H)$, is just $\G(\psiv,\alphabv)$.

Now, for each $(\psiv,\alphabv)$-admissible triple $(i,\ell,\tau)$, for each $(i,\ell,\tau)$-cycle $\zeta$ of $(\psiv,\overline{\vec{\alpha}})$ and each $(i,j)\in\supp(\zeta)$, the $(i,j)$ entry of
\[
(h\alphabv\psiv)^{\ord(\psiv)}
\]
is the $(\ord(\psiv)/\ell)$-th power of
\[
\bcp_{(i,j)}(\psiv,\alphabv).
\]
In particular, with $h$ ranging over $\Soc(H)$, we have the following:

\begin{itemize}
\item for each given $(i,\ell,\tau)\in\Adm(\psiv,\alphabv)$, the possible orders of each entry of
\[
(h\alphabv\psiv)^{\ord(\psiv)}
\]
whose index lies on an $(i,\ell,\tau)$-cycle of $(\psiv,\alphabv)$ are just the elements of
\[
\Ord((S_i\alpha(\tau))^{\ord(\psi)/\ell};
\]
\item entries of
\[
(h\alphabv\psiv)^{\ord(\psiv)}
\]
whose indices lie on the same cycle of $\psiv$ are conjugate (in particular, of the same order); and
\item the orders of entries of
\[
(h\alphabv\psiv)^{\ord(\psiv)}
\]
whose indices $(i_k,j_k)$ lie on pairwise distinct cycles, of lengths $\ell_k$ and with backward cycle product type $\tau_k$, of $\psiv$ can be chosen independently of each other from the respective set
\[
\Ord((S_{i_k}\alpha(\tau_k))^{\ord(\psiv)/\ell_k}).
\]
\end{itemize}

Hence for each given $(i,\ell,\tau)\in\Adm(\psiv,\alphabv)$, the possible orders (in a suitable power of $\Aut(S_i)$) of the projection of
\[
(h\alphabv\psiv)^{\ord(\psiv)}
\]
to the coordinates lying on one of the $(i,\ell,\tau)$-cycles of $(\psiv,\alphabv)$ are just the least common multiples formed from $\Gamma_{i,\ell,\tau}(\psiv,\alphabv)$-tuples with entries from
\[
\Ord((S_i\alpha(\tau))^{\ord(\psiv)/\ell},
\]
i.e., the elements of $M_{i,\ell,\tau}(\psiv,\alphabv)$ by definition. Moreover, the possible orders of the entire power
\[
(h\alphabv\psiv)^{\ord(\psiv)}
\]
are just the least common multiples formed from a tuple consisting of one element choice from each set $M_{i,\ell,\tau}(\psiv,\alphabv)$ with $(i,\ell,\tau)$ ranging over $\Adm(\psiv,\alphabv)$, i.e., the elements of $\G(\psiv,\alphabv)$ by definition, as required.

For statement (4): By statement (3),
\[
\omicron(C)=|\G(\psiv,\alphabv)|\leqslant\prod_{(i,\ell,\tau)\in\Adm(\psiv,\alphabv)}{\lambda(\F_{i,\ell,\tau}(\psiv,\alphabv))}.
\]
By the definitions of the notations $\lambda$ and $\Omicron_{i,\ell,\tau}$, the result now follows.

For statement (5): The first two inequalities are consequences of statements (2) and (4), and the last inequality is clear from the definition of $\tilde{q}(S)$ (which entails that $\q_{\Aut(S)}(S\alpha)\geqslant\tilde{\q}(S)\geqslant1$ for all nonabelian finite simple groups $S$ and all $\alpha\in\Aut(S)$). Hence we may restrict our attention to the third inequality. We are done if we can show that for each $(i,\ell,\tau)\in\Adm(\psiv,\alphabv)$,
\[
\frac{\Omega_{i,\ell,\tau}(\psiv,\alphabv)}{\Omicron_{i,\ell,\tau}(\psiv,\alphabv)}\geqslant\q_{\Aut(S_i)}(S_i\alpha(\tau)).
\]
And indeed, by the remarks after Notation \ref{heavyNot} and the fact that $\omicron(\tau)\leqslant\omega(\tau)$, we conclude that
\begin{align*}
&\frac{\Omega_{i,\ell,\tau}(\psiv,\alphabv)}{\Omicron_{i,\ell,\tau}(\psiv,\alphabv)}\geqslant\frac{{\Gamma_{i,\ell,\tau}(\psiv,\alphabv)+\omega(\tau)-1 \choose \omega(\tau)-1}}{{\Gamma_{i,\ell,\tau}(\psiv,\alphabv)+\omicron(\tau)-1 \choose \omicron(\tau)-1}}=\frac{(\Gamma_{i,\ell,\tau}(\psiv,\alphabv)+\omega(\tau)-1)!(\omicron(\tau)-1)!}{(\Gamma_{i,\ell,\tau}(\psiv,\alphabv)+\omicron(\tau)-1)!(\omega(\tau)-1)!} \\
&=\frac{\omega(\tau)+\Gamma_{i,\ell,\tau}(\psiv,\alphabv)-1}{\omicron(\tau)+\Gamma_{i,\ell,\tau}(\psiv,\alphabv)-1}\cdot\frac{\omega(\tau)+\Gamma_{i,\ell,\tau}(\psiv,\alphabv)-2}{\omicron(\tau)+\Gamma_{i,\ell,\tau}(\psiv,\alphabv)-2}\cdot\cdots\cdot\frac{\omega(\tau)+1}{\omicron(\tau)+1}\cdot\frac{\omega(\tau)!(\omicron(\tau)-1)!}{\omicron(\tau)!(\omega(\tau)-1)!} \\
&\geqslant 1\cdot 1\cdot\cdots\cdot 1\cdot\frac{\omega(\tau)!(\omicron(\tau)-1)!}{\omicron(\tau)!(\omega(\tau)-1)!}=\frac{\omega(\tau)}{\omicron(\tau)}=\q_{\Aut(S_i)}(S_i\alpha(\tau)),
\end{align*}
as required.
\end{proof}

\subsection{Another equivalent reformulation of Theorem \ref{mainTheo1}(2)}\label{subsec4P5}

Consider the following notation:

\begin{nottation}\label{classNot}
Let $\hat{m},\hat{d},\hat{p},c\geqslant1$.
\begin{enumerate}
\item We denote by $\Hcal^{(c)}$ the class of finite semisimple groups $H$ with $\q(H)\leqslant c$.
\item We denote by $\Scal_{\hat{m},\hat{d},\hat{p}}$ the class of nonabelian finite simple groups $S$ which are one of the following:
\begin{itemize}
\item a sporadic nonabelian finite simple group,
\item an alternating group $\Alt(m)$ where $m\leqslant\hat{m}$, or
\item a finite simple group of Lie type $\leftidx{^t}X_d(p^{ft})$ where $d\leqslant\hat{d}$ and $p\leqslant\hat{p}$.
\end{itemize}
\item We denote by $\Hcal_{\hat{m},\hat{d},\hat{p}}$ the class of finite semisimple groups $H$ such that $\Soc(H)$ only has composition factors from $\Scal_{\hat{m},\hat{d},\hat{p}}$.
\end{enumerate}
\end{nottation}

As an extension of Remark \ref{semisimpleRem}, we note the following:

\begin{remmark}\label{classRem}
The following are equivalent:
\begin{enumerate}
\item The existence of a function $f_2:\left[0,\infty\right)^2\rightarrow\left[1,\infty\right)$ that is monotonically increasing in each variable and such that $|G:\Rad(G)|\leqslant f_2(\q(G),\omicron(\Rad(G)))$ for all finite groups $G$, as asserted by Theorem \ref{mainTheo1}(2).
\item The existence of a monotonically increasing function $\g:\left[1,\infty\right)\rightarrow\left[1,\infty\right)$ such that $|H|\leqslant \g(\q(H))$ for all finite semisimple groups $H$.
\item The finiteness (up to isomorphism of the elements) of the classes $\Hcal^{(c)}$ for all $c\geqslant1$.
\end{enumerate}
Indeed, the equivalence of (1) and (2) was already shown in Remark \ref{semisimpleRem}, and the equivalence of (2) and (3) is obvious.
\end{remmark}

We will prove Theorem \ref{mainTheo1}(2) by ultimately showing that the classes $\Hcal^{(c)}$ are all finite. As an intermediate step toward proving this, we will now show the following, as an application of the theory developed so far:

\begin{lemmma}\label{hcalLem}
For each constant $c\geqslant1$, there are constants $\hat{m}=\hat{m}(c)$, $\hat{d}=\hat{d}(c)$ and $\hat{p}=\hat{p}(c)$, all in $\left[1,\infty\right)$, such that $\Hcal^{(c)}\subseteq\Hcal_{\hat{m},\hat{d},\hat{p}}$.
\end{lemmma}

\begin{proof}
Let $H\in\Hcal^{(c)}$, i.e., $H$ is a finite semisimple group with $\q(H)\leqslant c$. Write $\Soc(H)=S_1^{n_1}\times\cdots\times S_r^{n_r}$, where $S_1,\ldots,S_r$ are pairwise nonisomorphic nonabelian finite simple groups and $n_1,\ldots,n_r\in\IN^+$. Recall the partition $\P_H$ of $H$ introduced after the proof of Lemma \ref{mixedPartitionLem}. By Lemma \ref{partitionLem}, applied with $G:=H$, $M:=H$ and $\P:=\P_H$, as well as Lemma \ref{qTildeLem}(1,5), we find that (recalling $\tilde{\q}$ from Notation \ref{tildeNot}(1))
\begin{align}\label{partApplEq}
\notag &\max\{\tilde{\q}(S_i)\mid i=1,\ldots,r\}\leqslant\min\{\q_H(C)\mid C\text{ is a socle coset in }H\} \\
&=\min\{\q_H(N)\mid N\in\P_H\}\leqslant\q(H)\leqslant c.
\end{align}
By Lemma \ref{qTildeLem2}, there are constants $\hat{m}=\hat{m}(c)$, $\hat{d}=\hat{d}(c)$ and $\hat{p}=\hat{p}(c)$ in $\left[1,\infty\right)$ such that
\begin{itemize}
\item For all $m\geqslant5$, if $\tilde{\q}(\Alt(m))\leqslant c$, then $m\leqslant\hat{m}$.
\item For all $d\geqslant1$ and all primes $p$, if $S$ is a finite simple group of Lie type of rank $d$ and definining characteristic $p$ such that $\tilde{\q}(S)\leqslant c$, then $d\leqslant\hat{d}$ and $p\leqslant\hat{p}$.
\end{itemize}
Combining this with Formula (\ref{partApplEq}) shows that each composition factor $S_i$ of $\Soc(H)$ must lie in the class $\Scal_{\hat{m},\hat{d},\hat{p}}$, defined in Notation \ref{classNot}(2) above, and this just means by definition that $H\in\Hcal_{\hat{m},\hat{d},\hat{p}}$, as required.
\end{proof}

In order to derive further restrictions on the classes $\Hcal^{(c)}$, it is, in view of Lemma \ref{hcalLem}, natural to study the classes $\Hcal_{\hat{m},\hat{d},\hat{p}}$. This will be done in Subsection \ref{subsec4P7}, but before that, we will need some elementary number-theoretic preparations, which will be carried out in Subsection \ref{subsec4P6}.

\subsection{A bit of elementary number theory}\label{subsec4P6}

We start with the following notation:

\begin{nottation}\label{tfracNot}
Let $k\in\IN^+$.
\begin{enumerate}
\item For a positive integer $m$, set $m//k:=\frac{m}{\gcd(m,k)}$.
\item For a set (or multiset) $M$ of positive integers, set $M//k:=\{m//k \mid m\in M\}$.
\end{enumerate}
\end{nottation}

This occurs naturally in the following basic group-theoretic lemma:

\begin{lemmma}\label{tfracLem}
Let $G$ be a finite group, $g\in G$, $M\subseteq G$, and $k\in\IN^+$. Denote by $M^k$ the set of $k$-th powers of elements of $M$. Then
\begin{enumerate}
\item $\ord(g^k)=\ord(g)//k$.
\item $\Ord(M^k)=\Ord(M)//k$.
\end{enumerate}\qed
\end{lemmma}

We also have a number-theoretic lemma associated with Notation \ref{tfracNot}, and with Notation \ref{lambdaNot}:

\begin{lemmma}\label{ntLem4}
Let $(M_i)_{i\in I}$ be a finite family of finite subsets of $\IN^+$, and let $k\in\IN^+$. Then
\[
\Lambda((M_i//k)_{i\in I})=\Lambda((M_i)_{i\in I})//k.
\]
\end{lemmma}

\begin{proof}
Assume w.l.o.g.~that $I=\{1,\ldots,r\}$. Note that a general element of $\Lambda((M_i//k)_{i\in I})$ is of the form
\begin{equation}\label{expr1Eq}
\lcm(\frac{m_1}{\gcd(m_1,k)},\ldots,\frac{m_r}{\gcd(m_r,k)})
\end{equation}
for $m_i\in M_i$ for $i=1,\ldots,r$, whereas a general element of $\Lambda((M_i)_{i\in I})//k$ is of the form
\begin{equation}\label{expr2Eq}
\frac{\lcm(m_1,\ldots,m_r)}{\gcd(\lcm(m_1,\ldots,m_r),k)}
\end{equation}
for $m_i\in M_i$ for $i=1,\ldots,r$. We are done if we can show that for each $(m_1,\ldots,m_r)\in\prod_{i=1}^r{M_i}$, the numbers in Formulas (\ref{expr1Eq}) and (\ref{expr2Eq}) are equal. Now, for each prime $p$, write $\nu_p(m)$ for the \emph{$p$-adic valuation of $m$}, i.e., the largest nonnegative integer $a$ such that $p^a$ divides $m$. Then
\begin{align*}
&\nu_p(\lcm(\frac{m_1}{\gcd(m_1,k)},\ldots,\frac{m_r}{\gcd(m_r,k)})) \\
&=\max\{\nu_p(m_i)-\min\{\nu_p(m_i),\nu_p(k)\}\mid i=1,\ldots,r\} \\
&=\max\{\nu_p(m_i)\mid i=1,\ldots,r\}-\min\{\max\{\nu_p(m_i)\mid i=1,\ldots,r\},\nu_p(k)\} \\
&=\nu_p(\frac{\lcm(m_1,\ldots,m_r)}{\gcd(\lcm(m_1,\ldots,m_r),k)}),
\end{align*}
where the second equality uses the monotonicity of $\max$ and $\min$, which implies that the maximum value in the second expression is assumed for those $i\in\{1,\ldots,r\}$ where $\nu_p(m_i)$ is maximal. Hence the third expression (obtained from the second by substituting $\max\{\nu_p(m_i)\mid i=1,\ldots,r\}$ into $\nu_p(m_i)$) is the said maximum value, as asserted.
\end{proof}

Note that a subset $M\subseteq\IN^+$ is equal to the set $\Div(n)$ of positive divisors of some fixed positive integer $n$ if and only if $M$ is a finite subset of $\IN^+$ that is closed under the binary operations $\gcd$ and $\lcm$ (i.e., $M$ is a sublattice of $(\IN^+,\mid)$) as well as closed under taking divisors of its elements; we will henceforth call such sets $M$ \emph{divisors sets}. For example, $\Div(12)=\{1,2,3,4,6,12\}$ is a divisors set.

\begin{lemmma}\label{ntLem1}
Let $M\subseteq\IN^+$ be a divisors set, and let $a_1,\ldots,a_k\in\IN^+$ and $m_1,\ldots,m_k\in M$. Then $\lcm(a_1m_1,\ldots,a_km_k)\in\lcm(a_1,\ldots,a_k)M$.
\end{lemmma}

\begin{proof}
For each prime $p$, we have that
\begin{align*}
&\nu_p(\lcm(a_1m_1,\ldots,a_km_k))=\max\{\nu_p(a_i)+\nu_p(m_i)\mid i=1,\ldots,k\}\leqslant \\
&\max\{\nu_p(a_i)\mid i=1,\ldots,k\}+\max\{\nu_p(m_i)\mid i=1,\ldots,k\},
\end{align*}
and thus
\[
\nu_p(\frac{\lcm(a_1m_1,\ldots,a_km_k)}{\lcm(a_1,\ldots,a_k)})\leqslant\max\{\nu_p(m_i)\mid i=1,\ldots,k\}\leqslant\max\{\nu_p(m)\mid m\in M\}.
\]
As $M$ is closed under divisibility, the quotient $\lcm(a_1m_1,\ldots,a_km_k)/\lcm(a_1,\ldots,a_k)$ is thus a product of pairwise coprime prime powers that are in $M$, and since $M$ is closed under taking finitary least common multiples, it follows that said quotient is an element of $M$, as required.
\end{proof}

\begin{lemmma}\label{ntLem2}
Let $M\subseteq\IN^+$ be a divisors set, let $t,a\in\IN^+$ and $m\in M$. Then, using Notation \ref{tfracNot}, $((mt)//a)\in(t//a)M$.
\end{lemmma}

\begin{proof}
The quotient
\[
\frac{(mt)//a}{t//a}=\frac{\frac{mt}{\gcd(mt,a)}}{\frac{t}{\gcd(t,a)}}=\frac{m\cdot\gcd(t,a)}{\gcd(mt,a)}
\]
is a positive integer and divides $m$, and thus it lies in $M$.
\end{proof}

\begin{nottation}\label{coBDNot}
Let $f\in\IN^+$ and $h\geqslant1$. We denote by $\coBD_h(f)$ the set of divisors $g$ of $f$ such that $\frac{f}{g}\leqslant h$.
\end{nottation}

\begin{lemmma}\label{ntLem3}
The following hold:
\begin{enumerate}
\item Let $f,a\in\IN^+$, $h\geqslant1$ and $g\in\coBD_h(f)$. Then $(g//a)\in\coBD_h(f//a)$.
\item Let $f_1,\ldots,f_k\in\IN^+$, $h\geqslant1$ and $g_i\in\coBD_h(f_i)$ for $i=1,\ldots,k$. Then $\lcm(g_1,\ldots,g_k)\in\coBD_{\lcm(1,\ldots,\lfloor h\rfloor)}(\lcm(f_1,\ldots,f_k))$.
\end{enumerate}
\end{lemmma}

\begin{proof}
For statement (1): Firstly, we note that
\[
\frac{f//a}{g//a}=\frac{\frac{f}{\gcd(f,a)}}{\frac{g}{\gcd(g,a)}}=\frac{f}{g}\cdot\frac{\gcd(g,a)}{\gcd(f,a)}\leqslant\frac{f}{g}\cdot 1\leqslant h.
\]
Secondly, we show that $g//a$ divides $f//a$. Indeed, for each prime $p$, we can write $\nu_p(f)=\nu_p(g)+v_p$ where $v_p\in\IN$. Then
\[
\nu_p(g//a)=\nu_p(\frac{g}{\gcd(g,a)})=\nu_p(g)-\min\{\nu_p(g),\nu_p(a)\}
\]
and
\[
\nu_p(f//a)=\nu_p(\frac{f}{\gcd(f,a)})=\nu_p(f)-\min\{\nu_p(f),\nu_p(a)\}=\nu_p(g)+v_p-\min\{\nu_p(g)+v_p,\nu_p(a)\}.
\]
Hence the inequality $\nu_p(g//a)\leqslant\nu_p(f//a)$ is equivalent to
\[
\min\{\nu_p(g)+v_p,\nu_p(a)\}\leqslant v_p+\min\{\nu_p(g),\nu_p(a)\},
\]
which holds in view of the formula $\min\{\alpha,\beta\}+\gamma=\min\{\alpha+\gamma,\beta+\gamma\}$ and the monotonicity of $\min$ in each component. Thus $(g//a)\in\coBD_h(f//a)$.

For statement (2): It is clear that $\lcm(g_1,\ldots,g_k)$ divides $\lcm(f_1,\ldots,f_k)$, so we only need to show that their quotient is bounded from above by $\lcm(1,\ldots,\lfloor h\rfloor)$. For $i=1,\ldots,k$, let us write $f_i=g_i\cdot g'_i$, where $g'_i\leqslant h$. We claim and will prove that the quotient $\lcm(f_1,\ldots,f_k)/\lcm(g'_1,\ldots,g'_k)$ divides $\lcm(g_1,\ldots,g_k)=\lcm(\frac{f_1}{g'_1},\ldots,\frac{f_k}{g'_k})$.

Indeed, for each prime $p$,
\[
\nu_p(\frac{\lcm(f_1,\ldots,f_k)}{\lcm(g'_1,\ldots,g'_k)})=\max\{\nu_p(f_i)\mid i=1,\ldots,k\}-\max\{\nu_p(g'_i)\mid i=1,\ldots,k\},
\]
and
\[
\nu_p(\lcm(g_1,\ldots,g_k))=\max\{\nu_p(f_i)-\nu_p(g'_i)\mid i=1,\ldots,k\}.
\]
Hence our claim is equivalent to
\begin{align*}
&\max\{\nu_p(f_i)\mid i=1,\ldots,k\} \\
&\leqslant\max\{\nu_p(f_i)-\nu_p(g'_i)\mid i=1,\ldots,k\}+\max\{\nu_p(g'_j)\mid j=1,\ldots,k\}= \\
&\max\{\nu_p(f_i)-\nu_p(g'_i)+\max\{\nu_p(g'_j)\mid j=1,\ldots,k\}\mid i=1,\ldots,k\},
\end{align*}
which is clearly true, thus concluding the proof of the claim. Now the claim yields in particular that
\[
\frac{\lcm(f_1,\ldots,f_k)}{\lcm(1,\ldots,\lfloor h\rfloor)}\leqslant\frac{\lcm(f_1,\ldots,f_k)}{\lcm(g'_1,\ldots,g'_k)}\leqslant\lcm(g_1,\ldots,g_k),
\]
as required.
\end{proof}

\subsection{Some results concerning the classes \texorpdfstring{$\Hcal_{\hat{m},\hat{d},\hat{p}}$}{H(m,d,p)}}\label{subsec4P7}

Recall from the end of Subsection \ref{subsec4P5} that our next goal is to study the classes $\Hcal_{\hat{m},\hat{d},\hat{p}}$ introduced in Notation \ref{classNot}(2,3). We will do so in the form of Lemmas \ref{constantLem}, \ref{qGammaLem} and \ref{aLem} below. Before formulating and proving each of them, we introduce some more notation and give some motivation:

\begin{nottation}\label{ctClNot}
We introduce the following notation:
\begin{enumerate}
\item Let $n\in\IN^+$, and let $\psi\in\Sym(n)$. We denote by $\cl(\psi)$ the set of distinct cycle lengths of $\psi$ (where fixed points count as $1$-cycles).
\item Let $r\in\IN^+$, let $n_1,\ldots,n_r\in\IN^+$, and let $\psiv=(\psi_1,\ldots,\psi_r)\in\Sym(n_1)\times\cdots\times\Sym(n_r)$. We set $\cl(\psiv):=(\cl(\psi_1),\ldots,\cl(\psi_r))$.
\item For each finite semisimple group $H$ with $\Soc(H)=S_1^{n_1}\times\cdots\times S_r^{n_r}$ where $S_1,\ldots,S_r$ are pairwise nonisomorphic nonabelian finite simple groups and $n_1,\ldots,n_r\in\IN^+$, and for each socle coset $C$ of $H$, writing $C=\Soc(H)\alphabv\psiv$, we set $\cl(C):=\cl(\psiv)$ (note that this is independent of the choice of coset representative $\alphabv\psiv$).
\end{enumerate}
\end{nottation}

For the subsequent Notation \ref{hNot}, recall Definition \ref{sTypeDef}(1) as well as Notations \ref{gNot} and \ref{heavyNot}(3).

\begin{nottation}\label{hNot}
Let $S$ be a nonabelian finite simple group, let $\alpha\in\Aut(S)$, and let $\tau$ be an $S$-type.
\begin{enumerate}
\item We set $h(\alpha):=\frac{f(S)}{g(\alpha)}$, where $f(S)$ and $g(\alpha)$ are as in Notation \ref{gNot}.
\item We set $h(\tau):=\frac{f(S)}{g(\tau)}$, where $f(S)$ and $g(\tau)$ are as in Notation \ref{gNot}.
\end{enumerate}
Moreover, for a socle coset $C=\Soc(H)\alphabv\psiv$ in a finite semisimple group $H$ with $\Soc(H)\cong S_1^{n_1}\times\cdots\times S_r^{n_r}$, we set
\[
h(C):=\max\{h(\tau)\mid (i,\ell,\tau)\in\Adm(\psiv,\alphabv)\},
\]
which is independent of the choice of coset representative $\alphabv\psiv$. For a constant $\hat{h}\geqslant1$, we say that $C$ is \emph{$\hat{h}$-large} if and only if $h(C)>\hat{h}$, and \emph{$\hat{h}$-small} otherwise.
\end{nottation}

Note that by definition, if
\begin{itemize}
\item $\hat{h}\geqslant1$ is a constant,
\item $H$ is a finite semisimple group with $\Soc(H)=S_1^{n_1}\times\cdots\times S_r^{n_r}$ where $S_1,\ldots,S_r$ are pairwise nonisomorphic nonabelian finite simple groups and $n_1,\ldots,n_r\in\IN^+$,
\item $C=\Soc(H)\alphabv\psiv$ is a socle coset in $H$, and
\item $(i,\ell,\tau)$ is a $(\psiv,\alphabv)$-admissible triple as defined in Notation \ref{heavyNot}(2),
\end{itemize}
then $h(\tau)=1\leqslant\hat{h}$ for all admissible triples $(i,\ell,\tau)$ such that $S_i$ is alternating or sporadic (see Notation \ref{gNot}(1)). Thus the assumption that $C$ be $\hat{h}$-small gives no additional restrictions. On the other hand, if $S_i$ is neither alternating nor sporadic, then $h(\tau)\leqslant\hat{h}$ is by definition (see Notation \ref{gNot}(2)) equivalent to $g(\tau)\geqslant\frac{f(S_i)}{\hat{h}}$. This, in turn, is equivalent to the assumption that the common field or graph-field automorphism part order (note the case distinction in Notation \ref{gNot}(2)) of automorphisms of $S_i$ with $S_i$-type $\tau$ is at least $\frac{f(S_i)}{\hat{h}}$, thus \enquote{close to being maximal}.

The point behind introducing the concepts of $\hat{h}$-small and $\hat{h}$-large socle cosets is the following:

\begin{lemmma}\label{hHatLargeLem}
For each constant $c\geqslant1$, there is a constant $\hat{h}=\hat{h}(c)$ such that $\q_H(C)>c$ for every finite semisimple group $H$ and all $\hat{h}$-large socle cosets $C$ in $H$.
\end{lemmma}

\begin{proof}
Write $\Soc(H)=S_1^{n_1}\times\cdots\times S_r^{n_r}$ where $S_1,\ldots,S_r$ are pairwise nonisomorphic nonabelian finite simple groups and $n_1,\ldots,n_r\in\IN^+$. If $C$ is $\hat{h}$-large (i.e., $h(C)>\hat{h}$) for some constant $\hat{h}\geqslant1$, then
\[
h(C)=\max\{h(\tau)\mid (i,\ell,\tau)\in\Adm(\psiv,\alphabv)\}
\]
is attained at some $(\psiv,\alphabv)$-admissible triple $(i_0,\ell_0,\tau_0)$ such that $S_{i_0}$ is neither alternating nor sporadic (otherwise, $h(\tau_0)=1$ by definition). In particular, $S_{i_0}=\leftidx{^t}X_d(p^{ft})$ is of Lie type, and
\[
\hat{h}<h(\tau_0)=\frac{f(S_{i_0})}{g(\tau_0)}=\frac{6f}{g(\tau_0)},
\]
or equivalently,
\begin{equation}\label{hSixthEq}
\frac{f}{g(\tau_0)}>\frac{\hat{h}}{6}.
\end{equation}
By Lemma \ref{qTildeLem2}(2), there is a constant $h'=h'(c)\geqslant1$ such that if
\begin{equation}\label{hPrimeEq}
\frac{f}{g(\tau_0)}>h'(c),
\end{equation}
then $\q_{\Aut(S_{i_0})}(S_{i_0}\alpha(\tau_0))>c$ (see also Notation \ref{tildeNot}(1)). But in view of Formula (\ref{hSixthEq}), Formula (\ref{hPrimeEq}) can be forced to be true by setting $\hat{h}(c):=6h'(c)$, and then, applying Lemma \ref{qTildeLem}(5) (see also Notation \ref{heavyNot}),
\[
\q_H(C)\geqslant\max\{\q_{\Aut(S_i)}(S_i\alpha(\tau))\mid (i,\ell,\tau)\in\Adm(\psiv,\alphabv)\}\geqslant\q_{\Aut(S_{i_0})}(S_{i_0}\alpha(\tau_0))>c,
\]
as required.
\end{proof}

So, whenever we are in a situation where we need to show that $\q(H)>c$, Lemmas \ref{partitionLem}, \ref{mixedPartitionLem} and \ref{hHatLargeLem} imply that it is only the $\hat{h}(c)$-small socle cosets that we need to worry about. Also, the following Lemma \ref{constantLem} gives us some control over the number of element orders in $\hat{h}$-small socle cosets, which is useful with regard to the nature of the bound in Lemma \ref{mixedPartitionLem}:

\begin{lemmma}\label{constantLem}
Let $\hat{m},\hat{d},\hat{p},\hat{h}\geqslant1$. There is a constant $D=D(\hat{m},\hat{d},\hat{p},\hat{h})>0$ such that for all $H\in\Hcal_{\hat{m},\hat{d},\hat{p}}$, each union $U$ of all $\hat{h}$-small socle cosets $C$ in $H$ \emph{with a fixed $\cl$-value} satisfies $\omicron(U)\leqslant D$.
\end{lemmma}

\begin{proof}
Let $H\in\Hcal_{\hat{m},\hat{d},\hat{p}}$, say with $\Soc(H)=S_1^{n_1}\times\cdots\times S_r^{n_r}$ where $S_1,\ldots,S_r$ are pairwise nonisomorphic nonabelian finite simple groups and $n_1,\ldots,n_r\in\IN^+$. Denote by $N=N(\hat{m},\hat{d},\hat{p},\hat{h})\subseteq\IN^+$ the closure under the binary operation $\lcm$ of the union of the sets $\Ord(G)$, where $G$ ranges over the following (finitely many) finite groups:
\begin{itemize}
\item the automorphism groups of the sporadic nonabelian finite simple groups,
\item the groups $\Aut(\Alt(m))$ where $5\leqslant m\leqslant\hat{m}$, and
\item the inner-diagonal automorphism groups of the finite simple groups of Lie type $\leftidx{^t}X_d(p^{ft})$ where $d\leqslant\hat{d}$, $p\leqslant\hat{p}$ and $f\leqslant\hat{h}$.
\end{itemize}
Note that $N$ is a divisors set. Consider any fixed $\hat{h}$-small socle coset $C=\Soc(H)\alphabv\psiv$. Finally, fix also a $(\psiv,\alphabv)$-admissible triple $(i,\ell,\tau)$ as defined in Notation \ref{heavyNot}(2). The proof idea is to exhibit a superset for $\Ord(C)$ which only depends on $\cl(C)$ (see Notation \ref{ctClNot}(1,3)), and for this, we will use Lemma \ref{qTildeLem}(3), which gives us an explicit description of $\Ord(C)$ in general, and we will work \enquote{bottom-up}, exhibiting suitable supersets for sets of gradually increasing complexity which occur in the construction of $\Ord(C)$.

We start by claiming that
\begin{equation}\label{cobdEq}
\Ord(S_i\alpha(\tau))\subseteq\coBD_{\hat{h}}(f(S_i))\cdot N,
\end{equation}
where $f(S_i)$ is as in Notation \ref{gNot} and $\coBD_{\hat{h}}(f(S_i))$ is the set of divisors $g$ of $f(S_i)$ such that $f(S_i)/g\leqslant\hat{h}$, as defined in Notation \ref{coBDNot}. Let us argue why Formula (\ref{cobdEq}) holds. On the one hand, if $S_i$ is sporadic or alternating, then by definition (see Notation \ref{gNot}(1)), $f(S_i)=1$, and thus
\[
\coBD_{\hat{h}}(f(S_i))=\coBD_{\hat{h}}(1)=\{1\},
\]
while also
\[
\Ord(S_i\alpha(\tau))\subseteq\Ord(\Aut(S_i))\subseteq N=\{1\}\cdot N=\coBD_{\hat{h}}(f(S_i))\cdot N
\]
by the definition of $N$. On the other hand, if $S_i$ is neither alternating nor sporadic, then $S_i=\leftidx{^t}X_d(p^{ft})$ is of Lie type with $d\leqslant\hat{d}$ and $p\leqslant\hat{p}$. By Lemma \ref{gAlphaLem2}, the order of each element of $S_i\alpha(\tau)$ is of the form $g(\tau)\cdot o$, where $o$ is an element order in a group of the form $\Inndiag(\leftidx{^{t'}}X_d(p^{(uf/g(\tau))t'}))$ for some $t'\in\{1,2,3\}$ and some $u\in\{1,t\}$. Now by assumption,
\[
\hat{h}\geqslant\frac{f(S_i)}{g(\tau)}
,
\]
and thus both $o\in N$ by definition of $N$ and $g(\tau)\in\coBD_{\hat{h}}(f(S_i))$, which concludes the proof of Formula (\ref{cobdEq}).

Formula (\ref{cobdEq}) provides us with a superset for $\Ord(S_i\alpha(\tau))$. Consider next the set $\Ord((S_i\alpha(\tau))^{\ord(\psiv)/\ell})$, of all orders of $(\ord(\psiv)/\ell)$-th powers of elements of the coset $S_i\alpha(\tau)$ (recall that we are carrying out our arguments for a fixed $(\psiv,\alphabv)$-admissible triple $(i,\ell,\tau)$). Using Lemma \ref{tfracLem}(2) and Formula (\ref{cobdEq}), we have
\begin{align*}
\Ord((S_i\alpha(\tau))^{\ord(\psiv)/\ell}) &=\Ord((S_i\alpha(\tau))//(\ord(\psiv/\ell)) \\
&\subseteq(\coBD_{\hat{h}}(f(S_i))\cdot N)//(\ord(\psiv/\ell)).
\end{align*}
Fix an $o\in\coBD_{\hat{h}}(f(S_i))\cdot N$, and write $o=nf'$ with $n\in N$ and $f'\in\coBD_{\hat{h}}(f(S_i))$. In view of Lemma \ref{ntLem2} (and using that $N$ is a divisors set), we have
\[
o//(\ord(\psiv)/\ell)=(nf')//(\ord(\psiv/\ell))\in(f'//(\ord(\psiv)/\ell))N,
\]
and by an application of Lemma \ref{ntLem3}(1),
\[
f'//(\ord(\psiv)/\ell)\in\coBD_{\hat{h}}(f(S_i)//(\ord(\psiv)/\ell)).
\]
Thus we just proved that
\begin{equation}\label{cobdEq2}
\Ord((S_i\alpha(\tau))^{\ord(\psiv)/\ell})\subseteq\coBD_{\hat{h}}(f(S_i)//(\ord(\psiv)/\ell))\cdot N.
\end{equation}
Recall, again, that we are working with a fixed $(\alphabv,\psiv)$-admissible triple $(i,\ell,\tau)$, and also recall the notation $M_{i,\ell,\tau}(\psiv,\alphabv)$ from Notation \ref{heavyNot}(4(e)), which is by definition just the set of all positive integers that can be written as a least common multiple over tuples of elements of $\Ord((S_i\alpha(\tau))^{\ord(\psiv)/\ell})$ of length $\Gamma_{i,\ell,\tau}(\psiv,\alphabv)$. By Formula (\ref{cobdEq2}) and Lemma \ref{ntLem1}, each element of $M_{i,\ell,\tau}(\psiv,\alphabv)$ can be written as the product of
\begin{itemize}
\item an element of $N$, with
\item a least common multiple over some tuple of elements from the set
\[
\coBD_{\hat{h}}(f(S_i)//(\ord(\psiv/\ell))),
\]
\end{itemize}
and by Lemma \ref{ntLem3}(2), the said least common multiple always lies in the set
\[
\coBD_{\Psi(\lfloor\hat{h}\rfloor)}(f(S_i)//(\ord(\psiv)/\ell)),
\]
where $\Psi(k):=\lcm(1,\ldots,k)$ (so that $\log{\Psi(k)}$ is the second Chebyshev function). So we have the following:
\begin{equation}\label{cobdEq3}
M_{i,\ell,\tau}(\psiv,\alphabv)\subseteq\coBD_{\Psi(\lfloor\hat{h}\rfloor)}(f(S_i)//(\ord(\psiv)/\ell))\cdot N.
\end{equation}
At last, we are ready to exhibit a suitable overset for the set $\Ord(C)$ of element orders in our socle coset $C=\Soc(H)\alphabv\psiv$. By Lemma \ref{qTildeLem}(3), each element of $\Ord(C)$ can be written as the product of
\begin{itemize}
\item the number $\ord(\psiv)$, with
\item a least common multiple over a family of numbers indexed by the $(\psiv,\alphabv)$-admissible triples $(i,\ell,\tau)$, and whose entry corresponding to $(i,\ell,\tau)$ is some choice of element from $M_{i,\ell,\tau}(\psiv,\alphabv)$.
\end{itemize}
Using this information as well as Formula (\ref{cobdEq3}) and Lemmas \ref{ntLem1} and \ref{ntLem3}(2), one can conclude (analogously to how Formula (\ref{cobdEq3}) was derived from Formula (\ref{cobdEq2})) that
\begin{equation}\label{cobdEq4}
\Ord(C)\subseteq\ord(\psiv)\cdot\coBD_{\Psi(\Psi(\lfloor\hat{h}\rfloor))}(\lcm\{f(S_i)//(\ord(\psiv)/\ell)\mid (i,\ell,\tau)\in\Adm(\psiv,\alphabv)\})\cdot N.
\end{equation}
Note that the superset in Formula (\ref{cobdEq4}) depends on $\hat{m}$, $\hat{d}$, $\hat{p}$, $\hat{h}$ and $\cl(C)$ (see Notation \ref{ctClNot}(1,3)), but not on the exact choice of $C$. It follows that
\[
D(\hat{m},\hat{d},\hat{p},\hat{h}):=\Psi(\Psi(\lfloor\hat{h}\rfloor))\cdot|N(\hat{m},\hat{d},\hat{p},\hat{h})|
\]
is a suitable choice for the constant in Lemma \ref{constantLem}.
\end{proof}

Recall the notation $\Gamma(C)$, where $C$ is a socle coset in the finite semisimple group $H$, from the paragraph after Notation \ref{gammaNot}, which just denotes the total number of cycles involved in the permutation tuple $\psiv$ in any coset representative $\alphabv\psiv$ for $C$ in $H$. Apart from information on element orders in $\hat{h}$-small socle cosets as furnished by Lemma \ref{constantLem}, we will also need one more tool to show that a given socle coset has large $\q_H$-value, namely the following:

\begin{lemmma}\label{qGammaLem}
Let $\hat{m},\hat{d},\hat{p}\geqslant1$. There is a constant $D'=D'(\hat{m},\hat{d},\hat{p})>0$ such that for all $H\in\Hcal_{\hat{m},\hat{d},\hat{p}}$ and all socle cosets $C$ of $H$, one has $\q_H(C)\geqslant\frac{1}{D'(\hat{m},\hat{d},\hat{p})}\Gamma(C)$.
\end{lemmma}

\begin{proof}
Again, let $H\in\Hcal_{\hat{m},\hat{d},\hat{p}}$, say with $\Soc(H)=S_1^{n_1}\times\cdots\times S_r^{n_r}$ where $S_1,\ldots,S_r$ are pairwise nonisomorphic nonabelian finite simple groups and $n_1,\ldots,n_r\in\IN^+$. Write $C=\Soc(H)\alphabv\psiv$. Assume that for a given $\hat{h}\geqslant1$, we partition the set $\Adm(\psiv,\alphabv)$ of $(\psiv,\alphabv)$-admissible triples $(i,\ell,\tau)$ (see Notation \ref{heavyNot}(2,3)) into two subsets (recall Notation \ref{gNot}):
\begin{itemize}
\item $\Adm_-(\psiv,\alphabv):=\{(i,\ell,\tau)\in\Adm(\psiv,\alphabv)\mid\frac{f(S_i)}{g(\tau)}\leqslant\hat{h}\}$ and
\item $\Adm_+(\psiv,\alphabv):=\{(i,\ell,\tau)\in\Adm(\psiv,\alphabv)\mid\frac{f(S_i)}{g(\tau)}>\hat{h}\}$.
\end{itemize}
For $\epsilon\in\{+,-\}$, set (see Notations \ref{lambdaNot} and \ref{heavyNot}(4(d)))
\[
M_{\epsilon}:=\Lambda((M_{i,\ell,\tau}(\psiv,\alphabv))_{(i,\ell,\tau)\in\Adm_{\epsilon}(\psiv,\alphabv)}).
\]
Let $D(\hat{m},\hat{d},\hat{p},\hat{h})$ be as in Lemma \ref{constantLem}. We first show that
\begin{equation}\label{mMinusEq}
|M_-|\leqslant D(\hat{m},\hat{d},\hat{p},\hat{h}),
\end{equation}
as follows: By omitting all coordinates belonging to an $(i,\ell,\tau)$-cycle of $(\psiv,\alphabv)$ where $(i,\ell,\tau)\in\Adm_+(\psiv,\alphabv)$, we get a (size-wise) smaller socle coset $\tilde{C}=\Soc(\tilde{H})\tilde{\alphabv}\tilde{\psiv}$ in a smaller finite semisimple group $\tilde{H}\in\Hcal_{\hat{m},\hat{d},\hat{p}}$, and $\tilde{C}$ is $\hat{h}$-small. We assume that the isomorphism types of nonabelian finite simple factors in $\Soc(\tilde{H})$ are labelled by the same indices $i\in\{1,\ldots,r\}$ as in $\Soc(H)$ above (in particular, the set of all such indices is not necessarily an initial segment of $\IN^+$). This notational convention has the advantage that we can write
\[
\Adm(\tilde{\psiv},\tilde{\alphabv})=\Adm_-(\psiv,\alphabv).
\]
Note that for each fixed $(i,\ell,\tau)\in\Adm(\psiv,\alphabv)$, we are either omitting or keeping all $(i,\ell,\tau)$-cycles of $(\psiv,\alphabv)$ in the above construction of $\tilde{C}$, and so for all $(i,\ell,\tau)\in\Adm_-(\psiv,\alphabv)$,
\[
\Gamma_{i,\ell,\tau}(\psiv,\alphabv)=\Gamma_{i,\ell,\tau}(\tilde{\psiv},\tilde{\alphabv}).
\]
Now by definition, $\F_{i,\ell,\tau}(\tilde{\psiv},\tilde{\alphabv})$ is a constant tuple of length $\Gamma_{i,\ell,\tau}(\tilde{\psiv},\tilde{\alphabv})$ whose entries are equal to $\Ord((S_i\alpha(\tau))^{\ord(\tilde{\psiv})/\ell})$. On the other hand, using Lemma \ref{tfracLem} (see also Notation \ref{tfracNot}), $\F_{i,\ell,\tau}(\psiv,\alphabv)$ is a constant tuple of length $\Gamma_{i,\ell,\tau}(\psiv,\alphabv)$ whose entries are equal to
\begin{align*}
\Ord((S_i\alpha(\tau))^{\ord(\psiv)/\ell}) \\
&=\Ord(((S_i\alpha(\tau))^{\ord(\tilde{\psiv})/\ell})^{\ord(\psiv)/\ord(\tilde{\psiv})}) \\
&=\Ord((S_i\alpha(\tau))^{\ord(\tilde{\psiv})/\ell})//(\ord(\psiv)/\ord(\tilde{\psiv})).
\end{align*}
By the definitions of $M_{i,\ell,\tau}(\psiv,\alphabv)$ and $M_{i,\ell,\tau}(\tilde{\psiv},\tilde{\alphabv})$ as well as Lemma \ref{ntLem4}, it now follows that
\[
M_{i,\ell,\tau}(\psiv,\alphabv)=M_{i,\ell,\tau}(\tilde{\psiv},\tilde{\alphabv})//(\ord(\psiv)/\ord(\tilde{\psiv})).
\]
Hence, taking the least common multiple over all $(i,\ell,\tau)\in\Adm_-(\psiv,\alphabv)$ and applying Lemma \ref{ntLem4} again,
\[
M_-=\G(\tilde{\psiv},\tilde{\alphabv})//(\ord(\vec(\psi))/\ord(\tilde{\psiv})),
\]
whence, by Lemmas \ref{qTildeLem}(3) and \ref{constantLem}, applied to the $\hat{h}$-small socle coset $\tilde{C}$ in $\tilde{H}$,
\[
|M_-|\leqslant|\G(\tilde{\psiv},\tilde{\alphabv})|=\omicron(\tilde{C})\leqslant D(\hat{m},\hat{d},\hat{p},\hat{h}),
\]
as asserted above.

Now that we have shown Formula (\ref{mMinusEq}), we note that by Lemma \ref{qTildeLem}(3), applied to $C$ in $H$,
\[
\omicron(C)=|\G(\psiv,\alphabv)|=\lambda((M_+,M_-))\leqslant|M_+|\cdot|M_-|\leqslant D(\hat{m},\hat{d},\hat{p},\hat{h})\cdot\prod_{(i,\ell,\tau)\in\Adm_+(\psiv,\alphabv)}{\Omicron_{i,\ell,\tau}(\psiv,\alphabv)},
\]
and therefore (see Notation \ref{heavyNot}(4(b,c))), using Lemma \ref{qTildeLem}(2),
\[
\q_H(C)\geqslant D(\hat{m},\hat{d},\hat{p},\hat{h})^{-1}\cdot\prod_{(i,\ell,\tau)\in\Adm_-(\psiv,\alphabv)}{\Omega_{i,\ell,\tau}(\psiv,\alphabv)}\cdot\prod_{(i,\ell,\tau)\in\Adm_+(\psiv,\alphabv)}{\frac{\Omega_{i,\ell,\tau}(\psiv,\alphabv)}{\Omicron_{i,\ell,\tau}(\psiv,\alphabv)}}.
\]
Now note that for each $(i,\ell,\tau)\in\Adm(\psiv,\alphabv)$ and using Lemma \ref{omegaTildeLem} as well as Formula (\ref{omegailtEq}), we have that
\[
\Omega_{i,\ell,\tau}(\psiv,\alphabv)={\Gamma_{i,\ell,\tau}(\psiv,\alphabv)+\omega(\tau)-1 \choose \omega(\tau)-1}\geqslant{\Gamma_{i,\ell,\tau}(\psiv,\alphabv)+2-1 \choose 2-1}=\Gamma_{i,\ell,\tau}(\psiv,\alphabv)+1.
\]
Our goal will be to show that if $\hat{h}$ is chosen large enough (relative to $\hat{m}$, $\hat{d}$ and $\hat{p}$), then
\begin{equation}\label{starEq}
\frac{\Omega_{i,\ell,\tau}(\psiv,\alphabv)}{\Omicron_{i,\ell,\tau}(\psiv,\alphabv)}\geqslant\Gamma_{i,\ell,\tau}(\psiv,\alphabv)+1,
\end{equation}
for all $(i,\ell,\tau)\in\Adm_+(\psiv,\alphabv)$, so that then
\begin{align*}
\q_H(C)&\geqslant D(\hat{m},\hat{d},\hat{p},\hat{h})^{-1}\cdot\prod_{(i,\ell,\tau)\in\Adm_-(\psiv,\alphabv)}{\Omega_{i,\ell,\tau}(\psiv,\alphabv)}\cdot\prod_{(i,\ell,\tau)\in\Adm_+(\psiv,\alphabv)}{\frac{\Omega_{i,\ell,\tau}(\psiv,\alphabv)}{\Omicron_{i,\ell,\tau}(\psiv,\alphabv)}} \\
&\geqslant D(\hat{m},\hat{d},\hat{p},\hat{h})^{-1}\cdot\prod_{(i,\ell,\tau)\in\Adm(\psiv,\alphabv)}{(\Gamma_{i,\ell,\tau}(\psiv,\overline{\vec{\alpha}})+1)} \\
&\geqslant D(\hat{m},\hat{d},\hat{p},\hat{h})^{-1}\cdot\sum_{(i,\ell,\tau)\in\Adm(\psiv,\alphabv)}{(\Gamma_{i,\ell,\tau}(\psiv,\overline{\vec{\alpha}})+1)} \\
&\geqslant D(\hat{m},\hat{d},\hat{p},\hat{h})^{-1}\cdot\Gamma(\psiv)=D(\hat{m},\hat{d},\hat{p},\hat{h})^{-1}\cdot\Gamma(C),
\end{align*}
as asserted. It remains to prove our claim that Formula (\ref{starEq}) can be made true for sufficiently large $\hat{h}$. Assume that $\hat{h}$ has been chosen so large that for all $S\in\Scal_{\hat{m},\hat{d},\hat{p}}$ and all $S$-types $\tau$ with $\frac{f(S)}{g(\tau)}>\hat{h}$, one has $\omega(\tau)\geqslant\max\{\omicron(\tau)^4,4\}$ (which is possible by Lemma \ref{qTildeLem2}). Then let $(i,\ell,\tau)\in\Adm_+(\psiv,\alphabv)$. We make a case distinction.
\begin{enumerate}
\item Case: $\Gamma_{i,\ell,\tau}(\psiv,\alphabv)\leqslant\omega(\tau)-1$. Then
\begin{align*}
&\frac{\Omega_{i,\ell,\tau}(\psiv,\alphabv)}{\Omicron_{i,\ell,\tau}(\psiv,\alphabv)}\geqslant\frac{{{\Gamma_{i,\ell,\tau}(\psiv,\alphabv)+\omega(\tau)-1} \choose {\omega(\tau)-1}}}{{{\Gamma_{i,\ell,\tau}(\psiv,\alphabv)+\omicron(\tau)-1} \choose {\omicron(\tau)-1}}}=\frac{\frac{(\Gamma_{i,\ell,\tau}(\psiv,\alphabv)+\omega(\tau)-1)!}{(\omega(\tau)-1)!\Gamma_{i,\ell,\tau}(\psiv,\alphabv)!}}{\frac{(\Gamma_{i,\ell,\tau}(\psiv,\alphabv)+\omicron(\tau)-1)!}{(\omicron(\tau)-1)!\Gamma_{i,\ell,\tau}(\psiv,\alphabv)!}}= \\
&\frac{(\Gamma_{i,\ell,\tau}(\psiv,\alphabv)+\omega(\tau)-1)!\cdot(\omicron(\tau)-1)!}{(\Gamma_{i,\ell,\tau}(\psiv,\alphabv)+\omicron(\tau)-1)!\cdot(\omega(\tau)-1)!}=\frac{\omega(\tau)}{\omicron(\tau)}\cdot\frac{\omega(\tau)+1}{\omicron(\tau)+1}\cdot\cdots\cdot\frac{\omega(\tau)+\Gamma_{i,\ell,\tau}(\psiv,\alphabv)-1}{\omicron(\tau)+\Gamma_{i,\ell,\tau}(\psiv,\alphabv)-1}.
\end{align*}
If $\Gamma_{i,\ell,\tau}(\psiv,\alphabv)\leqslant3$, then that last product of fractions is bounded from below by
\[
(\frac{\omega(\tau)+2}{\omicron(\tau)+2})^{\Gamma_{i,\ell,\tau}(\psiv,\alphabv)}\geqslant 2^{\Gamma_{i,\ell,\tau}(\psiv,\alphabv)}\geqslant\Gamma_{i,\ell,\tau}(\psiv,\alphabv)+1,
\]
where the first inequality follows from the fact that $\omega(\tau)\geqslant2\omicron(\tau)+2$, which can be deduced from our assumption $\omega(\tau)\geqslant\max\{\omicron(\tau)^4,4\}$. And if $4\leqslant\Gamma_{i,\ell,\tau}(\psiv,\alphabv)\leqslant\omega(\tau)-1$, then the product of fractions is bounded from below by
\[
(\frac{2\omega(\tau)-2}{\omicron(\tau)+\omega(\tau)-2})^{\Gamma_{i,\ell,\tau}(\psiv,\alphabv)}\geqslant 1.5^{\Gamma_{i,\ell,\tau}(\psiv,\alphabv)}\geqslant\Gamma_{i,\ell,\tau}(\psiv,\alphabv)+1,
\]
where the first inequality follows from $\omega(\tau)\geqslant3\omicron(\tau)-2$, which is another consequence of our assumption $\omega(\tau)\geqslant\max\{\omicron(\tau)^4,4\}$. This concludes the proof of Formula (\ref{starEq}) in case $\Gamma_{i,\ell,\tau}(\psiv,\alphabv)\leqslant\omega(\tau)-1$.
\item Case: $\Gamma_{i,\ell,\tau}(\psiv,\alphabv)\geqslant\omega(\tau)$. In this case, we use the trivial (by definition) upper bound $\Omicron_{i,\ell,\tau}(\psiv,\alphabv)\leqslant2^{\omicron(\tau)}$. If $\omicron(\tau)\leqslant2$, then this yields
\begin{align*}
&\frac{\Omega_{i,\ell,\tau}(\psiv,\alphabv)}{\Omicron_{i,\ell,\tau}(\psiv,\alphabv)}\geqslant\frac{{\Gamma_{i,\ell,\tau}(\psiv,\alphabv)+\omega(\tau)-1 \choose \omega(\tau)-1}}{4}\geqslant\frac{{\Gamma_{i,\ell,\tau}(\psiv,\alphabv)+2 \choose 2}}{4}= \\
&\frac{1}{8}(\Gamma_{i,\ell,\tau}(\psiv,\alphabv)+2)(\Gamma_{i,\ell,\tau}(\psiv,\alphabv)+1)\geqslant\Gamma_{i,\ell,\tau}(\psiv,\alphabv)+1,
\end{align*}
where the second inequality uses that $\omega(\tau)\geqslant4>3$. If $\omicron(\tau)\geqslant3$, then we have the following:
\begin{align*}
&\frac{\Omega_{i,\ell,\tau}(\psiv,\alphabv)}{\Omicron_{i,\ell,\tau}(\psiv,\alphabv)}\geqslant\frac{{\Gamma_{i,\ell,\tau}(\psiv,\alphabv)+\omega(\tau)-1 \choose \omega(\tau)-1}}{2^{\omicron(\tau)}}\geqslant\frac{{\Gamma_{i,\ell,\tau}(\psiv,\alphabv)+\omicron(\tau) \choose \omicron(\tau)}}{2^{\omicron(\tau)}}\geqslant\frac{(\frac{\Gamma_{i,\ell,\tau}(\psiv,\alphabv)+\omicron(\tau)}{\omicron(\tau)})^{\omicron(\tau)}}{2^{\omicron(\tau)}}= \\
&(\frac{\Gamma_{i,\ell,\tau}(\psiv,\alphabv)+\omicron(\tau)}{2\omicron(\tau)})^{\omicron(\tau)}\geqslant(\frac{\Gamma_{i,\ell,\tau}(\psiv,\alphabv)}{\omicron(\tau)^2})^{\omicron(\tau)}\geqslant\Gamma_{i,\ell,\tau}(\psiv,\alphabv)^{\omicron(\tau)/2}\geqslant\Gamma_{i,\ell,\tau}(\psiv,\alphabv)+1.
\end{align*}
Here, the last inequality in the first line is by the binomial coefficient bound ${n \choose k} \geqslant (\frac{n}{k})^k$, see e.g.~\cite[Formula (2), p.~2]{Das}. Moreover, the second inequality in the second line follows from the observation that $\Gamma_{i,\ell,\tau}(\psiv,\alphabv)\geqslant\omega(\tau)\geqslant\omicron(\tau)^4$, and thus $\omicron(\tau)^2\leqslant\Gamma_{i,\ell,\tau}(\psiv,\alphabv)^{1/2}$. Finally, the last inequality in the second line uses that $\omicron(\tau)\geqslant 3$ and $\Gamma_{i,\ell,\tau}(\psiv,\alphabv)\geqslant\omega(\tau)\geqslant4$. This concludes the proof of Formula (\ref{starEq}) in case $\Gamma_{i,\ell,\tau}(\psiv,\alphabv)\geqslant\omega(\tau)$.\qedhere
\end{enumerate}
\end{proof}

Lemmas \ref{constantLem} and \ref{qGammaLem} allow us to prove a certain technical result, Lemma \ref{aLem} below, which provides lower bounds on $\q_H$-values of certain unions of cosets of $\Soc(H)$. This will be used in the proof of Lemma \ref{coupLem} below. Before we can formulate Lemma \ref{aLem}, we need some more notation and concepts.

\begin{nottation}\label{ctNot}
For a permutation $\sigma$ on a finite set, the \emph{cycle type of $\sigma$}, denoted by $\ct(\sigma)$, is defined as the multiset of cycle lengths of $\sigma$ (including $1$). Moreover, we introduce the following notation:

\begin{enumerate}
\item Let $\delta$ be a multiset of positive integers.
\begin{enumerate}
\item We denote by $\Gamma(\delta)$ the (multiset) cardinality of $\delta$. Equivalently, $\Gamma(\delta)$ is the number of cycles of any permutation on a finite set with cycle type $\delta$.
\item We denote by $\ord(\delta)$ the least common multiple of the elements of $\delta$. Equivalently, $\ord(\delta)$ is the order of any permutation on a finite set with cycle type $\delta$.
\item For $e\in\IN^+$, we denote by $\delta^e$ the multiset which for each occurrence of an element $\ell\in\delta$ contains $\gcd(\ell,e)$ occurrences of $\ell//e=\frac{\ell}{\gcd(\ell,e)}$ (see Notation \ref{tfracNot}(1)) but nothing else. Equivalently, $\delta^e$ is the cycle type of the $e$-th power of any permutation on a finite set with cycle type $\delta$.
\end{enumerate}
\item Let $\vec{\delta}=(\delta_1,\ldots,\delta_r)$ be a tuple of multisets of positive integers.
\begin{enumerate}
\item We set $\Gamma(\vec{\delta}):=\sum_{i=1}^r{\Gamma(\delta_i)}$.
\item We set $\ord(\vec{\delta}):=\lcm\{\ord(\delta_i)\mid i=1,\ldots,r\}$.
\item We set $\vec{\delta}^e:=(\delta_1^e,\ldots,\delta_r^e)$.
\end{enumerate}
\item Let $H$ be a finite semisimple group, say with $\Soc(H)=S_1^{n_1}\times\cdots\times S_r^{n_r}$ where $S_1,\ldots,S_r$ are pairwise nonisomorphic nonabelian finite simple groups and $n_1,\ldots,n_r\in\IN^+$. Moreover, let $C=\Soc(H)\alphabv\psiv$ be a socle coset in $H$, where $\psiv=(\psi_1,\ldots,\psi_r)$. Then we set $\ct(C):=(\ct(\psi_1),\ldots,\ct(\psi_r))$, called the \emph{cycle type of $C$}, which is independent of the choice of coset representative of $C$.
\end{enumerate}
\end{nottation}

To avoid confusion among readers, let us briefly recall the different usages of the notation $\Gamma(x)$ in this paper, to which Notation \ref{ctNot} has added two:
\begin{itemize}
\item When $\psi$ is a permutation on a finite set, then $\Gamma(\psi)$ denotes the number of distinct cycles of $\psi$ including fixed points, see Notation \ref{gammaNot}. This is the most basic use of this notation, from which the others are derived.
\item When $\psiv=(\psi_1,\ldots,\psi_r)$ is a tuple of permutations on finite sets, then $\Gamma(\psiv):=\sum_{i=1}^r{\Gamma(\vec{\psi})}$, as explained in the paragraph after Notation \ref{gammaNot}.
\item When $C$ is a coset of the socle $\Soc(H)\cong S_1^{n_1}\times\cdots\times S_r^{n_r}$ of a finite semisimple group $H$, then $C=\Soc(H)\alphabv\vec{\psi}$ for some tuple $\alphabv=(\vec{\alpha_1},\ldots,\vec{\alpha_r})$ with $\vec{\alpha_i}\in\Aut(S_i)^{n_i}$ for $i=1,\ldots,r$ and for a \emph{unique} permutation tuple $\psiv=(\psi_1,\ldots,\psi_r)\in\prod_{i=1}^r{\Sym(n_i)}$, so that it makes sense to set $\Gamma(C):=\Gamma(\psiv)$ in the sense of the paragraph after Notation \ref{gammaNot}, see also that same paragraph.
\item When $\delta$ is a multiset of positive integers, then $\Gamma(\delta)$ is just $\Gamma(\psi)$ (in the sense of Notation \ref{gammaNot}) for any permutation $\psi$ of cycle type $\delta$, see Notation \ref{ctNot}(1,a).
\item When $\vec{\delta}=(\delta_1,\ldots,\delta_r)$ is a tuple of multisets of positive integers, then $\Gamma(\vec{\delta}):=\sum_{i=1}^r{\Gamma(\delta_i)}$ (in the sense of Notation \ref{ctNot}(1,a)), see Notation \ref{ctNot}(2,a).
\end{itemize}

\begin{deffinition}\label{aDef}
Let $H$ be a finite semisimple group, $A>2$ a constant.
\begin{enumerate}
\item Denote by $\CT(H)$ the set of cycle types of socle cosets of $H$.
\item Say that $\vec{\delta}\in\CT(H)$ is \emph{$A$-good} if and only if $\Gamma(\vec{\delta})>A$, and otherwise, say that $\vec{\delta}$ is \emph{$A$-bad}.
\item We denote the set of $A$-good $\vec{\delta}\in\CT(H)$ by $\CT^{(A)}_{\good}(H)$, and the set of $A$-bad $\vec{\delta}\in\CT(H)$ by $\CT^{(A)}_{\bad}(H)$.
\item We distinguish further between two kinds of $\vec{\delta}\in\CT^{(A)}_{\bad}(H)$:
\begin{enumerate}
\item $\vec{\delta}$ is called \emph{$A$-bad of the first kind} if and only if $\ord(\vec{\delta})$ is divisible by some prime strictly larger than $A$, and we denote the set of such $\vec{\delta}$ by $\CT^{(A)}_{\bad,1}(H)$.
\item $\vec{\delta}$ is called \emph{$A$-bad of the second kind} if and only if all prime divisors of $\ord(\vec{\delta})$ are at most $A$, and we denote the set of such $\vec{\delta}$ by $\CT^{(A)}_{\bad,2}(H)$.
\end{enumerate}
\item We denote by $\beta^{(A)}_H$ the function $\CT^{(A)}_{\bad,1}(H)\rightarrow\CT^{(A)}_{\good}(H)$ mapping $\vec{\delta}\mapsto\vec{\delta}^{\max\{p\in\IP\,\mid\, p\text{ divides }\ord(\vec{\delta})\}}$.
\end{enumerate}
\end{deffinition}

Concerning the function $\beta^{(A)}_H$ from Definition \ref{aDef}(5), note the following two observations:
\begin{enumerate}
\item The set $\CT(H)$ is closed under taking powers in the sense of Notation \ref{ctNot}(2(c)). This is because for each $e\in\IN^+$,
\[
\ct((\Soc(H)\alphabv\psiv)^e)=\ct(\Soc(H)\alphabv\psiv)^e.
\]
\item If $\vec{\delta}=(\delta_1,\ldots,\delta_r)\in\CT(H)$ is $A$-bad of the first kind and $p_0:=\max\{p\in\IP\,\mid\, p\text{ divides }\ord(\vec{\delta})\}$, then $p_0>A$ by definition of \enquote{$A$-bad of the first kind}. Moreover, there is an $i\in\{1,\ldots,r\}$ and an $\ell\in\delta_i$ with $p_0\mid\ell$, and so by Notation \ref{ctNot}(1(c)), $\delta_i^{p_0}$ contains at least $p_0$ occurrences of the number $\ell/p_0$, whence
\begin{equation}\label{p0Eq}
\Gamma(\vec{\delta}^{p_0})\geqslant\Gamma(\delta_i^{p_0})\geqslant p_0>A.
\end{equation}
This shows that $\vec{\delta}^{p_0}=\beta^{(A)}_H(\vec{\delta})$ is not only an element of $\CT(H)$ (as follows from the first observation), but it is also $A$-good. Hence $\beta^{(A)}_H$ indeed maps into $\CT^{(A)}_{\good}(H)$, as asserted in Definition \ref{aDef}(5).
\end{enumerate}

\begin{nottation}\label{vwNot}
Let $H$ be a finite semisimple group, let $\hat{h}\geqslant1$, and let $\vec{\delta}\in\CT(H)$.
\begin{enumerate}
\item We denote by $V_{\vec{\delta}}(H)$ the union of all socle cosets in $H$ of cycle type $\vec{\delta}$.
\item We denote by $W^{(\hat{h})}_{\vec{\delta}}(H)$ the union of all $\hat{h}$-small (see Notation \ref{hNot}) socle cosets in $H$ of cycle type $\vec{\delta}$.
\end{enumerate}
\end{nottation}

We are now ready for formulating and proving Lemma \ref{aLem}:

\begin{lemmma}\label{aLem}
Let $\hat{m},\hat{d},\hat{p},\hat{h}\geqslant1$ be constants. There is a function $g_{\hat{m},\hat{d},\hat{p},\hat{h}}:\left[1,\infty\right)\rightarrow\left[1,\infty\right)$ with $g_{\hat{m},\hat{d},\hat{p},\hat{h}}(x)\to\infty$ as $x\to\infty$ such that
\[
\q_H\left(V_{\vec{\delta}}(H)\cup\bigcup\{W^{(\vec{h})}_{\vec{\epsilon}}(H)\mid\vec{\epsilon}\in(\beta^{(A)}_H)^{-1}[\{\vec{\delta}\}]\}\right)\geqslant g_{\hat{m},\hat{d},\hat{p},\hat{h}}(A)
\]
for every constant $A>2$, for all $H\in\Hcal_{\hat{m},\hat{d},\hat{p}}$ and every $\vec{\delta}\in\CT^{(A)}_{\good}(H)$.
\end{lemmma}

\begin{proof}
For $\vec{\delta}\in\CT^{(A)}_{\good}(H)$, set
\[
\varphi(\vec{\delta})=\varphi^{(A)}_H(\vec{\delta}):=|(\beta^{(A)}_H)^{-1}[\{\vec{\delta}\}]|.
\]
We make the following two observations:
\begin{enumerate}
\item As in Formula (\ref{p0Eq}) above, for each $\vec{\delta}\in\CT^{(A)}_{\good}(H)$ and each $\vec{\epsilon}\in(\beta^{(A)}_H)^{-1}[\{\vec{\delta}\}]$, we have $\Gamma(\vec{\delta})\geqslant \max\{p\in\IP\,\mid\, p\text{ divides }\ord(\vec{\delta})\}$.
\item For fixed $\vec{\delta}$, the function that assigns to each element $\vec{\epsilon}$ of the $\beta^{(A)}_H$-fibre of $\vec{\delta}$ the largest prime divisor of $\ord(\vec{\epsilon})$ is injective. This is because of the following: For any prime $p>A$, if there is any $A$-bad cycle type $\vec{\epsilon}$ whose $p$-th power is $\vec{\delta}$, then it is the one which for each $\ell\in\{1,\ldots,n\}$ has exactly
\[
\gamma_{\ell}-p\cdot\lfloor\frac{\gamma_{\ell}}{p}\rfloor+\lfloor\frac{\gamma_{\ell/p}}{p}\rfloor
\]
cycles of length $\ell$, where $\gamma_x$ denotes the number of $x$-cycles of $\vec{\delta}$ if $x$ is a positive integer, and $\gamma_x=0$ otherwise. Indeed, all other cycle types that are $p$-th roots of $\vec{\delta}$ have at least $p$ (and thus more than $A$) cycles of some given length $\ell\in\{1,\ldots,n\}$.
\end{enumerate}

Combining these two observations, we conclude that $\Gamma(\vec{\delta})\geqslant p$ for at least $\varphi(\vec{\delta})$ many primes $p>A$, and so, denoting by $p_k$ for $k\in\IN^+$ the $k$-th prime,
\[
\Gamma(\vec{\delta})\geqslant p_{k(A)+\varphi(\vec{\delta})},
\]
where $k(A)\in\IN^+$ is such that $p_{k(A)}$ is the largest prime that is at most $A$; note that $k(A)\to\infty$ as $A\to\infty$. By the Prime Number Theorem, $p_x\sim x\log{x}$ as $x\to\infty$, and so there is an absolute constant $c'>0$ such that
\[
p_{k(A)+\varphi(\vec{\delta})}\geqslant c'(k(A)+\varphi(\vec{\delta}))\log{(k(A)+\varphi(\vec{\delta}))}.
\]
Therefore and by Lemma \ref{qGammaLem},
\begin{equation}\label{goodDeltaEq}
\q_H(V_{\vec{\delta}})\geqslant D'(\hat{m},\hat{d},\hat{p})^{-1}\cdot c'(k(A)+\varphi(\vec{\delta}))\log{(k(A)+\varphi(\vec{\delta}))}.
\end{equation}
On the other hand, by Lemma \ref{constantLem},
\begin{equation}\label{badDeltaEq}
\omicron(\bigcup\{W^{(\hat{h})}_{\vec{\epsilon}}(H)\mid\vec{\epsilon}\in(\beta^{(A)}_H)^{-1}[\{\vec{\epsilon}\}]\})\leqslant\varphi(\vec{\delta})\cdot D(\hat{m},\hat{d},\hat{p},\hat{h}).
\end{equation}
We now claim (and will show) that
\begin{equation}\label{gDefEq}
g_{\hat{m},\hat{d},\hat{p},\hat{h}}:=\frac{c'}{(D(\hat{m},\hat{d},\hat{p},\hat{h})+1)\cdot D'(\hat{m},\hat{d},\hat{p})}\cdot\log{k(A)}
\end{equation}
is a suitable choice for the function in the statement of Lemma \ref{aLem}. To verify this, assume first that the fibre $(\beta^{(A)}_H)^{-1}[\{\vec{\delta}\}]$ is empty, i.e., that $\varphi(\vec{\delta})=0$. Then by Formula (\ref{goodDeltaEq})
\[
\q_H\left(V_{\vec{\delta}}(H)\cup\bigcup\{W^{(\vec{h})}_{\vec{\epsilon}}(H)\mid\vec{\epsilon}\in(\beta^{(A)}_H)^{-1}[\{\vec{\delta}\}]\}\right)=\q_H(V_{\vec{\delta}})\geqslant \frac{c'}{D'(\hat{m},\hat{d},\hat{p})}\cdot k(A)\log{k(A)},
\]
which is indeed bounded from below by $g_{\hat{m},\hat{d},\hat{p},\hat{h}}(A)$ as in Formula (\ref{gDefEq}). Now assume that the fibre $(\beta^{(A)}_H)^{-1}[\{\vec{\delta}\}]$ is nonempty. Using Formulas (\ref{goodDeltaEq}) and (\ref{badDeltaEq}), an application of Lemma \ref{mixedPartitionLem} with
\begin{itemize}
\item $G:=H$,
\item $M:=V_{\vec{\delta}}\cup\bigcup\{W^{(\hat{h})}_{\vec{\epsilon}}(H)\mid\vec{\epsilon}\in(\beta^{(A)}_H)^{-1}[\{\vec{\epsilon}\}]\}$,
\item $M_{\good}:=V_{\vec{\delta}}$ and
\item $M_{\bad}:=\bigcup\{W^{(\hat{h})}_{\vec{\epsilon}}(H)\mid\vec{\epsilon}\in(\beta^{(A)}_H)^{-1}[\{\vec{\epsilon}\}]\}$ (which is disjoint from $M_{\good}$ because $M_{\bad}$ is by definition a union of socle cosets with cycle type in $(\beta^{(A)}_H)^{-1}[\{\vec{\epsilon}\}]\subseteq\CT^{(A)}_{\bad}(H)$, whereas $M_{\good}$ is a union of socle cosets with cycle type in $\CT^{(A)}_{\good}(H)$, and by definition, $\CT^{(A)}_{\good}(H)\cap\CT^{(A)}_{\bad}(H)=\varnothing$)
\end{itemize}
yields that
\begin{align*}
&\q_H(V_{\vec{\delta}}\cup\bigcup\{W^{(\hat{h})}_{\vec{\epsilon}}(H)\mid\vec{\epsilon}\in(\beta^{(A)}_H)^{-1}[\{\vec{\epsilon}\}]\})\geqslant \\
&\frac{1}{1+\varphi(\vec{\delta})D(\hat{m},\hat{d},\hat{p},\hat{h})}\cdot\frac{1}{D'(\hat{m},\hat{d},\hat{p})}\cdot c'(k(A)+\varphi(\vec{\delta}))\log{(k(A)+\varphi(\vec{\delta}))}\geqslant \\
&\frac{c'}{(D(\hat{m},\hat{d},\hat{p},\hat{h})+1)\cdot D'(\hat{m},\hat{d},\hat{p})}\cdot\frac{k(A)+\varphi(\vec{\delta})}{\varphi(\vec{\delta})}\cdot\log{(k(A)+\varphi(\vec{\delta}))}\geqslant \\
&\frac{c'}{(D(\hat{m},\hat{d},\hat{p},\hat{h})+1)\cdot D'(\hat{m},\hat{d},\hat{p})}\cdot\log{k(A)}=g_{\hat{m},\hat{d},\hat{p},\hat{h}}(A),
\end{align*}
as required.
\end{proof}

\subsection{More restrictions on finite semisimple groups with bounded \texorpdfstring{$\q$}{q}-value}\label{subsec4P8}

This subsection provides the last few remaining jigsaw pieces for completing the proof of Theorem \ref{mainTheo1}(2) (or, rather, of the finiteness of the classes $\Hcal^{(c)}$ from Notation \ref{classNot}(1), see Remark \ref{classRem}). Recall Lemma \ref{hcalLem}, which states that for each constant $c>0$, the class $\Hcal^{(c)}$, of all finite semisimple groups $H$ with $\q(H)\leqslant c$, is contained in the class $\Hcal_{\hat{m},\hat{d},\hat{p}}$ of finite semisimple groups (with restrictions on the simple factors in the socle, see Notation \ref{classNot}(2,3) for details), where $\hat{m}=\hat{m}(c)$, $\hat{d}=\hat{d}(c)$ and $\hat{p}=\hat{p}(c)$. So this already provides some restrictions on finite semisimple groups with bounded $\q$-value, and using Lemmas \ref{constantLem} and \ref{qGammaLem}, we will be able to add even more restrictions to this list, see Lemma \ref{hcalLem2} below.

\begin{nottation}\label{classNot2}
Let $\hat{m},\hat{d},\hat{p},\hat{r}\geqslant1$, and let $f:\left[1,\infty\right)\rightarrow\left[1,\infty\right)$. We denote by $\Hcal_{\hat{m},\hat{d},\hat{p},\hat{r},f}$ the class of finite semisimple groups $H$ such that
\begin{enumerate}
\item $H\in\Hcal_{\hat{m},\hat{d},\hat{p}}$,
\item the number of nonisomorphic nonabelian simple factors in $\Soc(H)$ is at most $\hat{r}$, and
\item the composition length of $\Soc(H)$ is at least $f(|H|)$.
\end{enumerate}
\end{nottation}

\begin{lemmma}\label{hcalLem2}
For each $c\geqslant1$ there are constants $\hat{m}=\hat{m}(c)$, $\hat{d}=\hat{d}(c)$, $\hat{p}=\hat{p}(c)$ and $\hat{r}=\hat{r}(c)$, all in $\left[1,\infty\right)$, as well as a monotonically increasing function $f_c:\left[1,\infty\right)\rightarrow\left[1,\infty\right)$ with $f_c(x)\to\infty$ as $x\to\infty$ such that $\Hcal^{(c)}\subseteq\Hcal_{\hat{m},\hat{d},\hat{p},\hat{r},f_c}$.
\end{lemmma}

\begin{proof}
Let $H\in\Hcal^{(c)}$, i.e., $H$ is a finite semisimple group with $\q(H)\leqslant c$. By Lemma \ref{hcalLem}, we can fix constants $\hat{m},\hat{d},\hat{p}\geqslant1$, all depending on $c$, such that $H\in\Hcal_{\hat{m},\hat{d},\hat{p}}$. Write $\Soc(H)=S_1^{n_1}\times\cdots\times S_r^{n_r}$, where $S_1,\ldots,S_r$ are pairwise nonisomorphic nonabelian finite simple groups and $n_1,\ldots,n_r\in\IN^+$. By Lemma \ref{qGammaLem}, for all socle cosets $C$ in $H$, we have that
\[
\q_H(C)\geqslant D'(c)^{-1}\Gamma(C)\geqslant D'(c)^{-1}r,
\]
so letting $\hat{r}(c):=cD'(c)$, we have that $\Soc(H)$ has at most $\hat{r}(c)$ nonisomorphic nonabelian simple factors.

It remains to prove the existence of $f_c$. Note that by Lemma \ref{qsnLem}(3), there is a monotonically increasing function $F:\left[1,\infty\right)\rightarrow\left[1,\infty\right)$ with $F(x)\to\infty$ as $x\to\infty$ such that
\begin{equation}\label{fSocEq}
\q(\Soc(H))\geqslant F(|\Soc(H)|).
\end{equation}

Moreover, let $\hat{h}=\hat{h}(c)$ be so large that $\q_H(C)>c$ for every $\hat{h}$-large socle coset $C$ in $H$. Finally, denote by $N(H)$ the composition length of $\Soc(H)$.

By definition, for each socle coset $C=\Soc(H)\alphabv\psiv$ of $H$, where $\psiv=(\psi_1,\ldots,\psi_r)$, we have $\cl(C)=\cl(\psiv)=(\cl(\psi_1),\ldots,\cl(\psi_r))$, and, for $i=1,\ldots,r$,
\[
\cl(\psi_i)\subseteq\{1,\ldots,n_i\}\subseteq\{1,\ldots,N(H)\},
\]
where the last inclusion uses that $N(H)=\sum_{j=1}^r{n_j}\geqslant n_i$. So $\cl(C)$ is always an $r$-tuple of subsets of $\{1,\ldots,N(H)\}$, and so the number of distinct $\cl$-values of socle cosets in $H$ is at most
\begin{equation}\label{subsetTuplesEq}
2^{N(H)r}\leqslant 2^{N(H)\hat{r}(c)}.
\end{equation}
By Formula (\ref{subsetTuplesEq}) and Lemma \ref{constantLem}, if we denote by $U$ the union of all $\hat{h}$-small socle cosets in $H$, then
\begin{equation}\label{omicronUEq}
\omicron(U)\leqslant D(c)\cdot 2^{N(H)\hat{r}(c)}.
\end{equation}
Set $M:=\Soc(H)\cup U$, $M_{\good}:=\Soc(H)$ and $M_{\bad}:=M\setminus\Soc(H)$. By Formula (\ref{omicronUEq}),
\begin{equation}\label{omicronBadEq}
\omicron(M_{\bad})\leqslant D(c)\cdot 2^{N(H)\hat{r}(c)}.
\end{equation}
In view of Formulas (\ref{fSocEq}) and (\ref{omicronBadEq}), an application of Lemma \ref{mixedPartitionLem} yields that
\begin{equation}\label{qhmEq}
\q_H(M)\geqslant\frac{F(|\Soc(H)|)}{1+D(c)\cdot 2^{N(H)\hat{r}(c)}}.
\end{equation}
Set $M':=H\setminus M$. Then $M'$ is an $\Aut(H)$-invariant union of $\hat{h}$-large socle cosets, and so by the choice of $\hat{h}$ from above and Lemma \ref{partitionLem}, we have
\[
\q_H(M')>c.
\]
But we are assuming that $\q(H)\leqslant c$, so we must have $\q_H(M)\leqslant c$ (otherwise, an application of Lemma \ref{partitionLem} with $\P:=\{M,M'\}$ yields that $\q_H(H)=\q(H)>c$). Together with Formula (\ref{qhmEq}), this yields that
\[
c\geqslant\frac{F(|\Soc(H)|)}{1+D(c)\cdot 2^{N(H)\hat{r}(c)}},
\]
or equivalently
\[
N(H)\geqslant\frac{\log{\frac{F(|\Soc(H)|)-c}{cD(c)}}}{\hat{r}(c)\log{2}}.
\]
Hence, denoting by $h(x)$ the smallest order of the socle of a finite semisimple group $H$ with $|H|\geqslant x$ (note that the function $h$ is also monotonically increasing), we find that
\[
f_c(x):=\frac{\log{\frac{F(h(x))-c}{cD(c)}}}{\hat{r}(c)\log{2}}
\]
defines a suitable choice for $f_c$ in the statement of Lemma \ref{hcalLem2}.
\end{proof}

Recall that by Remark \ref{classRem}, our goal is to show that for each $c\geqslant1$, the class $\Hcal^{(c)}$, of finite semisimple groups $H$ with $\q(H)\leqslant c$, is finite. Now Lemma \ref{hcalLem2} tells us that $\Hcal^{(c)}$ can also be written as the intersection of itself with one of the classes $\Hcal_{\hat{m},\hat{d},\hat{p},\hat{r},f}$ from Notation \ref{classNot2}. And the following lemma says that each such intersection is finite:

\begin{lemmma}\label{coupLem}
For all constants $c,\hat{m},\hat{d},\hat{p},\hat{r}\geqslant1$ and all functions $f:\left[1,\infty\right)\rightarrow\left[1,\infty\right)$ with $f(x)\to\infty$ as $x\to\infty$, the intersection $\Hcal^{(c)}\cap\Hcal_{\hat{m},\hat{d},\hat{p},\hat{r},f}$ is finite.
\end{lemmma}

\begin{proof}
We begin by declaring some parameters:
\begin{itemize}
\item Let $\hat{h}=\hat{h}(c)$ be so large that for all finite semisimple groups $H$ and all $\hat{h}$-large (see Notation \ref{hNot}) socle cosets $C$ in $H$, we have $\q_H(C)>c$ (this is possible by Lemma \ref{hHatLargeLem}).
\item Let $D'=D'(\hat{m},\hat{d},\hat{p})$ be as in Lemma \ref{qGammaLem}.
\item Set $\tilde{D}:=D(\hat{m},\hat{d},\hat{p},\hat{h}(c))$, where the quaternary function $D$ is as in Lemma \ref{constantLem}.
\item Let $A=A(\hat{m},\hat{d},\hat{p},c)$ be so large that $\min\{\frac{A}{D'(\hat{m},\hat{d},\hat{p})},g_{\hat{m},\hat{d},\hat{p},\hat{h}(c)}(A)\}>c$, where $g_{\hat{m},\hat{d},\hat{p},\hat{h}(c)}$ is as in Lemma \ref{aLem}.
\item Let $N_0=N_0(\hat{m},\hat{d},\hat{p},\hat{r},c)\geqslant A(\hat{m},\hat{d},\hat{p},c)$ be so large that for all $N\in\IN^+$ with $N>N_0$, we have
\[
\frac{N}{D'(1+A(1+\log_2{N})\tilde{D})(1+\tilde{D}(\hat{r}A(1+\log_2{N}))^A)}>c.
\]
\end{itemize}
We claim that if $H\in\Hcal^{(c)}\cap\Hcal_{\hat{m},\hat{d},\hat{p},\hat{r},f}$ has socle $\Soc(H)=S_1^{n_1}\times\cdots\times S_r^{n_r}$, where $S_1,\ldots,S_r$ are pairwise nonisomorphic nonabelian finite simple groups and $n_1,\ldots,n_r\in\IN^+$, then $n(H):=\max\{n_i\mid i=1,\ldots,r\}$ is at most $N_0$.

Indeed, assume one could have an $H\in\Hcal^{(c)}\cap\Hcal_{\hat{m},\hat{d},\hat{p},\hat{r},f}$ with $n(H)>N_0$. We will show that $\q(H)>c$, and thus a contradiction. By choice of $\hat{h}$, all $\hat{h}$-large socle cosets $C$ in $H$ satisfy $\q_H(C)>c$, and so do all socle cosets $C$ whose cycle type $\ct(C)$ is $A$-good, by choice of $A$ and Lemma \ref{qGammaLem}. Therefore, it is only the $\hat{h}$-small socle cosets with an $A$-bad cycle type that we need to worry about -- the idea is to carefully join unions of such \enquote{bad} socle cosets with unions of previously mentioned \enquote{good} socle cosets such that each \enquote{mixed} union still has $\q_H$-value strictly larger than $c$, using Lemmas \ref{partitionLem} and \ref{mixedPartitionLem} for this.

Assume that in a first step, we take care of the $\hat{h}$-small socle cosets with a cycle type that is $A$-bad of the first kind by joining each socle coset union $W^{(\hat{h})}_{\vec{\epsilon}}(H)$, where $\vec{\epsilon}\in\CT^{(A)}_{\bad,1}(H)$, with the union $V_{\beta^{(A)}_H(\vec{\epsilon})}(H)$ of socle cosets with the $A$-good cycle type $\beta^{(A)}_H(\vec{\epsilon})$. This results in a partition of
\[
H\setminus\bigcup_{\vec{\epsilon}\in\CT^{(A)}_{\bad,2}(H)}{W^{(\hat{h})}_{\vec{\epsilon}}(H)}
\]
into blocks of the form
\[
B_{\vec{\delta}}:=V_{\vec{\delta}}(H)\cup\bigcup\{W^{(\vec{h})}_{\vec{\epsilon}}(H)\mid\vec{\epsilon}\in(\beta^{(A)}_H)^{-1}[\{\vec{\delta}\}]\},
\]
where $\vec{\delta}$ ranges over the $A$-good cycle types of socle cosets of $H$. Lemma \ref{aLem} and the choice of $A$ guarantee us that
\begin{equation}\label{bEq1}
\q_H(B_{\vec{\delta}})>c
\end{equation}
for all $\vec{\delta}\in\CT^{(A)}_{\good}(H)$.

It remains to deal with the part
\[
\bigcup_{\vec{\epsilon}\in\CT^{(A)}_{\bad,2}(H)}{W^{(\hat{h})}_{\vec{\epsilon}}(H)},
\]
consisting of all $\hat{h}$-small socle cosets whose cycle type is $A$-bad of the second kind. Denote by $\vec{\delta_0}:=\ct(\Soc(H))$ the trivial cycle type, which is $A$-good because
\[
\Gamma(\vec{\delta_0})=\sum_{i=1}^r{n_i}\geqslant n(H)>N_0\geqslant A.
\]
One of the blocks of the partition of $H\setminus\bigcup_{\vec{\epsilon}\in\CT^{(A)}_{\bad,2}(H)}{W^{(\hat{h})}_{\vec{\epsilon}}(H)}$ mentioned just above is
\[
B_{\vec{\delta_0}}:=V_{\vec{\delta_0}}(H)\cup\bigcup\{W^{(\hat{h})}_{\vec{\epsilon}}(H)\mid\epsilon\in(\beta^{(A)}_H)^{-1}[\{\vec{\delta_0}\}]\},
\]
and we claim that the union
\[
B'_{\vec{\delta_0}}:=B_{\vec{\delta_0}}\cup\bigcup_{\vec{\epsilon}\in\CT^{(A)}_{\bad,2}(H)}{W^{(\hat{h})}_{\vec{\epsilon}}(H)}
\]
still has $\q_H$-value strictly larger than $c$.

Following the proof of Lemma \ref{aLem} with $\vec{\delta}:=\vec{\delta_0}$ and using again that $\Gamma(\vec{\delta_0})=\sum_{i=1}^r{n_i}\geqslant \max\{n_1,\ldots,n_r\}=n(H)$, we find that
\[
\q_H(B_{\vec{\delta_0}})\geqslant\frac{n(H)}{D'(1+\varphi^{(A)}_H(\vec{\delta_0})\tilde{D})}.
\]
Now $(\beta^{(A)}_H)^{-1}[\{\vec{\delta_0}\}]$ consists only of cycle types $\vec{\epsilon}$ such that $A<\ord(\vec{\epsilon})=p$ is a prime and $p$ divides one of the numbers $n(H),n(H)-1,\ldots,n(H)-A+1$ (otherwise, each of the corresponding cycle types $\vec{\epsilon}$ has too many cycles to be $A$-bad). Moreover, each such prime $p$ corresponds to at most one $A$-bad cycle type $\vec{\epsilon}$ of the first kind. It follows that
\[
\varphi^{(A)}_H(\vec{\delta_0})=|(\beta^{(A)}_H)^{-1}[\{\vec{\delta_0}\}]|\leqslant A(1+\log_2{n(H)}),
\]
and so
\begin{equation}\label{goodEq2}
\q_H(B_{\vec{\delta_0}})\geqslant\frac{n(H)}{D'(1+A(1+\log_2{n(H)})\tilde{D})}.
\end{equation}
How many element orders are there in
\[
\bigcup_{\vec{\epsilon}\in\CT^{(A)}_{\bad,2}(H)}{W^{(\hat{h})}_{\vec{\epsilon}}(H)},
\]
the union of all the $\hat{h}$-small socle cosets with an $A$-bad cycle type of the second kind? The number of such cycle types is at most $(\hat{r}A(1+\log_2{n(H)}))^A$ (think of length $A$ sequences of pairs of choices of an index $i\in\{1,\ldots,r\}\subseteq\{1,\ldots,\lfloor\hat{r}\rfloor\}$ and of a cycle length in $\{1,\ldots,n_i\}\subseteq\{1,\ldots,n(H)\}$ which is a power of a prime $p\leqslant A$). Hence, by Lemma \ref{constantLem},
\begin{equation}\label{badEq2}
\omicron(\bigcup_{\vec{\epsilon}\in\CT^{(A)}_{\bad,2}(H)}{W^{(\hat{h})}_{\vec{\epsilon}}(H)})\leqslant\tilde{D}(\hat{r}A(1+\log_2{n(H)}))^A.
\end{equation}
Applying Lemma \ref{mixedPartitionLem} with
\begin{itemize}
\item $M:=B'_{\vec{\delta_0}}=B_{\vec{\delta_0}}\cup\bigcup_{\vec{\epsilon}\in\CT^{(A)}_{\bad,2}(H)}{W^{(\hat{h})}_{\vec{\epsilon}}(H)}$,
\item $M_{\good}:=B_{\vec{\delta_0}}$ and
\item $M_{\bad}:=\bigcup_{\vec{\epsilon}\in\CT^{(A)}_{\bad,2}(H)}{W^{(\hat{h})}_{\vec{\epsilon}}(H)}$,
\end{itemize}
and using Formulas (\ref{goodEq2}) and (\ref{badEq2}), we conclude that
\begin{equation}\label{bEq2}
\q_H\left(B'_{\vec{\delta_0}}\right)\geqslant\frac{n(H)}{D'(1+A(1+\log_2{n(H)})\tilde{D})(1+\tilde{D}(\hat{r}A(1+\log_2{n(H)}))^A)}>c,
\end{equation}
where the second inequality is by the assumption $n(H)>N_0$ and the choice of $N_0$.

We are now ready to show that $\q(H)>c$. Consider the partition $\P$ of $H$ whose members are the following:
\begin{itemize}
\item the set $B'_{\vec{\delta_0}}=B_{\vec{\delta_0}}\cup\bigcup_{\vec{\epsilon}\in\CT^{(A)}_{\bad,2}(H)}{W^{(\hat{h})}_{\vec{\epsilon}}(H)}$, and
\item the sets $B_{\vec{\delta}}$ where $\vec{\delta}\in\CT^{(A)}_{\good}(H)\setminus\{\delta_0\}$.
\end{itemize}
By Formulas (\ref{bEq1}) and (\ref{bEq2}), each partition member has $\q_H$-value strictly larger than $c$, and so $\q_H(H)=\q(H)>c$ by an application of Lemma \ref{partitionLem}. This is the desired contradiction confirming that $n(H)\leqslant N_0$.

Now that we know that $n(H)\leqslant N_0$, and in view of our assumption that $H\in\Hcal_{\hat{m},\hat{d},\hat{p},\hat{r},f}$, it follows that the composition length $\sum_{i=1}^r{n_i}$ of $\Soc(H)$ is at most $\hat{r}N_0$, and so $\hat{r}N_0\geqslant f(|H|)$. But $f(x)\to\infty$ as $x\to\infty$, so there are indeed only finitely many possibilities for $H$, as required.
\end{proof}

\subsection{Completing the proof of Theorem \ref{mainTheo1}(2)}\label{subsec4P9}

Let us now give a proof of Theorem \ref{mainTheo1}(2) using the results developed in the previous subsections. By Remark \ref{classRem}, it suffices to show that for each constant $c\geqslant1$, the class $\Hcal^{(c)}$, defined in Notation \ref{classNot}(1), is finite. By Lemma \ref{hcalLem2}, we find that there are
\begin{itemize}
\item constants $\hat{m}=\hat{m}(c)$, $\hat{d}=\hat{d}(c)$, $\hat{p}=\hat{p}(c)$, $\hat{r}=\hat{r}(c)$, all in $\left[1,\infty\right)$, as well as
\item a monotonically increasing function $f_c:\left[1,\infty\right)\rightarrow\left[1,\infty\right)$ with $f_c(x)\to\infty$ as $x\to\infty$
\end{itemize}
such that $\Hcal^{(c)}$ is contained in the class $\Hcal_{\hat{m},\hat{d},\hat{p},\hat{r},f_c}$, as defined in Notation \ref{classNot2}. In other words, we have
\begin{equation}\label{ultimateEq}
\Hcal^{(c)} \cap \Hcal_{\hat{m},\hat{d},\hat{p},\hat{r},f_c} = \Hcal^{(c)}.
\end{equation}
But an application of Lemma \ref{coupLem} yields that the intersection $\Hcal^{(c)}\cap\Hcal_{\hat{m},\hat{d},\hat{p},\hat{r},f_c}$, i.e., the class $\Hcal^{(c)}$ by Formula (\ref{ultimateEq}), is finite, which concludes the proof.

\section*{Acknowledgements}

The authors would like to thank Stefan Kohl for checking and providing some feedback on their GAP source code. Moreover, they would like to thank the anonymous referee for their careful reading of this long manuscript and their helpful comments.

\end{document}